\documentclass[11pt,fleqn]{article}
\usepackage[margin=1in]{geometry}               
\usepackage{graphicx}
\usepackage{amssymb}
\usepackage{amsmath}
\usepackage{xfrac}
\usepackage{epstopdf}
\usepackage{amsthm}
\usepackage{threeparttable}
\usepackage{booktabs}
\usepackage{setspace}
\usepackage[ruled,vlined]{algorithm2e}
\usepackage[usenames,dvipsnames]{xcolor}
\usepackage{MnSymbol,wasysym}
\usepackage{authblk}
\usepackage{caption}
\usepackage{natbib}

\usepackage{xcolor}
\usepackage[verbose]{hyperref}
\usepackage{tikz}
\usepackage{pgfplots}
\usetikzlibrary{arrows,snakes,backgrounds,automata,positioning}
\usetikzlibrary{shapes,decorations}
\usepackage{tikz-cd}
\usepackage[skins,theorems,most]{tcolorbox}
\usepackage{float}
\usepackage{bm}
\tcbset{highlight math style={enhanced,colframe=red,colback=white,arc=0pt,boxrule=1pt}}
\hypersetup{colorlinks=false,allbordercolors=blue,pdfborderstyle={/S/U/W 1}}

\pgfplotsset{compat=1.18} 

\newcommand{\dss}{\displaystyle}

\newtheorem{proposition}{Proposition}

\newtheorem*{remark}{Remark}

\begin{document}
\title{Effects of Heterogeneity in Two-Cell Feedforward Networks}
\author[1]{Abdullah Ahmed\thanks{labdullah@umd.edu}}
\author[1]{Maria Cameron\thanks{mariakc@umd.edu}}
\author[2]{Antonio Palacios\thanks{apalacios@sdsu.edu}}
\author[1]{Hengyuan Qi\thanks{hqi@umd.edu}}
\author[2]{Samir Sahoo\thanks{samirkumsahoo@gmail.com}}
\affil[1]{Department of Mathematics, University of Maryland, College Park, MD 20742, USA}
\affil[2]{Nonlinear Dynamical Systems Group, Department of Mathematics, San Diego State University, San Diego, CA 92182-7720, USA}

\maketitle

\abstract{
As the need for higher performance from biological and electronic sensors continues to outpace current technologies, new strategies for designing, developing, and implementing novel sensor systems are emerging. 
A recently introduced feedforward network-based approach can simultaneously enhance a signal while steering a radiating beam in radio frequency communication systems. Furthermore, the approach is also model-independent, thus making it suitable for other applications. In this work, we aim to understand the effects of inhomogeneities in feedforward arrays, which are inevitable in real-world implementations. We investigate a collection of two-cell feedforward networks composed of pitch-fork cells and Stuart-Landau oscillators and quantify the effects of parameter inhomogeneities using system reduction, analytical and computational bifurcation analyses, and a singularity theory approach. Contrary to common intuition, inhomogeneity in the excitation parameter can be exploited to enhance the network output growth rate. While frequency inhomogeneity in Stuart-Landau networks primarily has an adverse effect on signal amplification, phase locking persists over a surprisingly broad range of inhomogeneity.}


\section{Introduction}
\subsection{An overview}
Signal amplification is a quintessential paradigm in science and engineering, in which an input signal is enhanced to facilitate detection, classification, and subsequent processing. This process allows for the development of highly sensitive sensors. To achieve this highly sought-after effect, various strategies, which depend primarily on the type of signal to be amplified, have been developed across many disciplines.
In bioelectrochemistry, fuel cells and electrolyzers, biosensors' response to detect trace numbers of analytes can be enhanced, while reducing noise signals, through nanomaterials, enzyme catalysis, and biological reactions~\cite{LIU20185}.
In mass cytometry, thermal-cycling-based DNA enables signal amplification, resulting in higher quantitative measurements of proteins or protein modifications at single-cell resolution~\cite{Cytometry}.
In electronic systems, traditional methods for signal amplification are achieved through transistors. In the lower-frequency stages, for example, many amplifying transistors are integrated on a single substrate, enabling high gain with very high stability and requiring few external components~\cite{BOLE201429}.
A novel approach to electric- and magnetic-field sensors exploits the symmetry of networks and infinite-period bifurcations that generate heteroclinic cycles to achieve sensitivities on the order of pico-teslas and femto-amps, respectively~\cite{IBPLKN,LON,IKBNB,esensor2,PVIL,Patent2}.
In phototransduction, signal amplification, which involves activation of a relatively small number of G protein-coupled receptors, is achieved through a cascade~\cite {AVB}.
\medskip

Beam steering is another fundamental problem, mainly in engineering, in which the goal is to control the direction of a radiating far-field intensity pattern~\cite{TedHeath,YOIT,POMAYO}. Conventional and modern methods involve manipulating the phase shift between oscillating components, which consists of arrays of nonlinear oscillators. This is usually accomplished by leveraging the inherent nonlinearities of individual components and the collective dynamics of the oscillator array. In optics, beam steering can be done by either changing the refractive index of the medium through which the beam is transmitted, or by the use of mirrors, prisms, or rotating diffraction gratings~\cite{CHENG2021106700}. 
In acoustics, beam steering is about changing the direction of audio from speakers, and it can be accomplished by changing the magnitude and phase of speakers arranged in
an array~\cite{WMH}. Other applications can be found in aerospace communication~\cite{Kacker:25}, in light detection and ranging (LiDARs)~\cite{2006SPIE}, in laser scanning microscopy~\cite{Lechleiter:2002aa}, in imaging of organs~\cite{OTFHI},
and in antenna and radar systems~\cite{HWY,YOIT,YORK}. 
\medskip

Over the past few years, we have been crafting novel strategies to simultaneously address both problems, signal amplification and beam steering. The fundamental principle is the use of feedforward networks as the underlying strategy. Feedforward networks are a specific type of network characterized by a {\em homogeneous} chain of unidirectionally coupled nodes~\cite{TYLER,Levasseur_Palacios_2021}. The first node or cell may be self-coupled. In both cases, the unidirectional coupling prevents feedback in the system, so that each cell may influence another without being itself affected. This coupling configuration departs from traditional methods, which typically employ bidirectionally coupled nonlinear oscillators. It has been shown that under the right conditions, the feedforward network causes certain bifurcations to exhibit accelerated growth rates~\cite{GENH,CURIOUS,GP2012}. These bifurcations are the source of achieving signal amplification. In follow-up work, we considered the feedforward network as a replacement for the commonly used bidirectionally coupled arrays of nonlinear oscillators. 
Thus, we first studied the {\em transmission problem}: generating and steering the 
radiating patterns that emanate from the sources~\cite{Levasseur2022beam1}. A fundamental result of that work was to show the existence of stable phase locking~\cite{PRK} in the collective dynamics. This result is deemed fundamental because the phase-locking angle can be used to steer the beam, while the feedforward dynamics amplify the signal. In subsequent work, we addressed the {\em reception problem}: understanding the interaction of the feedforward array of nonlinear oscillators with external signals. Studying the reception mode in a feedforward network is more complicated because incident signals introduce time-dependent forcing terms; thus, the model is non-autonomous, whereas in the transmission case it is autonomous. The main contribution from this latter work was to show the regions of parameter space where stable, 1:1 synchronization with the external signals exists~\cite {Levasseur2025beam2}. It was also shown that the interaction of the collective pattern of oscillation produced by the feedforward network can lead to narrower main lobes and lower sidelobes than in other commonly used array configurations. It is desirable to have a beamformer system with a narrow mainlobe and low sidelobes, as these characteristics can improve resolution and reduce susceptibility to interference from strong signals.
\medskip

\subsection{A summary of main results}
The next phase of work is the actual development and implementation of the technology. But to do that, we must first account for the fact that, in practice, even the most advanced systems are imperfect. That is, they do not necessarily conform to the underlying assumption of homogeneity. In other words, real-life implementation requires us to consider the effects of inhomogeneities or disorder in a system. In the case of beam steering and signal amplification, inhomogeneities may arise naturally through fluctuations in the parameters of each cell, e.g., excitation and frequency parameters, or through variations in the coupling strengths of the cells. 

In this work, we abandon the assumption that all nodes are identical and examine the effects of fluctuations on system parameters in a feedforward network. We focus on two-cell arrays of pitchfork and Stuart-Landau cells. 
Adding inhomogeneities into the model equations introduces unfolding parameters that automatically increase the codimension of the bifurcations.  Naturally, higher codimension yields more complex dynamics, which we investigate in this work. We use model reduction and singularity theory to conduct a bifurcation analysis, plot phase and bifurcation diagrams, and obtain a complete quantification of the effects of parameter inhomogeneities.

We show that, contrary to intuition, inhomogeneity in the excitation parameter $\mu$ in two-cell pitchfork and Stuart-Landau feedforward networks can result in response amplification beyond that achievable with a homogeneous network.

In addition, feedforward arrays of two Stuart-Landau cells admit inhomogeneities in frequency $\omega$ and in the cubic nonlinearity parameter $\gamma$.  We demonstrate that signal amplification is neutral with respect to the inhomogeneity in $\gamma$. By contrast, frequency inhomogeneity has an adverse effect on signal amplification. However,  phase locking persists for a surprisingly broad range of frequency gaps. Beyond this range, the system settles on an invariant torus attractor, i.e., exhibits quasiperiodic oscillations. 

The results of this work can serve as guiding principles for engineers designing and operating emerging technologies for sensors, beam steering, and signal amplification systems. They also pave the way for the study of large feed-forward arrays with inhomogeneities, planned for future work. 
\medskip

The manuscript is organized as follows.
In Section~\ref{sec:Background}, a review of the effect of signal amplification and beam steering through feedforward networks is introduced. 
Section~\ref{sec:system1} is devoted to feed-forward arrays of pitch-fork cells. The effects of inhomogeneities in the excitation parameter are studied, with an emphasis on the network's equilibrium states. 
Section~\ref{sec:Hopf} presents a systematic study of the effects of inhomogeneity in frequency and the excitation and cubic nonlinearity parameters in two-cell feed-forward arrays of Stuart-Landau oscillators.
The results are discussed in Section \ref{sec:discussion}. 
\medskip

\section{Background} \label{sec:Background}
In this section, two fundamental applications, which serve as a motivation for the analysis of feedforward networks, {\em signal amplification} and {\em beam steering}, are reviewed for completeness purposes. 

\subsection{Signal amplification via feedforward networks}

Feedforward networks are a specific type of network characterized by a homogeneous chain of unidirectionally coupled nodes~\cite{LEPA,TYLER}.
The first node may be self-coupled. Under the right conditions, the 
feedforward network causes certain bifurcations to exhibit accelerated growth rates. 
\medskip

Consider, for example, the three-cell feedforward network shown in Fig.~\ref{fig:bckgd_feed_Hopf}
The internal dynamics of each cell are governed by a Hopf bifurcation:
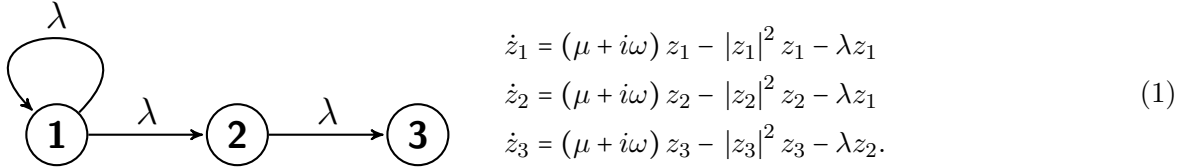
\begin{figure}[H]
\begin{minipage}{.4\textwidth}
  \centering
\begin{tikzpicture}[->,>=stealth',shorten >=1pt,auto,node distance=2.4cm,
  thick,main node/.style={circle,draw,font=\sffamily\Large\bfseries}]

  \node[main node] (1) {1};
  \node[main node] (2) [right of =1] {2};
  \node[main node] (3) [right of =2] {3};

  \path[every node/.style={font=\sffamily\Large}]
   (1) edge [loop] node [above] {$\lambda$} (1)
   (1) edge node {$\lambda$} (2)
   (2) edge node {$\lambda$} (3);
   
\end{tikzpicture}
\end{minipage}%
\begin{minipage}{.6\textwidth}
 \centering
\begin{equation} \label{eq:ffwd_v1}
\begin{aligned}
\dot z_1 &= \left(\mu +i\omega\right)z_1-\,\left|z_1\right|^2z_1 - \lambda z_1 \\
\dot z_2 &= \left(\mu +i\omega\right)z_2-\,\left|z_2\right|^2z_2 - \lambda z_1 \\
\dot z_3 &= \left(\mu +i\omega\right)z_3-\,\left|z_3\right|^2z_3 - \lambda z_2 .
\end{aligned}
\end{equation}
\end{minipage}
\caption{Representative example of a three-cell feedforward network. Arrows indicate coupling, with coupling strength $\lambda$. Each cell represents a dynamical system assumed to be 
operating near a Hopf bifurcation.}
\label{fig:bckgd_feed_Hopf}
\end{figure}
The authors in~\cite{GENH,CURIOUS,GP2012} have found that the coupling causes the amplitudes of oscillations that arise from the onset of the Hopf bifurcation to grow at a larger rate. If $\mu$ is the bifurcation parameter, and $\mu = 0$ is the onset of a supercritical Hopf bifurcation, then the third cell undergoes oscillations of amplitude approximately equal to $\mu^{1/6}$, rather than the \emph{expected} amplitude of $\mu^{1/2}$. This phenomenon showcases an accelerated growth rate that has the potential for the design and fabrication of advanced filters in signal processing~\cite{MOTIF,MMG2007}). An example of a time series, obtained from simulations of 
Eq.~\eqref{eq:ffwd_v1}, which exhibits this growth phenomenon, is shown in Fig.~\ref{fig:ffwd_amp}(a). Without self-coupling on the first cell, see Fig.~\ref{fig:ffwd_amp}(b), the amplification effect is even larger.
\begin{figure}[htbp]
\centering
(a): \includegraphics[width=0.4\textwidth]{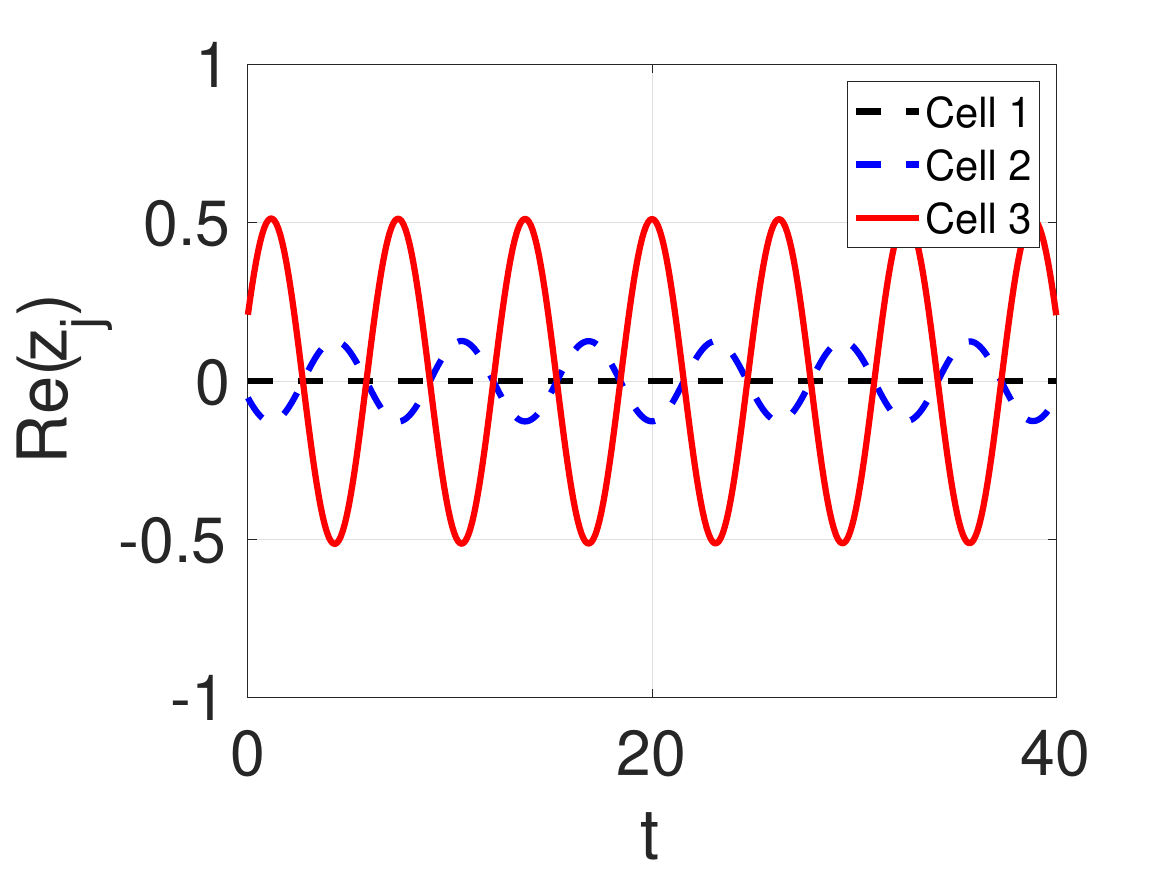}
(b): \includegraphics[width=0.4\textwidth]{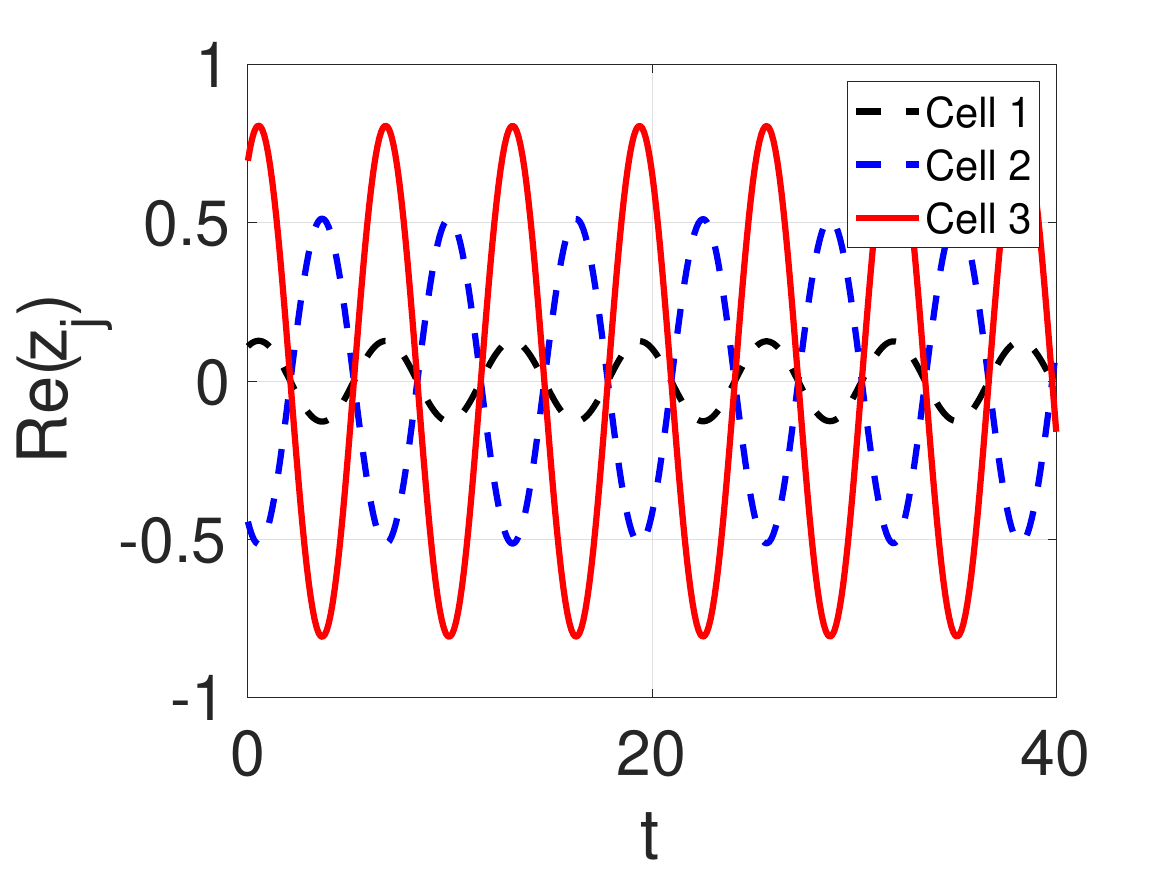}
\caption{Signal amplification in a feedforward network (a) with self-coupling of the first cell and (b) without self-coupling. Parameters are: $\mu = (0.5)^6$, $\omega = 1$,  $\lambda = 1$. (a): With self-coupling. (b): Without self-coupling.
}
\label{fig:ffwd_amp}
\end{figure}

As the feedforward network grows in size, the growth rate of oscillations in the final cell is determined by taking successive cube roots~\cite{MOTIF}. Thus, in a five-cell 
In a feedforward network, the growth rate should be proportional to the $54^{\text{th}}$ root of the bifurcation parameter, as shown in~\cite{RINK}. Achieving such large growth rates is exciting indeed, as it can lead to novel mechanisms for signal amplification
beyond the simple square-root growth rate. This, of course, means that large amplitude oscillations may arise very soon after the onset of the Hopf bifurcation.
\medskip

The phenomenon of such large-amplitude oscillations in the third cell can be understood as a type of nonlinear resonance as well as being the result of the combination of the unidirectional coupling and the higher-degree nonlinearities~\cite{LEPA}. Related articles proving that anomalous growth rates can occur for equilibria in (unusual) regular networks and for bifurcations at simple eigenvalues can be found in~\cite{SG1,SG2}.

\subsection{Beam Steering}
Beam steering involves strategies to manipulate and control the direction of a radiating far-field intensity pattern~\cite{TedHeath,YOIT,POMAYO}. 
In antennas and radar systems, beam steering can be achieved either by switching the antenna elements or by controlling the phase differences between oscillating components, which typically consist of arrays of nonlinear oscillators. Overall, modern methods for beam steering exploit the inherent nonlinearities of individual components and the collective dynamics of oscillator arrays to manipulate phase shifts. None of
those methods includes, however, signal amplification.
\medskip

Recently, the idea of employing a feedforward network, which enables simultaneous beam steering and signal amplification, was introduced in~\cite{Levasseur2022beam1, Levasseur2025beam2}. These works showed that the branches of oscillations in the feedforward network can, under certain conditions, exhibit stable phase-locking patterns, in which the phase difference between consecutive oscillators is constant. Those phase differences are critical to steering a beam towards a desired location. The beam steering can be accomplished as follows.
\medskip

The total radiation pattern of an antenna array at a point $P$ is given by the equation
\begin{equation} \label{eq:radiation_field}
  E(P) = A ( \Psi ) \, E_0 e^{i k r_0},
\end{equation} 
where $E_0 e^{i k r_0}$ is the electric field produced by a single patch element, $\Psi = k d \sin{\varphi}$, in which $k$ is a free-space wave vector, $d$ is the distance between consecutive array elements, and $\varphi$ is the angle of emission of the beam. The term $A ( \Psi )$ is an {\em Array Factor} defined by
\begin{equation} \label{eq:Array_Factor}
   A ( \Psi ) =  
   \dss \frac{\sin{ \left( \dss \frac{N \Psi}{2} \right) }}{N \sin{ \left( \dss \frac{\Psi}{2} \right) }} e^{i (N-1) \Psi / 2}.
\end{equation}
\medskip

The absolute value of the Array Factor, $|A(\Psi)|$, is symmetric with respect to $\Psi = 0$, and it always attains a maximum at $\Psi = 0$, which corresponds to an angle of emission of $\varphi = 0$. This angle is also known as the {\em broadside direction}, as it is normal to the plane of the array. A critical observation is that varying $\Psi$ changes the direction of the radiating beam. Thus, the problem of beam steering translates into finding strategies to vary $\Psi$. One possibility is to employ a feedforward network and tune it to a phase-locking regime. If $\theta$ is the constant phase-locking angle among consecutive oscillators, then it has been shown~\cite{Levasseur2022beam1, Levasseur2025beam2} that this angle enters into the array factor as
\begin{equation} \label{eq:Array_Factor_v2}
   A ( \Psi ) =  
   \dss \frac{\sin{ \left( \dss \frac{N (\Psi + \theta)}{2} \right) }}{N \sin{ \left( \dss \frac{\Psi + \theta}{2} \right) }} e^{i (N-1) (\Psi + \theta)/ 2}.
\end{equation}
\medskip

By varying the phase-locking angle, we can control the beam direction, thereby steering the beam pattern from angle $\Psi$ to $\Psi + \theta$.  Figure~\ref{fig:beam_steering} illustrates the strategy of beam steering through varying the phase-locking angle, $\theta$, in a feedforward array with 20 nonlinear oscillators.
\begin{figure}[htbp]
\centerline{
\includegraphics[width=0.9\textwidth]{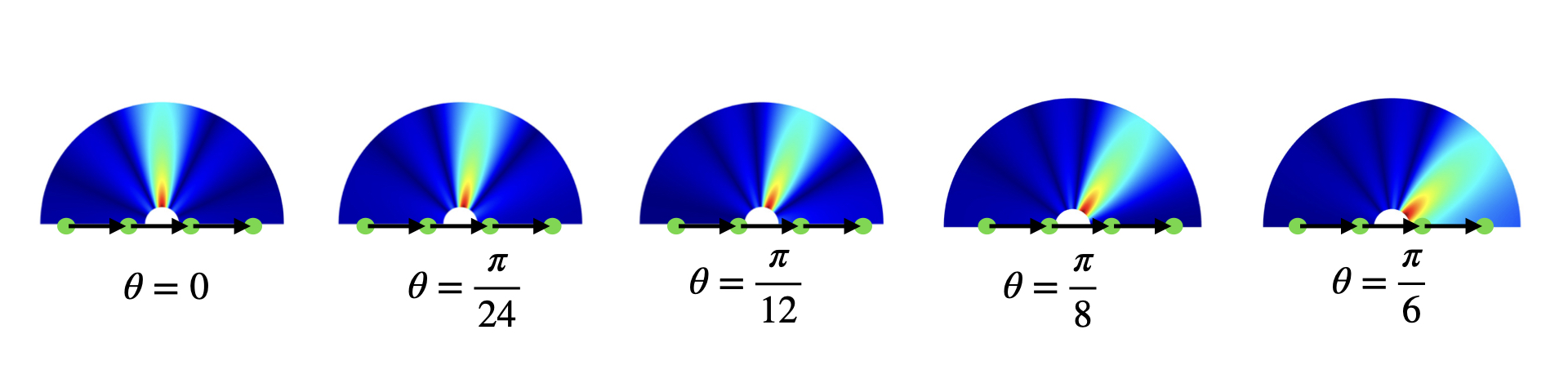}
}
\caption{Beam steering can be achieved in a feedforward network by varying the phase-locking angle $\theta$. The sequence of snapshots shows the radiating pattern with 20 nonlinear oscillators.}
\label{fig:beam_steering}
\vspace*{12pt}
\end{figure}
\medskip

Together, the larger growth rate of oscillation amplitudes and phase-locking characteristics enable innovative solutions and improvements (e.g., simultaneous signal amplification and beam steering) over conventional methods.
\medskip

\section{A two-cell feedforward network of cells with pitchfork bifurcations}
\label{sec:system1}

In this section, we study the effects of inhomogeneity in the excitation parameter of a feedforward network of cells, each of which admits a supercritical pitchfork bifurcation. We examine a two-cell system and conduct a thorough study of its equilibria via an asymptotic analysis of cubic roots, a numerical investigation, and a singularity analysis.

\subsection{A two-cell system: the setup}
We consider a simple two-cell system exhibiting steady-state bifurcations:
\begin{figure}[H]
\begin{minipage}{.4\textwidth}
  \centering
\begin{tikzpicture}[->,>=stealth',shorten >=1pt,auto,node distance=2.4cm,
  thick,node1/.style={rectangle,fill=blue!10,minimum width = 0.75cm,minimum height = 0.75cm,draw,font=\sffamily\Large\bfseries},node2/.style={rectangle,fill=red!10,minimum width = 0.75cm,minimum height = 0.75cm,draw,font=\sffamily\Large\bfseries}]

  \node[node1] (1) {1};
  \node[node2] (2) [right of =1] {2};

  \path[every node/.style={font=\sffamily\Large}]
   (1) edge node {$\lambda$} (2);
   
\end{tikzpicture}
\end{minipage}%
\begin{minipage}{.6\textwidth}
 \centering
\begin{equation} \label{eq:system1} 
\begin{aligned}
\dot{x} & = \mu x - x^3,\\
\dot{y} & = (\mu + \varepsilon) y - y^3  - \lambda x.
\end{aligned}
\end{equation}
\end{minipage}
\caption{A schematic of a two-cell feed-forward network with inhomogeneous cells. Each cell represents a system prone to a pitchfork bifurcation.}
\label{fig:bckgd_feed_Equil}
\end{figure}

The parameter $\mu$ represents the input signal: $\mu = 0$ corresponds to the absence of an input signal, while $\mu > 0$ means that there is a signal that we want to detect. The parameter $\varepsilon$ represents the inhomogeneity of the cells. It may be positive or negative, due to manufacturing imperfections or intentional design. 
The parameter $\lambda$ is the coupling strength. 
\medskip

Varying the parameter $\mu$ unfolds an ensemble of qualitatively different bifurcation diagrams depending on the relationship between $\mu$ and $\varepsilon$, demonstrating an unexpectedly complex set of stable and unstable equilibria -- see Fig.~\ref{fig:sys1_phase_bifur}. 

\subsection{Structure and stability of equilibria}
\label{sec:system1_equilibria}
We start analyzing system~\eqref{eq:system1} by finding its equilibria and establishing their stability at various combinations of $\mu$ and $\varepsilon$. The first cell, whose state variable is $x$, is a closed system. At $\mu < 0$, its only equilibrium  is $x=0$, while at $\mu > 0$ it has three equilibria: $0$ and $\pm\sqrt{\mu}$.   Throughout this and the next sections, we assume that $\lambda >0$, because $\lambda < 0$ is equivalent to switching the sign of $x$.

The second cell, with state variable $y$, is influenced by the first cell. If $x =0$, its equilibria are the roots of the cubic polynomial
\begin{equation}
    \label{eq:p0}
    p_0(y): = (\mu+\varepsilon)y - y^3.
\end{equation}
If $\mu+\varepsilon < 0$, $p_0(y)$ has a unique root at $y = 0$. If $\mu + \varepsilon > 0$, it has three roots, $0$ and $\pm\sqrt{\mu+\varepsilon}$. 

If $\mu > 0$ and the first cell is at a nonzero equilibrium, i.e., $x = \pm\sqrt{\mu}$, the equilibria of the second cell are the roots of the cubic polynomials
\begin{equation}
    \label{eq:p1cubic_main}
    p_{+}(y): = (\mu + \varepsilon) y - y^3  - \lambda \sqrt{\mu}\quad {\rm and}\quad p_{-}(y): = (\mu + \varepsilon) y - y^3  + \lambda \sqrt{\mu},
\end{equation}
respectively. The critical relationship between $\mu$, $\varepsilon$, and $\lambda$, at which the polynomials $p_{\pm}(y)$ have exactly two roots is (see Eq. \eqref{eq:sys1:critical}):
\begin{equation}
    \label{eq:poly2roots}
    2(\mu+\varepsilon)^{3/2} = 3
    \sqrt{3}\lambda\mu^{1/2}.
\end{equation}
Dividing Eq. \eqref{eq:poly2roots} by $\lambda^{3/2}$ and introducing $\tilde{\mu}:=\mu/\lambda$ and $\tilde{\varepsilon}$, we obtain an equation where $\lambda$ is eliminated: $3(\tilde{\mu}+\tilde{\varepsilon})^{3/2} = 3\sqrt{3}\tilde{\mu}^{1/2}$. This means that the parameter $\lambda$ in Eq. \eqref{eq:poly2roots} merely scales $\mu$ and $\varepsilon$. In other words, we can measure $\mu$ and $\varepsilon$ in $\lambda$-units. The curve defined by Eq. \eqref{eq:poly2roots} separates the regions where $p_{\pm}(y)$ have one and three solutions -- see Fig. \ref{fig:sys1_phase_bifur}.
\begin{figure}[htbp]
\centerline{
\includegraphics[width=0.75\textwidth]{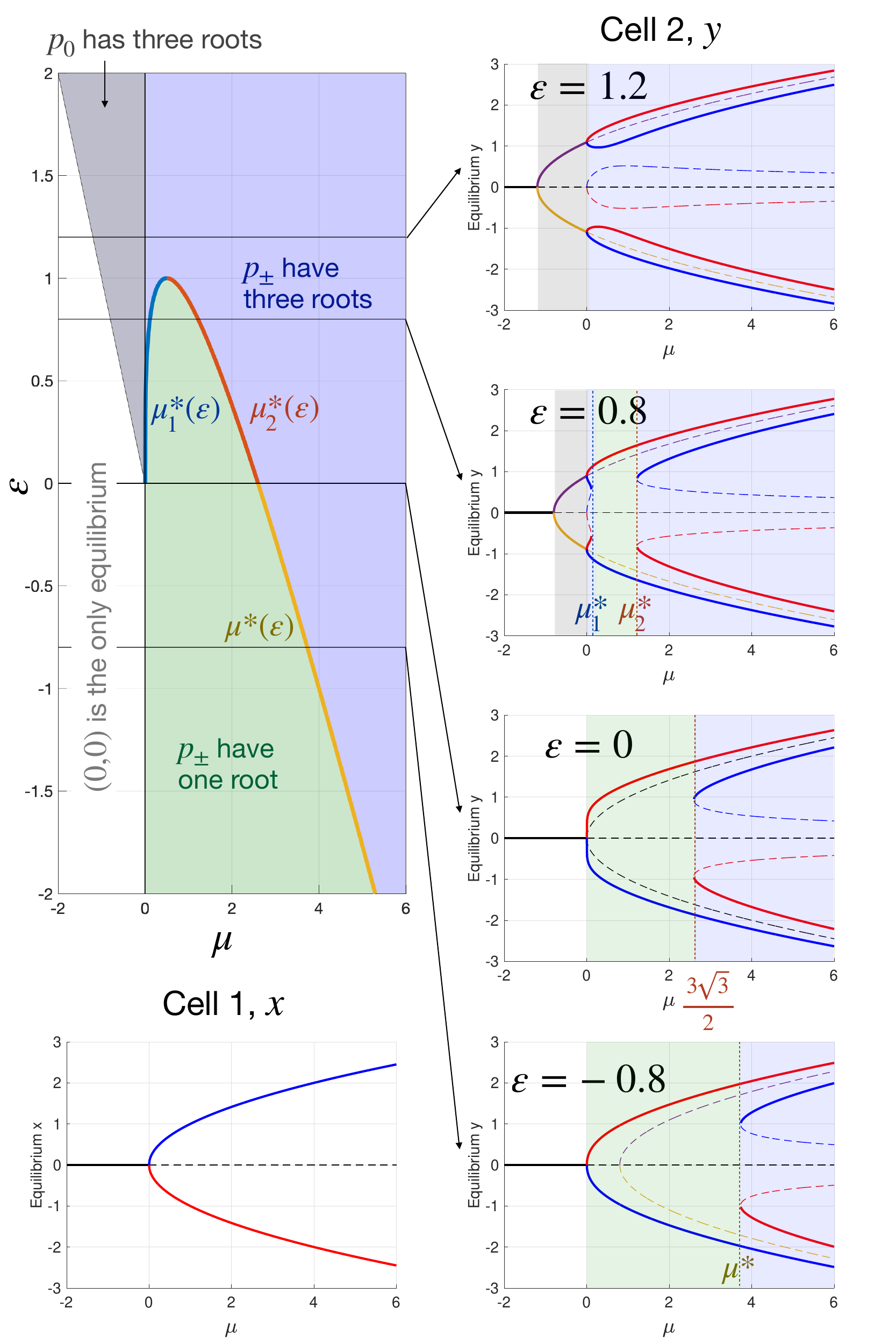}
}
\caption{A phase diagram and representative bifurcation diagrams for system~\eqref{eq:system1}. Solid and dashed lines represent stable and unstable branches, respectively. The branches of $y$ existing at $x = \sqrt{\mu}$ and $x = -\sqrt{\mu}$ are blue and red, respectively. Black, yellow, and purple branches of $y$ correspond to $x = 0$. The curve separating the one- and three-root regions, plotted in blue, red, and yellow, is defined by Eq. \eqref{eq:poly2roots}.}
\label{fig:sys1_phase_bifur}
\vspace*{12pt}
\end{figure}

The stability of the equilibria of Eq.~\eqref{eq:system1} is determined by its Jacobian
\begin{align}
\label{eq:system1_J}
    J=\begin{bmatrix}
        \mu-3x^2 &0\\ -\lambda & \mu+\varepsilon-3y^2
    \end{bmatrix}
\end{align}
evaluated at these equilibria.
The Jacobian is lower-triangular due to the feedforward structure of the network; therefore, its eigenvalues are its diagonal entries, $\mu-3x^2$ and $\mu + \varepsilon - 3y^2$.
We consider four cases with a distinct structure of the equilibria. The summarizing phase diagram in the $(\mu,\varepsilon)$-plane and representative bifurcation diagrams are displayed in Fig. \ref{fig:sys1_phase_bifur}. A detailed description of the structure and stability of the roots of $p_\pm(y)$ is given in Appendix~\ref{AppA}.

\subsection{What is the system's response to a jump in the excitation parameter?}
\label{sec:system1_jump}
In Section \ref{sec:system1_equilibria}, we found the equilibria of system~\eqref{eq:system1} and determined their stability. The next question, motivated by the sensor design, concerns how the system will respond to a sudden jump in the excitation parameter $\mu$ from zero to a small positive value. 

To answer this question, we study the basin structure of system~\eqref{eq:system1}. At positive $\mu$, the system may have two or four attractors (sinks) of stable node type, depending on the values of $\mu$ and $\varepsilon$ -- see Fig. \ref{fig:sys1_basins}(a). As found in Section \ref{sec:system1_equilibria}, four sinks exist if $0< \varepsilon < \lambda$ and $0<\mu<\mu_1^{*}(\varepsilon)$, where $\mu_1^*(\varepsilon)$ is the smallest positive root of the equation $ 2 (\mu+\varepsilon)^{3/2} = 3\sqrt{3}\lambda\mu^{1/2}$, which is well-approximated by $\mu_1^*(\varepsilon)\approx\tfrac{4\varepsilon^3}{27\lambda^2}$ at $\varepsilon< 0.5\lambda$ (see Eq.~\eqref{eq:sys1:critical}) and Fig.~\ref{fig:sys1_cubic}(b)). Otherwise, system~\eqref{eq:system1} has two sinks.

\begin{figure}[htbp]
\centering
(a)\includegraphics[width=0.8\textwidth]{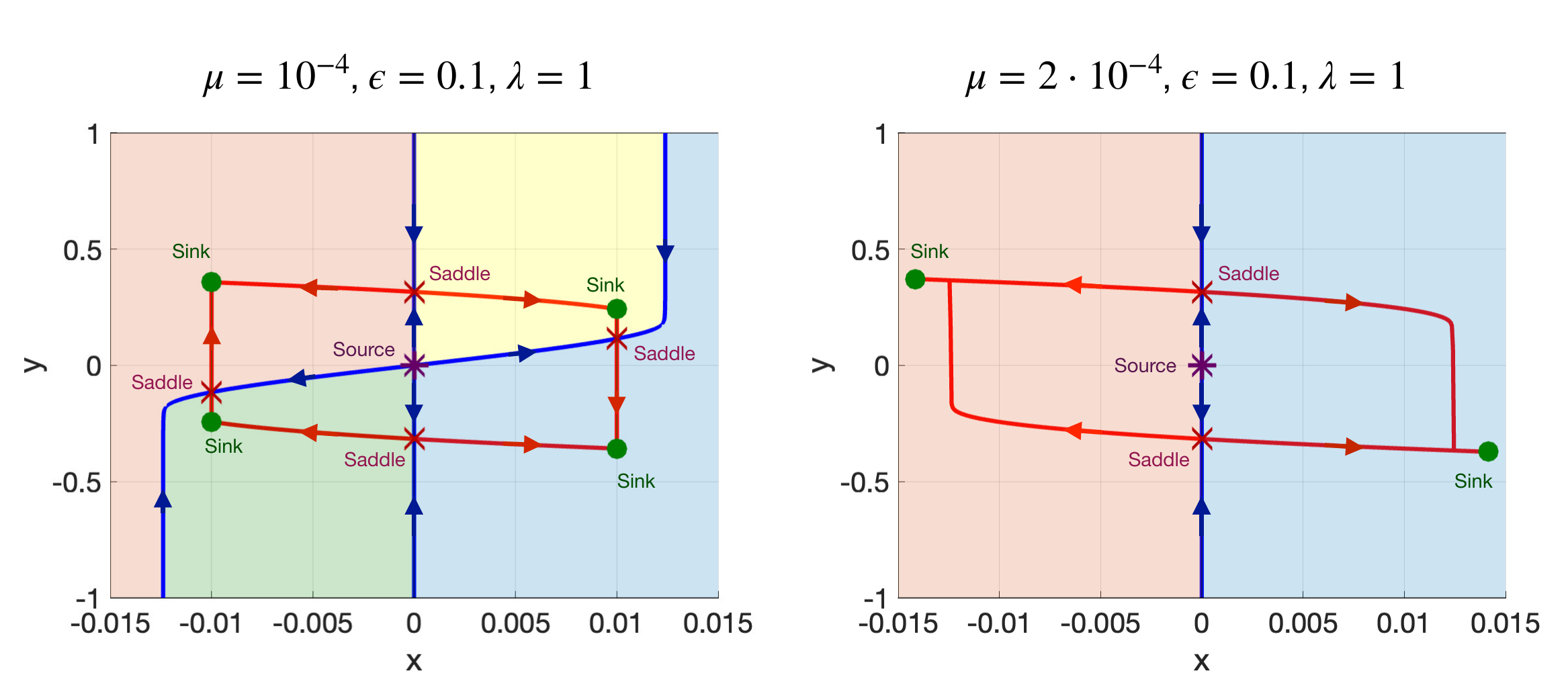}\\
(b)\includegraphics[width=0.8\textwidth]{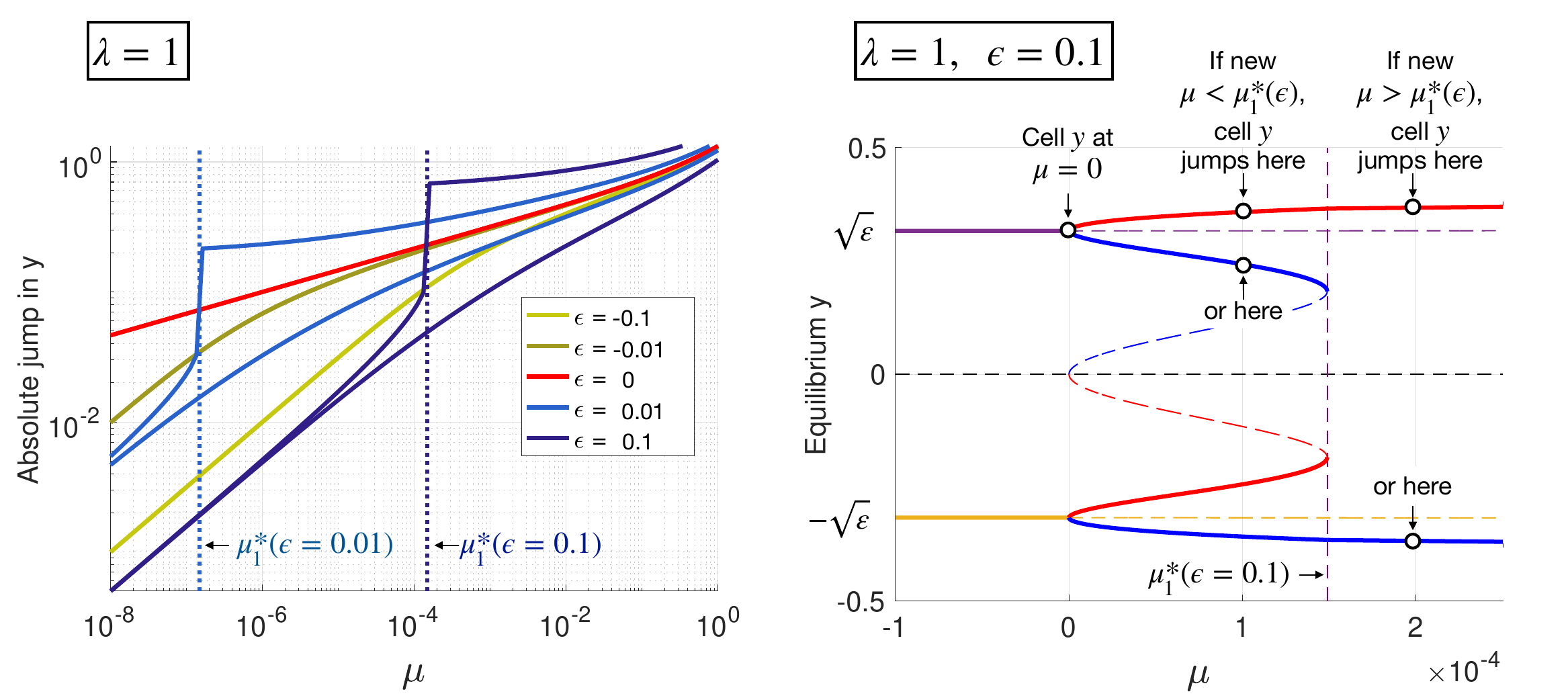}
\caption{(a): Two possible topologically different basin and heteroclinic structures of system~\eqref{eq:system1}. For $0<\varepsilon<\lambda$, as $\mu$ increases from 0, the sinks with yellow and green basins collide with the nearest saddles and disappear. Then these pairs of sink and saddle reappear as $\mu$ keeps increasing. (b, left): The absolute value of the jump in $y$ versus $\mu$ in log-log scale occurring due to a sudden change in $\mu$ from zero to a positive value. If $\varepsilon > 0$, two absolute values of the jump in $y$ are possible at each $\mu$-jump, depending on whether the initial $y$ is $\sqrt{\varepsilon}$ or $-\sqrt{\varepsilon}$ and whether $x$ jumps to $\sqrt{\mu}$ or $-\sqrt{\mu}$. The dotted vertical line corresponds to the critical values of $\mu$ at which sinks and saddles collide as described in (a), i.e., a saddle-node bifurcation takes place. (b, right): A zoom-in of the bifurcation diagram at $\varepsilon = 0.1$. Stable equilibrium branches for $y$ are shown with solid lines, while unstable ones with dashed lines. The vertical dashed line is at the critical value of $\mu$ at which the saddle-node bifurcation takes place.
}
\label{fig:sys1_basins}
\end{figure}

If $\mu = 0$, the system resides at $(x=0,y=0)$ if $\varepsilon\le 0$, and at $(x=0,y=\sqrt{\varepsilon})$ or $(x=0,y=-\sqrt{\varepsilon})$, if $\varepsilon > 0$. These initial conditions lie on the basin boundary at any basin structure for $\mu > 0$, making it uncertain where the system switches due to a sudden jump in $\mu$ from zero to a positive value. As $\mu$ suddenly jumps to a positive value, $x$ switches to $x=\sqrt{\mu}$ or $x = -\sqrt{\mu}$. To understand where $y$  jumps, we consider four cases.
\begin{itemize}
    \item {\bf Case $\boldsymbol{\varepsilon}\boldsymbol{ <}\boldsymbol{ 0}$.} In this case, the initial $y$ is zero, and system~\eqref{eq:system1} has two symmetric sinks. The qualitative behavior of solution branches for $y$ is shown in the bifurcation diagram in the bottom right in Fig. \ref{fig:sys1_phase_bifur}. Thus, $y$ will jump from zero to one of two sinks, and the magnitude of the jump will be the same whether $x$ jumps to $\sqrt{\mu}$ or $-\sqrt{\mu}$. The magnitudes of the jump in $y$ versus $\mu$  at $\varepsilon = -0.1$ and $\varepsilon = -0.01$ are plotted, respectively, in olive and yellow in Fig.~\ref{fig:sys1_basins}(b). Comparing these plots with the red plot of the jump in $y$ at $\varepsilon = 0$, we conclude that negative $\varepsilon$ may only weaken the jump in $y$ compared to the one at $\varepsilon = 0$. Hence, it is not beneficial for signal amplification to reduce the excitation coefficient in the second cell. However, if $\varepsilon$ happens to be slightly negative due to a manufacturing imprecision, the amplification of the response in $y$ will be almost as strong as it is at $\varepsilon = 0$. 

    \item {\bf Case $\boldsymbol{0}\boldsymbol{<}\boldsymbol{\varepsilon}\boldsymbol{ <}\boldsymbol{\lambda}$, $\boldsymbol0\boldsymbol{<}\boldsymbol{\mu}\boldsymbol{<}\boldsymbol{\mu}^{\boldsymbol{*}}\boldsymbol{(}\boldsymbol{\varepsilon}\boldsymbol{)}$.} Suppose that $y = \sqrt{\varepsilon}$ before the jump in $\mu$. If the new value of $\mu$ satisfies $0<\mu<\mu_1^{*}(\varepsilon)$, the system acquires four sinks as shown in Fig.~\ref{fig:sys1_basins}(a, left). Depending on whether $x$ switches to $\sqrt{\mu}$ or $-\sqrt{\mu}$,  $y$ will jump, respectively, to the positive blue branch or the positive red branch as shown in Fig. \ref{fig:sys1_basins}(b, right). In this case, the response in $y$ is weaker than at $\varepsilon = 0$ -- see the light and dark blue plots in Fig. \ref{fig:sys1_basins}(b, left) at $\mu <\mu_1^*(\varepsilon)$. If the initial $y = -\sqrt{\varepsilon}$, the situation is similar.

  \item {\bf Case $\boldsymbol{0}\boldsymbol{<}\boldsymbol{\varepsilon}\boldsymbol{ <}\boldsymbol{\lambda}$, $\boldsymbol{\mu}\boldsymbol{>}\boldsymbol{\mu}^{\boldsymbol{*}}\boldsymbol{(}\boldsymbol{\varepsilon}\boldsymbol{)}$.} Suppose that $y = \sqrt{\varepsilon}$ before the jump in $\mu$. Since the system has only two sinks at $\mu > \mu_1^{*}(\varepsilon)$, it will jump at one of them depending on the sign of $x$. Contrary to the case where $\varepsilon < 0$, the magnitude of these two possible jumps will be different, as evident from Fig.~\ref{fig:sys1_basins}(b). It is impossible to predict which branch this jump will be to, the upper red or the lower blue in Fig.~\ref{fig:sys1_basins}(b, right), because its starting point is very close to the basin boundary. If the jump is to the lower blue branch, its magnitude will be approximately $2\sqrt{\varepsilon}$, a macroscopic quantity even at small $\mu$. However, the jump to the upper red branch will change the value of $y$ less than if $\varepsilon = 0$. If the initial $y = -\sqrt{\varepsilon}$, the situation is similar.

    \item {\bf Case $\boldsymbol{\varepsilon}\boldsymbol{ >}\boldsymbol{\lambda}$.} In this case, the system has two sinks at any positive $\mu$. The bifurcation diagram in Fig.~\ref{fig:sys1_phase_bifur} at the top right shows that we cannot expect a larger response in $y$ to a jump in $\mu$ than in the case $\varepsilon = 0$. 
\end{itemize}

In summary, the most interesting case from the viewpoint of sensor design is $0<\varepsilon<\lambda$, $\mu > \mu_1^*(\varepsilon)$, as a macroscopic jump in $y$ is possible. But how can one guarantee it will happen? We envision an upgrade of the system shown in Fig.~\ref{fig:bckgd_feed_Equil_upgrade}.
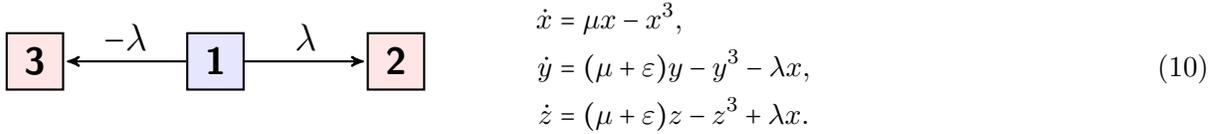
\begin{figure}[H]
\begin{minipage}{.4\textwidth}
  \centering
\begin{tikzpicture}[->,>=stealth',shorten >=1pt,auto,node distance=2.4cm,
  thick,node1/.style={rectangle,fill=blue!10,minimum width = 0.75cm,minimum height = 0.75cm,draw,font=\sffamily\Large\bfseries},node2/.style={rectangle,fill=red!10,minimum width = 0.75cm,minimum height = 0.75cm,draw,font=\sffamily\Large\bfseries}]

   \node[node2] (3) [left of =1] {3};
 \node[node1] (1) {1};
  \node[node2] (2) [right of =1] {2};

  \path[every node/.style={font=\sffamily\Large}]
   (1) edge node {$\lambda$} (2);
  \path[every node/.style={font=\sffamily\Large}]
   (1) edge node[above]{$-\lambda$} (3);
   
\end{tikzpicture}
\end{minipage}%
\begin{minipage}{.6\textwidth}
 \centering
\begin{equation} \label{eq:system1modified} 
\begin{aligned}
\dot{x} & = \mu x - x^3,\\
\dot{y} & = (\mu + \varepsilon) y - y^3  - \lambda x,\\
\dot{z} & = (\mu + \varepsilon) z - z^3  + \lambda x.
\end{aligned}
\end{equation}
\end{minipage}
\caption{A schematic of a three-cell network of inhomogeneous pitchfork cells. The proposed coupling can yield a dramatic amplification effect under certain settings. }
\label{fig:bckgd_feed_Equil_upgrade}
\end{figure}
In the modified system, cell one is coupled to two identical cells, two and three, with coupling coefficients of equal magnitudes but opposite signs. In this case, if $\mu$ jumps from zero to a positive value, the response in one of these cells will be significant. To ensure that $\mu > \mu_1^*(\varepsilon)$, we can prescribe the sensitivity threshold $\mu_0$, and choose $\varepsilon$  such that $\mu_1^*(\varepsilon) < \mu_0$, i.e., 
\begin{equation}
    \label{eq:epsilon_choice}
    \varepsilon < \frac{3\mu_0^{1/3}\lambda^{2/3}}{4^{1/3}}.
\end{equation}


\subsection{Singularity theory approach}
\label{sec:system1_singularity_equilibria}

In this section, we describe changes that occur in the number of equilibrium points and their stability properties from the standpoint of the {\em singularity theory}~\cite{GS1,GSS}. 
\medskip

Let $\mu > 0$. We start by assuming that the first cell in Eq.~\eqref{eq:system1} has already converged to a nonzero equilibrium, $x=\pm \sqrt{\mu}$. 
Since $x$ always converges to an equilibrium point, the dependence of $y$ on $x$ vanishes as $x$ becomes indistinguishable from a constant. Theorem 17.0.3 in~\cite{WIG} makes this precise. 
We consider the case $x = \sqrt{\mu}$. The case $x = -\sqrt{\mu}$ is treated similarly. The governing ODE for the second cell then becomes 
\begin{equation} \label{eq:y_system}
  \dot y = - y^3 + (\mu + \varepsilon) y - \lambda \sqrt{\mu}\equiv p_+(y).
\end{equation}
The perturbation $\varepsilon$ lets the free term and the coefficient of $y$ in the right-hand side of Eq. \eqref{eq:y_system} change independently at any fixed $\lambda$.
We denote $\lambda\sqrt{\mu}$ by $\hat{\lambda}$ and $\mu + \varepsilon$ by $\hat{\mu}$ and rewrite Eq. \eqref{eq:y_system} as 
\begin{equation} \label{eq:y_system_hat}
  \dot y = - y^3 + \hat{\mu} y  - \hat{\lambda} = :G(y,\hat{\lambda},\hat{\mu})\equiv g(y,\hat{\lambda}) + \hat{\mu} y,
\end{equation}
where $\hat{\lambda}\in\mathbb{R}$. We treat the function $ g(y,\hat{\lambda}) = -y^3 -\hat{\lambda}$ as the main function, or a \emph{germ}, and the term $\hat{\mu} y$ as a perturbation to it.  Looking for the equilibria of the unperturbed system $\dot{y} = g(y,\hat{\lambda})$, we solve $g(y,\hat{\lambda}) = 0$ for $y$ and plot
 $y$ as a function of $\hat{\lambda}$. 
 This germ has a unique root $y = -\hat{\lambda}^{1/3}$. The graph of $y$ versus $\hat{\lambda}$ has a vertical tangent and an inflection point at $\hat{\lambda} = 0$. At this point, $g = g_y = g_{yy} = 0$, which means that this is a \emph{hysteresis point}.  The bifurcation diagram at a hysteresis point is not persistent: a small perturbation of the germ changes its root structure.

Singularity theory \cite{GS1} is a tool for identifying qualitative changes in bifurcation diagrams of germs, i.e., in their root structure,  arising from small perturbations. The germ $ g(y,\hat{\lambda}) = -y^3 -\hat{\lambda}$ matches the first example in (\cite{GS1}, Section III.4(b)), where
it is proven that the function $G(x,\hat{\lambda},\hat{\mu}) = g(y,\hat{\lambda}) + \hat{\mu}y$ is a \emph{universal unfolding} of $g$. This means that
\begin{enumerate}
\item $G(x,\hat{\lambda},0) = g(x,\hat{\lambda})$,
\item $G(x,\hat{\lambda},\boldsymbol{\alpha})$, $\boldsymbol{\alpha}\in\mathbb{R}^k$, captures all possible qualitative behaviors of the solution branches of $g(x,\hat{\lambda}) +\alpha p(x,\hat{\lambda},\alpha) = 0$, as $\hat{\lambda}$ runs through the real line, resulting from arbitrary small perturbations one-parameter perturbations $\alpha p(x,\hat{\lambda},\alpha)$,  and
\item $k$ is the minimal number of parameters, in addition to $\hat{\lambda}$, to do so.
\end{enumerate}
In our case, this minimal number of parameters is one.
Thus, we find all possible bifurcations of $\dot{y} = G(y,\hat{\lambda},\hat{\mu})$. The necessary condition for $(y_0,\hat{\lambda}_0,\hat{\mu}_0)$ to be a \emph{singularity}, or a bifurcation point, i.e., a point where the number of the solutions $y$ to $G(y,\hat{\lambda},\hat{\mu}) = 0$ may change, is 
\begin{equation}
    \label{eq:singularity_cond}
    G(y_0,\hat{\lambda}_0,\hat{\mu}_0) = \frac{\partial}{\partial y}G(y_0,\hat{\lambda}_0,\hat{\mu}_0) = 0.
\end{equation}
Eq. \eqref{eq:singularity_cond} yields the condition $\hat{\lambda}_0^2 = \tfrac{4}{27}\hat{\mu}_0^3$.
Recalling that $\hat{\mu} = \mu + \varepsilon$ and $\hat{\lambda} = \lambda\sqrt{\mu}$, we get the locus of the bifurcation point:
\begin{equation} \label{eq:locus_bif}
   \lambda^2 = {4 \over 27 \mu}(\mu + \varepsilon)^3,
\end{equation}
which, unsurprisingly, follows from Eq. \eqref{eq:poly2roots} obtained from the condition for the cubic polynomials in Eq. \eqref{eq:y_system} to have exactly two roots.
\medskip

Fig.~\ref{fig:Two_Par_Diag} shows a two-parameter, $(\varepsilon, \lambda)$, diagram of the locus of the bifurcation captured by Eq.~\eqref{eq:locus_bif} for a representative value of $\mu=0.2$. For this particular value, the hysteresis point is located at $\varepsilon = -\mu =-0.2$. 
To the left of the locus of the bifurcation, there is only one equilibrium point, and it is stable. To the right of it, there are three equilibria, two of which are stable. The two additional equilibria are born at a saddle-node bifurcation, which is captured by the locus. The insets with the red markers showcase the transitions in the number and stability of the zeros of the polynomial $p_+(y)$ for a path in which $\varepsilon$ varies while $\lambda$ is held fixed.

Furthermore, since Eq.~\eqref{eq:locus_bif} is invariant under the change $\lambda \mapsto - \lambda$, then the transitions with $\lambda < 0$ are the same as those with positive $\lambda$. It follows that the locus of the saddle-node bifurcation applies identically to both $p_-(y)$ and $p_+(y)$.
\begin{figure}[htbp]
\centerline{\includegraphics[width=0.9\textwidth]{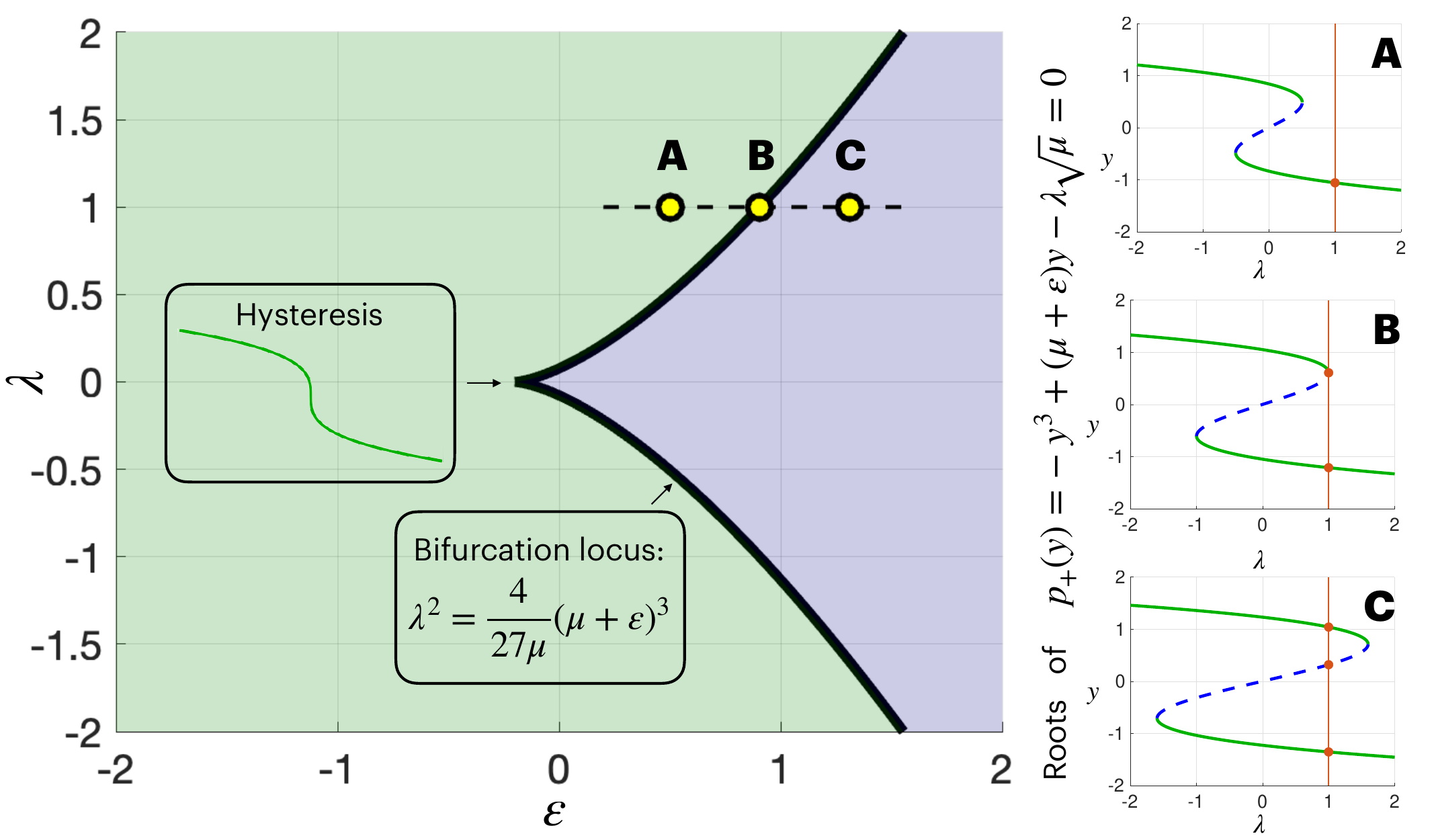}}
\caption{\small Two-parameter bifurcation diagram in the $(\varepsilon, \lambda)$-plane. Solid (green) curves in the insets indicate stable equilibrium points. Dashed (blue) curves depict unstable equilibrium points. The parameter $\mu$ was set to $\mu = 0.2$, as a representative example. For a fixed value of $\lambda$, for instance $\lambda = 1$, while varying $\varepsilon$ from left to right, the number of equilibrium points (red filled-in circles in the insets), i.e., roots of $p_+(y)$, changes from one (to the left of the locus of a saddle-node bifurcation) to two (on the locus), and then to three (to the right of the locus).}
\label{fig:Two_Par_Diag}
\vspace*{12pt}
\end{figure}


\section{A two-cell feedforward network of Stuart-Landau oscillators}
\label{sec:Hopf}

In this section, we study the effects of inhomogeneities in the excitation parameter $\mu$, the frequency $\omega$, and the cubic nonlinearity parameter $\gamma$, in systems of two Stuart-Landau oscillators with a feedforward coupling. 
We first systematically study the case with an inhomogeneity in the natural frequency alone, as this paves the way for analyzing the system with inhomogeneities in all parameters. 
Then we add inhomogeneity in the excitation and nonlinearity parameters.

\subsection{Inhomogeneity in the natural frequency}
\label{sec:system2}
Fig.~\ref{fig:FFwd_Disorder_Freq} shows a schematic diagram of a two-cell system with inhomogeneity in the frequency parameter.
\begin{figure}[H]
\begin{minipage}{.4\textwidth}
  \centering
\begin{tikzpicture}[->,>=stealth',shorten >=1pt,auto,node distance=2.4cm,
  thick,node1/.style={circle,fill=violet!30,draw,font=\sffamily\Large\bfseries},node2/.style={circle,fill=yellow!30,draw,font=\sffamily\Large\bfseries}]

  \node[node1] (1) {1};
  \node[node2] (2) [right of =1] {2};

  \path[every node/.style={font=\sffamily\Large}]
   (1) edge node {$\lambda$} (2);
   
\end{tikzpicture}
\end{minipage}%
\begin{minipage}{.6\textwidth}
 \centering
\begin{equation} \label{eq:system2}
\begin{aligned}
\dot{z_1} & = (\mu +i\omega)z_1 - |z_1|^2z_1,\\
\dot{z_2} & = (\mu  + i(\omega +\sigma) )z_2 - |z_2|^2z_2  - \lambda z_1.   
\end{aligned}
\end{equation}
\end{minipage}
\caption{A schematic diagram of a two-cell feed-forward network with Stuart-Landau cells with an inhomogeneity in the natural frequency represented by $\sigma$.}
\label{fig:FFwd_Disorder_Freq}
\end{figure}
System~\eqref{eq:system2} was analyzed in~\cite{Levasseur_Palacios_2021} using the two-timing method (see e.g.~\cite{Strogatz}, Section 7.6), aiming at finding periodic solutions and focusing on plots of the squared amplitude of the second cell, $|z_2|^2$, versus $\sigma$ at selected representative values of $\mu$. Here, we adopt a different approach by switching to a co-rotating frame and reducing the system, thereby eliminating the parameter $\lambda$ and redefining the parameters $\mu$ and $\sigma$. This allows us to present a complete phase diagram in the $(\sigma,\mu)$-plane in Fig.~\ref{fig:system2_phase} and show that system~\eqref{eq:system2} admits two types of attractors: stable limit cycles and invariant tori.  

\medskip
Let $\mu > 0$. We assume that the first oscillator is at its periodic attractor, i.e., 
\begin{equation}
\label{eq:z1_sys2}
    z_1(t) = \sqrt{\mu}e^{\omega t i}.
\end{equation} 
Then we seek a solution for the second oscillator of the form 
\begin{equation}
\label{eq:z2_sys2}
    z_2(t) = u(t)e^{\omega t i}\equiv \left(u_R(t) + iu_I(t)\right)e^{\omega t i},
\end{equation} 
 where $u(t)$ is a complex-valued function that needs to be found, and $u_R$ and $u_I$ are its real and complex parts, respectively. Plugging Eqs.~\eqref{eq:z1_sys2} and \eqref{eq:z2_sys2} into Eq.~\eqref{eq:system2} and canceling the term $e^{\omega t i}$, we obtain the following equation for $u$:
 \begin{equation}
     \label{eq:u_sys2}
     \dot{u} = (\mu + i\sigma)u - |u|^2u - \lambda\sqrt{\mu}.
 \end{equation}

\subsubsection{Reducing the number of parameters}
\label{sec:system2_reducing}
We aim to investigate the structure and stability of equilibria of Eq. \eqref{eq:u_sys2} at all $\mu > 0$, $\sigma\in\mathbb{R}$, and $\lambda > 0$. Making $\lambda < 0$ has the same effect as multiplying $u$ by the factor of $e^{i\pi/2}$.  The equilibria of Eq. \eqref{eq:u_sys2} correspond to periodic attractors of system \eqref{eq:system2}. We aim to reduce the number of parameters from three to two. First, we normalize $u$ by introducing  $v: = u/\sqrt{\mu} $. Then $|v| > 1$ means that the amplitude of $z_2$ is larger than that of $z_1$, and the other way around if $|v| < 1$.
Next, we rescale the time as $\tau = \lambda t$. Then $\tfrac{d}{dt}= \lambda\tfrac{d}{d\tau}$. Finally, we divide the resulting ODE for $v$ by $\lambda$ and introduce $\tilde{\mu}: = \mu/\lambda$ and $\tilde{\sigma} = \sigma/\lambda$. 
In summary, the reduction setup is
\begin{equation}
    \label{eq:system2:reduction_setup}
    u = v\sqrt{\mu},\quad \tau = \lambda t,\quad \tilde{\mu} = \frac{\mu}{\lambda},\quad \tilde{\sigma} = \frac{\sigma}{\lambda}.
\end{equation}
The resulting ODE for $v$ has only two parameters, $\tilde{\mu}$ and $\tilde{\sigma}$:
\begin{equation}
    \label{eq:v_sys2}
    \dot{v} = \tilde{\mu}(v - |v|^2v) + i\tilde{\sigma} v  -1.
\end{equation}
The corresponding system for the real and complex parts of $v$, $v = v_R + iv_I$, is:
 \begin{equation}
 \label{eq:vrvi}
  \begin{array}{ll}
\dot{v}_R & = \tilde{\mu}(1 - |v|^2)v_R - \tilde{\sigma} v_I - 1,\\[5pt]
\dot{v}_I & = \tilde{\mu}(1 - |v|^2)v_I + \tilde{\sigma} v_R.  
\end{array}
 \end{equation}

\subsubsection{Periodic solutions and their stability}
\label{sec:system2_periodic}
We are primarily interested in phase-locked attractors of ODE \eqref{eq:system2}, which correspond to equilibria of Eq. \eqref{eq:v_sys2}. Thus, we seek the roots of the system of algebraic equations
 \begin{equation}
 \label{eq:vrvi_equil}
  \begin{array}{ll}
 \tilde{\mu}(1 - |v|^2)v_R - \tilde{\sigma} v_I & = 1,\\[5pt]
 \tilde{\mu}(1 - |v|^2)v_I + \tilde{\sigma} v_R & = 0.  \end{array}
 \end{equation}
  Squaring the left-and right-hand sides of each equation in Eq. \eqref{eq:vrvi_equil}, adding them together, we find that the level sets of $|v|^2$ are semi-ellipses in the $(\tilde{\sigma},\tilde{\mu})$-plane:
  \begin{equation}
      \label{eq:ellipse0}
      |v|^2\tilde{\sigma}^2 +|v|^2(1-|v|^2)^2 \tilde{\mu}^2 = 1,\quad \tilde{\mu}> 0.
  \end{equation}
  These semi-ellipses, shown in Fig. \ref{fig:system2_phase}, are centered at the origin and have the semiaxes $r_{\tilde{\sigma}} = |v|^{-1}$ and $r_{\tilde{\mu}} = |v|^{-1}|1-|v|^2|^{-1}$. If $|v|^2 > 2$, $r_{\tilde{\sigma}}$ is the major semiaxis, while if $|v|^2 < 2$, this is the other way around. If $|v|^2 = 1$, the ellipse degenerates into two lines, $\tilde{\sigma} = \pm 1$. 
        If $|v|^2 \in (0,1)$, the semi-ellipses fill out the exterior of the region bounded by $\tilde{\sigma} = 0$, the dark green dash-dotted curve in Fig. \ref{fig:system2_phase} defined by Eq. \eqref{eq:dash-dot}, and the lines $|\tilde{\sigma}| = \pm 1$. As $|v|^2$ increases from $0$ to $\tfrac{1}{3}$, the semi-ellipses, plotted with gray dashed lines in Fig. \ref{fig:system2_phase}, shrink to the semi-ellipse passing through $(0,\tfrac{3\sqrt{3}}{2})$. As $|v|^2$ further increases from $\tfrac{1}{3}$ to $1$, the semi-ellipses, plotted with solid or dashed lines of grey shades, if $|v|^2 \in (\tfrac{1}{3},\tfrac{1}{2})$, or red,  brown and dark yellow shades, if $|v|^2 \in (\tfrac{1}{2},1)$, become thinner and taller, finally degenerating to the lines $\tilde{\sigma} = \pm 1$ as $|v|^2\rightarrow 1$. If $|v|^2 > 1$, the semi-ellipses, plotted with blue and green solid lines,  foliate the strip $|\tilde{\sigma}| < 1$ and shrink to a point around the origin as $|v|^2\rightarrow\infty$. Throughout the rest of this paper, we will omit ``semi-" for brevity and refer to the semi-ellipses as ``ellipses".

        \begin{figure}[htbp]
\centerline{\includegraphics[width=0.9\textwidth]{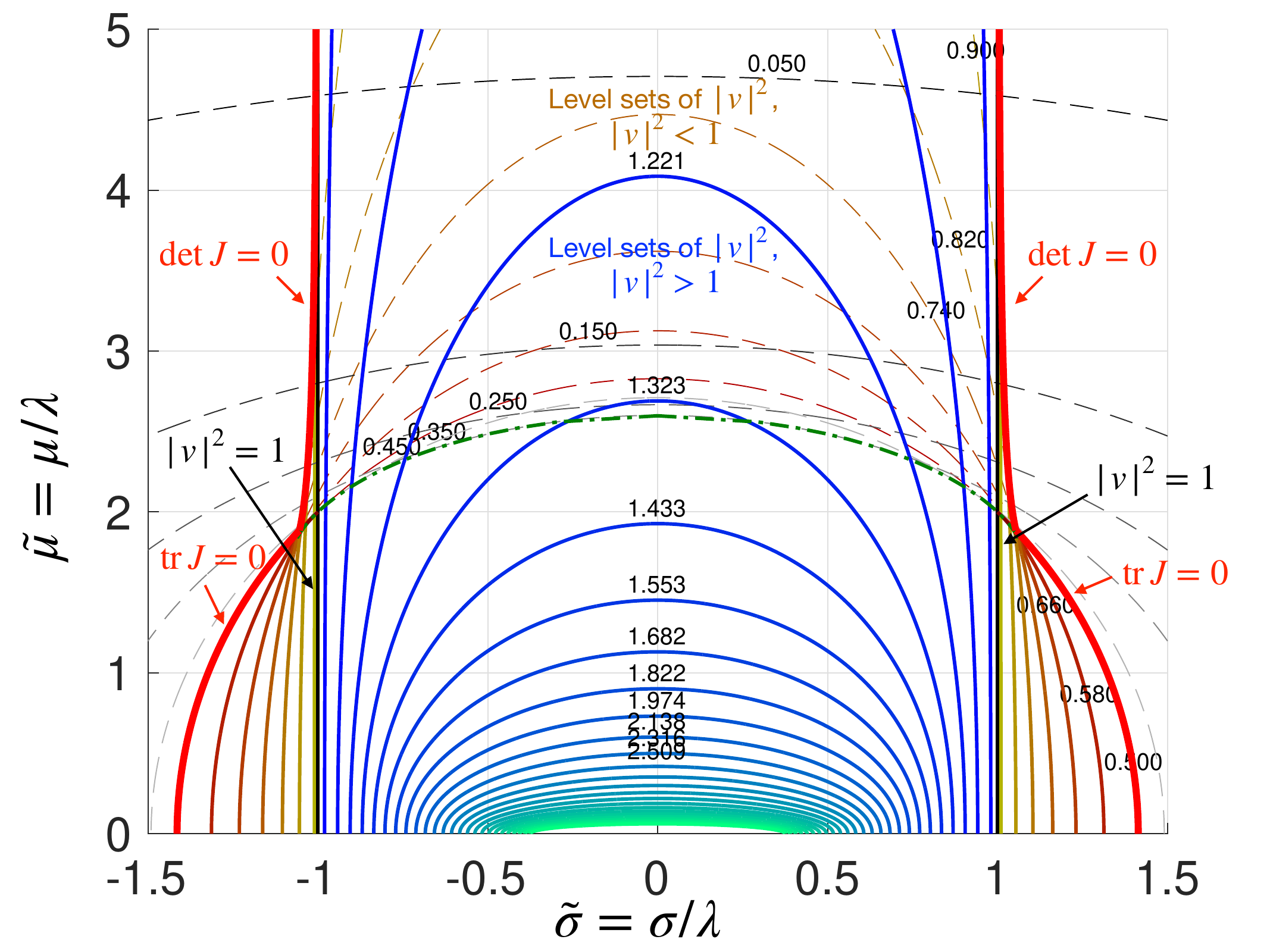}}
 \caption{The phase diagram for system \eqref{eq:system2} with the inhomogeneity in the natural frequency in the parameter space $(\tilde{\sigma} = \sigma/\lambda,\tilde{\mu} = \mu/\lambda)$.
The thick red curves belong to the boundary of the region where an asymptotically stable periodic solution to ODE \eqref{eq:system2} exists, i.e., the Stuart-Landau feedforward network with inhomogeneity in frequency is phase locked. A part of this thick red curves belong to the ellipse \eqref{eq:ellipse} corresponding to ${\rm tr}\thinspace J(|v|^2) = 0$, i.e., $|v|^2 = \tfrac{1}{2}$. The  other part belongs to the curve where $\det J = 0$ (Eq. \eqref{eq:det_bdry}) plotted dash-dotted green. The medium-weight solid curves of blue and green shades depict the level sets of $|v|^2 > 1$ corresponding to asymptotically stable periodic solutions. The black vertical lines  $\tilde{\sigma} = \pm 1$ correspond to stable periodic solutions with $|v|^2 = 1$. The medium-weight solid and thin dashed curves of red,  brown and dark yellow shades correspond, respectively, to asymptotically stable and unstable periodic solutions with $\tfrac{1}{2} < |v|^2 < 1$. The dashed curves of grey tones depict ellipses with $0 < |v|^2 < \tfrac{1}{2}$. The black numbers next to the curves indicate the corresponding values of $|v|^2$. 
}
\label{fig:system2_phase}
\vspace*{12pt}
\end{figure}

        The analysis of this family of ellipses allows 
        us to discover dynamical properties of system \eqref{eq:system2}. The stability of the equilibria of ODE \eqref{eq:vrvi} is determined by the signs of the determinant and the  trace of the Jacobian matrix of Eq. \eqref{eq:vrvi}:
            \begin{equation}
        J(v_R,v_I) = \left[\begin{array}
        {cc}\tilde{\mu}(1-|v|^2) - 2\tilde{\mu}v_R^2 & -\tilde{\sigma} - 2\tilde{\mu}v_Iv_R \\
        \tilde{\sigma} - 2\tilde{\mu}v_Rv_I & \tilde{\mu}(1-|v|^2) - 2\tilde{\mu}v_I^2\end{array}\right].
    \end{equation}
The conditions for the asymptotic stability are
\begin{align}
    \det J& = \tilde{\mu}^2 (1 - |v|^2)(1 - 3|v|^2) + \tilde{\sigma}^2 > 0,\label{eq:sys2_det_main}\\
    {\rm tr}\thinspace J & = 2\tilde{\mu}(1 - 2|v|^2) <0. \label{eq:sys2_trace_main}
\end{align}
The curve corresponding to ${\rm tr}\thinspace J = 0$ is the ellipse with $|v|^2 = \tfrac{1}{2}$:
\begin{equation}
    \label{eq:ellipse_trace0}
    \frac{\tilde{\sigma}^2}{2} + \frac{\tilde{\mu}^2}{8} = 1.
\end{equation}
It is depicted by the thick red solid line and the red dashed line in Fig. \ref{fig:system2_phase}.
The curve corresponding to $\det J = 0$ is plotted with the green dash-dotted line and the thick red solid line. A zoom-in of the kink of this curve is shown in Fig. \ref{fig:system2_phase_zoomin}. It is given parametrically by Eq. \eqref{eq:dash-dot} below.  Since the same pair $(\tilde{\sigma},\tilde{\mu})$ admits up to three solutions, the task of finding the boundary of the region where an asymptotically stable equilibrium of Eq. \eqref{eq:vrvi} exists is nontrivial. The resulting boundary is highlighted with the thick red solid curve in Fig. \ref{fig:system2_phase}.  It turns out that the curve $\det J = 0$ also bounds the region where Eq. \eqref{eq:vrvi} has three equilibria. The ellipses with $|v|^2\in(\tfrac{1}{3},1)$ touch the curve $\det J = 0$, rather than cross it, and the stability of the corresponding equilibria of Eq. \eqref{eq:vrvi} along each such ellipse changes at these tangent points. The structure and stability of the equilibria of Eq. \eqref{eq:vrvi} is summarized in Proposition \ref{prop1}.      

\begin{figure}[htbp]
\centerline{\includegraphics[width=0.9\textwidth]{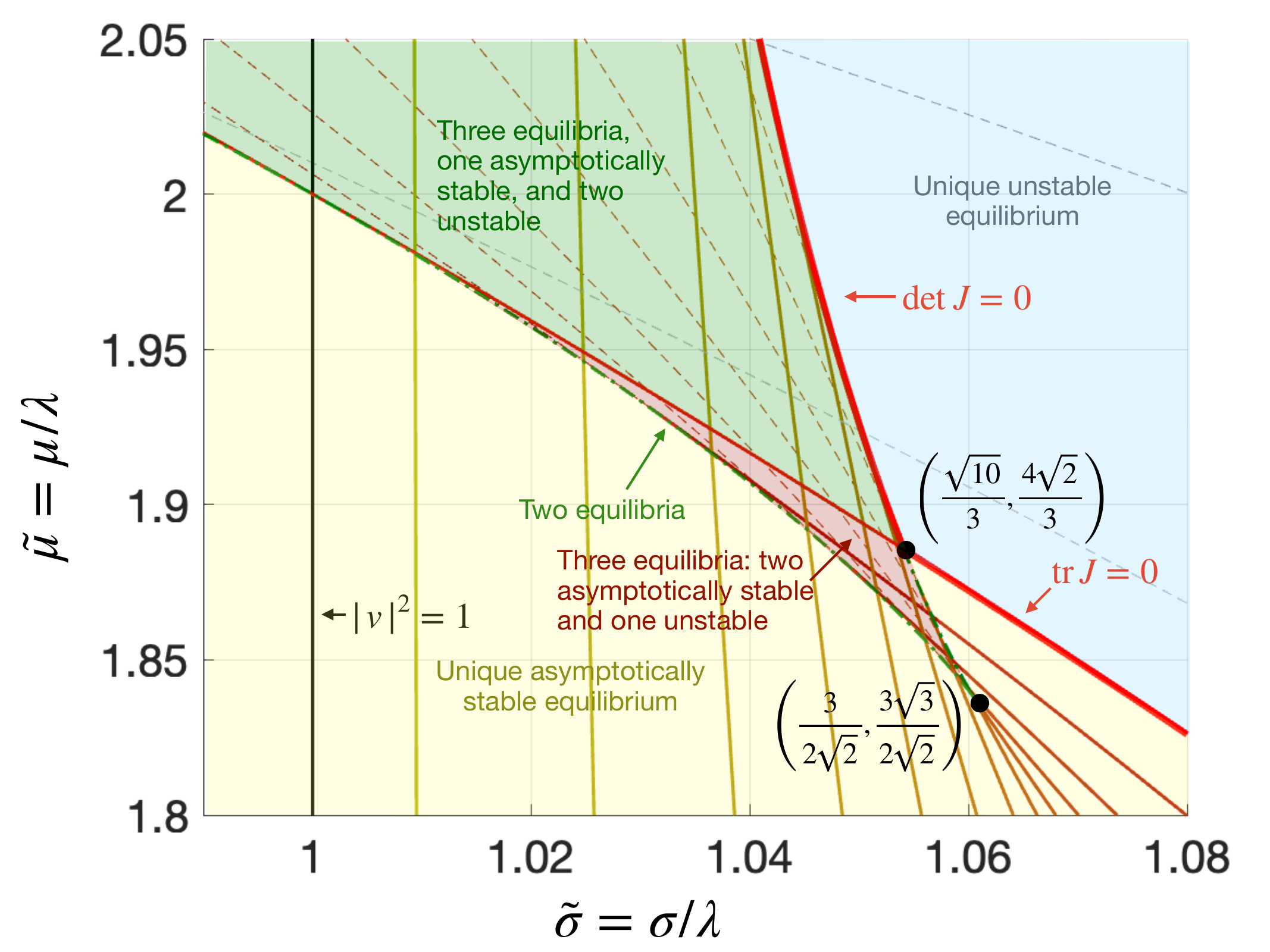}}
 \caption{The phase diagram for system \eqref{eq:system2} with the inhomogeneity in the natural frequency in the parameter space $(\tilde{\sigma} = \sigma/\lambda,\tilde{\mu} = \mu/\lambda)$: a zoom-in of the region surrounding the tip of the region where three equilibria of ODE \eqref{eq:vrvi}, one asymptotically stable exist. The region containing two asymptotically stable equilibria and one unstable equilibrium is shaded pink. The area with one asymptotically stable equilibrium and two unstable is shaded green. The areas where the unique equilibrium is asymptotically stable or unstable are yellow and blue, respectively.
}
\label{fig:system2_phase_zoomin}
\vspace*{12pt}
\end{figure}

\begin{proposition}
\label{prop1}
\begin{enumerate}
    \item  Let $|\tilde{\sigma}| \le 1$. Then ODE \eqref{eq:vrvi} has a unique equilibrium with $|v|\ge 1$ at all $\tilde{\mu} > 0$, and this equilibrium is asymptotically stable. 

    \item If $\tilde{\sigma} = 0$, the asymptotically stable equilibrium solution to ODE \eqref{eq:vrvi} is real and blows up as $|v|\approx\tilde{\mu}^{-1/3}$ as $\tilde{\mu}\rightarrow0$.

    \item The region in the $(\tilde{\sigma},\tilde{\mu})$-space where an asymptotically stable equilibrium of ODE \eqref{eq:vrvi} exists is bounded by the arcs of the ellipse
    \begin{equation}
    \label{eq:ellipse}
        \frac{\tilde{\mu}^2}{8} + \frac{\tilde{\sigma}^2}{2} = 1,\quad \tilde{\mu} \le \frac{4}{\sqrt{5}}|\tilde{\sigma}|,
    \end{equation}
    and the curves defined implicitly as:
    \begin{equation}
    \label{eq:det_bdry}
    |v|^2 \in \left[\frac{3}{4},1\right),\quad \tilde{\sigma} = \pm\sqrt{\frac{3|v|^2 - 1}{2|v|^4}},\quad \tilde{\mu} = \frac{1}{|v|^2\sqrt{2(1-|v|^2)}},
    \end{equation}
    The amplitude squared, $|v|^2$, of the asymptotically stable equilibrium solution ranges from $\tfrac{1}{2}$ to $+\infty$. 
    
    \item Three equilibria of ODE \eqref{eq:vrvi} exist in the region above the concatenation of the parametric curve 
     \begin{equation}
        \label{eq:dash-dot}
        |v|^2 \in \left[\frac{1}{3},1\right),\quad \tilde{\sigma} = \sqrt{\frac{3|v|^2 - 1}{2|v|^4}},\quad \tilde{\mu} = \frac{1}{|v|^2\sqrt{2(1-|v|^2)}}
    \end{equation}
    and its reflection with respect to the $\tilde{\mu}$-axis. This curve
    starts at $(0,\tfrac{3\sqrt{3}}{2})$ at $|v|^2 = \tfrac{1}{3}$, has a singularity at $|v|^2 = \tfrac{2}{3}$, where $(\tilde{\sigma},\tilde{\mu}) = (\tfrac{3}{2\sqrt{2}},\tfrac{3\sqrt{3}}{2\sqrt{2}})$, associated with $\tfrac{d\tilde{\sigma}}{d|v|^2}(\tfrac{2}{3}) = 0$, and approaches the vertical asymptote $|v|^2 = 1$ from the right as $|v|^2\rightarrow 1$. 

    The region with three equilibria is divided into three subregions by the ellipse $\tfrac{\tilde{\mu}^2}{8} + \tfrac{\tilde{\sigma}^2}{2} =1$. Two asymptotically stable equilibria exist in the two subregions lying inside this ellipse and above the curve \eqref{eq:dash-dot}. Only one asymptotically stable equilibrium exists in the subregion outside the ellipse and above the curve \eqref{eq:dash-dot}.

\end{enumerate} 
\end{proposition}

A proof of Proposition \ref{prop1} is found in Appendix \ref{appB}. Here, we make a few remarks.
\begin{remark}
    \begin{enumerate}
        \item The blow up of $|v|$ as $ \tilde{\mu}^{-1/3}$ at $\tilde{\sigma} = 0$ as $\tilde{\mu}\rightarrow0$ is consistent with the result proven in \cite{Levasseur_Palacios_2021}. Indeed, recalling that $\tilde{\mu} = \mu/\lambda$ and $z_2 = \sqrt{\mu}ve^{i\omega t}$, we obtain $|z_2|\approx \lambda^{1/3}\mu^{1/6}$ as $\mu\rightarrow 0$.

        \item If $\tilde{\sigma} \neq 0$, the amplification factor of the amplitude in the second cell is bounded from above by $\tilde{\sigma}^{-2} = \lambda^2/\sigma^2$. Hence, increasing the coupling strength and reducing frequency inhomogeneity increases the amplification effect.
        

        \item In the phase-locked regime, i.e., when an asymptotically stable periodic solution to ODE \eqref{eq:system2} exists, the amplitude in the second cell may be larger or smaller than the amplitude in the first cell. The ratio of the amplitudes is bounded from below by $\tfrac{1}{2}$, i.e., $|z_2|/|z_1| > \tfrac{1}{2}$. This immediately follows from the stability condition that the trace of the Jacobian of Eq. \eqref{eq:vrvi} must be negative -- see the proof.

    \end{enumerate}
\end{remark}

\subsubsection{Invariant tori}
 Fig.~\ref{fig:system2_phase} shows that when the natural frequency difference $|\tilde{\sigma}|$ is large enough, then the system has no periodic attractors. In this case, system \eqref{eq:system2} settles on an invariant torus attractor, so that cell $z_1$ performs a periodic motion with frequency $\omega$, and cell $z_2$ settles onto a limit cycle in the co-rotating frame of frequency $\omega$. If we fix $\tilde{\mu}$ and move along $\tilde{\sigma}$, the torus attractor appears as a result of a saddle-node bifurcation, if $\tilde{\mu} > \tfrac{4\sqrt{2}}{3}\approx 1.8856$, and as a result of a Hopf bifurcation, if $0 <\tilde{\mu} < \tfrac{4\sqrt{2}}{3}$. Fig. \ref{fig:system2_quiver} illustrates the change of the direction field and the birth of the torus attractor, which corresponds to a limit cycle in the co-rotating frame, at $\tilde{\mu} = 3.0$, $1.91$, and $0.5$.
 The invariant tori at $\lambda = 1$ and selected values of $\mu$ and $\sigma$ soon after the bifurcation point are shown in Fig. \ref{fig:system2_inv_tori}.
 
 \begin{figure}[htbp]
\centerline{\includegraphics[width=0.9\textwidth]{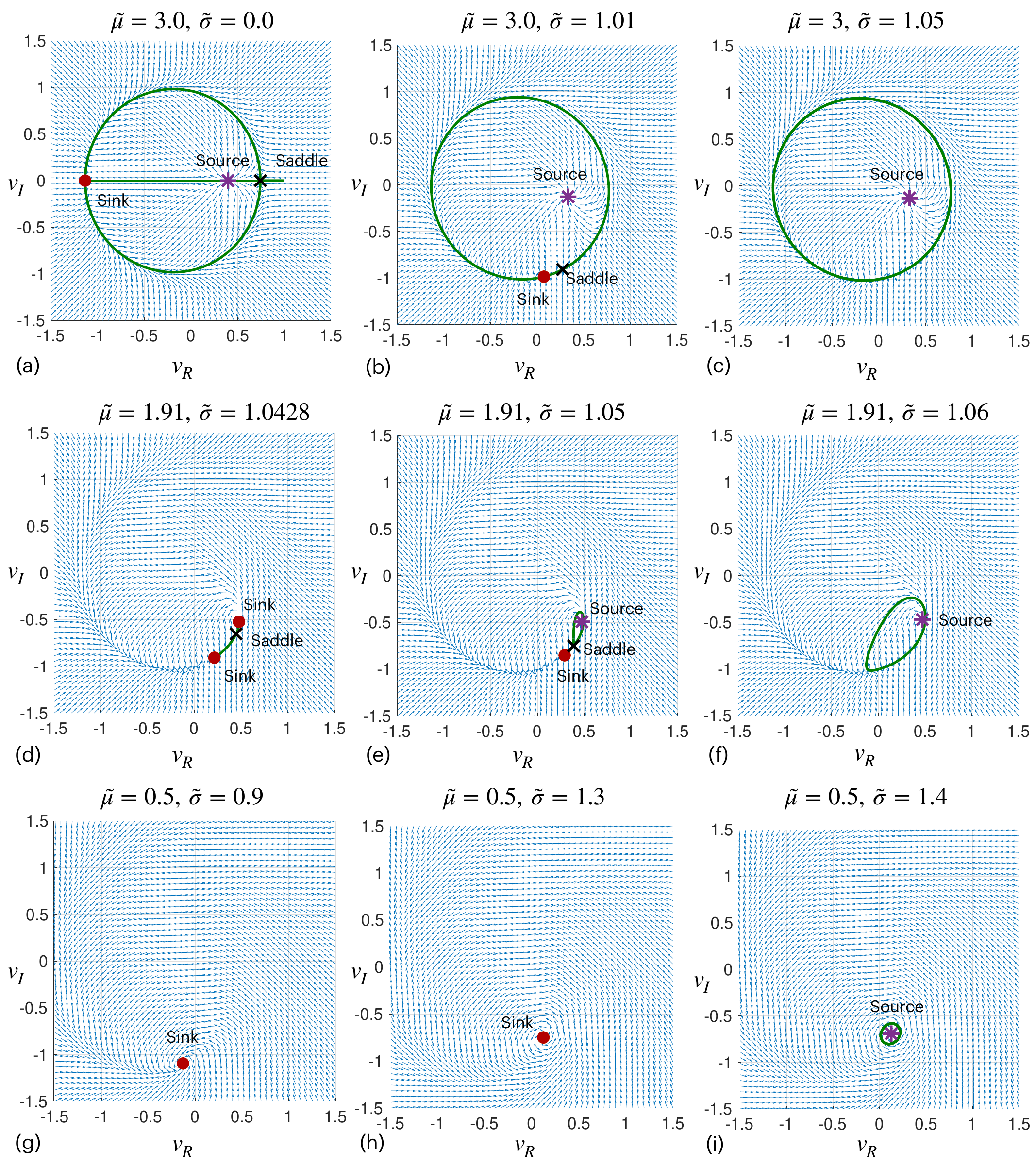}}
 \caption{The direction fields of ODE \eqref{eq:vrvi} at selected values of $\tilde{\mu}$ and $\tilde{\sigma}$. The birth of the limit cycle of ODE \eqref{eq:vrvi}, which corresponds to the invariant torus of ODE \eqref{eq:system2}, occurs via a saddle-node bifurcation, as at $\tilde{\mu} = 3$ and $1.91$, or a Hopf bifurcation, as at $\tilde{\mu} = 0.5$. The plot in (d) exemplifies a bistability situation.
}
\label{fig:system2_quiver}
\vspace*{12pt}
\end{figure}

 
 \begin{figure}[htbp]
\centerline{\includegraphics[width=0.9\textwidth]{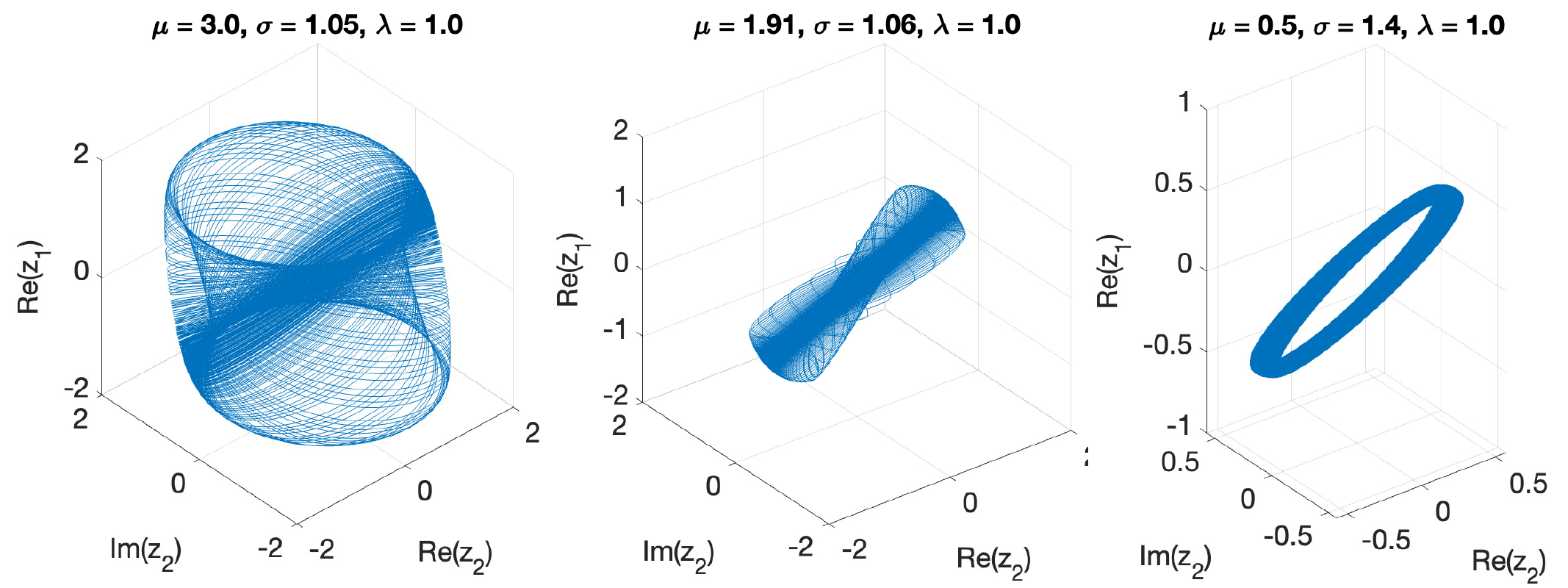}}
 \caption{The invariant tori attractors of ODE \eqref{eq:system2}.   
}
\label{fig:system2_inv_tori}
\vspace*{12pt}
\end{figure}

\subsection{Inhomogeneity in the excitation parameter and the natural frequency}
\label{sec:system3}

In this section, we consider system \eqref{eq:system3} with inhomogeneities in the excitation parameter $\mu$ and  the natural frequency $\omega$, represented by $\varepsilon$ and $\sigma$, respectively. The schematic of the feedforward network is shown in Fig.~\ref{fig:FFwd_Disorder_mu_omega}. 

\begin{figure}[H]
\begin{minipage}{.4\textwidth}
  \centering
\begin{tikzpicture}[->,>=stealth',shorten >=1pt,auto,node distance=2.4cm,
  thick,node1/.style={circle,fill=violet!30,draw,font=\sffamily\Large\bfseries},node2/.style={circle,fill=green!30,draw,font=\sffamily\Large\bfseries}]

  \node[node1] (1) {1};
  \node[node2] (2) [right of =1] {2};

  \path[every node/.style={font=\sffamily\Large}]
   (1) edge node {$\lambda$} (2);
   
\end{tikzpicture}
\end{minipage}%
\begin{minipage}{.6\textwidth}
 \centering
\begin{equation} \label{eq:system3}
\begin{aligned}
\dot{z}_1 & = (\mu +i\omega)z_1 - |z_1|^2 z_1\\
\dot{z}_2 & = ((\mu + \varepsilon)  + i(\omega +\sigma))z_2 - |z_2|^2z_2  - \lambda z_1. \end{aligned}
\end{equation}
\end{minipage}
\caption{A schematic diagram of a two-cell feed-forward network with Stuart-Landau cells with inhomogeneities in the natural frequency and the excitation parameter represented by $\sigma$ and $\varepsilon$, respectively.}
\label{fig:FFwd_Disorder_mu_omega}
\end{figure}

We assume that $\sigma\in\mathbb{R}$, $\mu > 0$ and $\lambda > 0$. Further, we will distinguish two cases: $\mu + \varepsilon \le 0$ and $\mu + \epsilon > 0$. If $\mu + \varepsilon \le 0$, there exists a unique globally attracting periodic solution to Eq. \eqref{eq:system3}. If $\mu + \varepsilon > 0$, an appropriate variable change together with a time scaling and a redefinition of parameters reduces the number of parameters to two and reduces Eq. \eqref{eq:system3} to Eq. \eqref{eq:vrvi} analyzed in Section \ref{sec:system2}.  

\subsubsection{The reduced system}
\label{sec:system3_reduced}
 We assume that the first cell, $z_1$, is at its periodic attractor, i.e., $z_1 = \sqrt{\mu}e^{i\omega t}$. Let $z_2(t) = u(t)e^{i\omega t}$, where $u(t)\in\mathbb{C}$ is the second cell in the co-rotating frame. The ODE for $u$ is
\begin{equation}
    \label{eq:syste3_u}
    \dot{u} = (\mu + \varepsilon) u - |u|^2u +i\sigma u - \lambda \sqrt{\mu}.
\end{equation}
\begin{itemize}
    \item {\bf Case $\boldsymbol{\mu} \boldsymbol{+} \boldsymbol{\varepsilon} \boldsymbol{>} \boldsymbol{0}$.}
Rescaling time, introducing a new dependent variable $v$, and redefining the parameters as
\begin{equation}
    \label{eq:redef_mu_sigma1}
    \tau = \lambda\sqrt{\tfrac{\mu}{\mu + \varepsilon}} t,\quad u = v\sqrt{\mu + \varepsilon},\quad \tilde{\mu}: = \frac{\mu + \varepsilon}{\lambda}\sqrt{\frac{\mu + \varepsilon}{\mu}},\quad \tilde{\sigma} =  \frac{\sigma}{\lambda}\sqrt{\frac{\mu + \varepsilon}{\mu}},
\end{equation} 
we obtain the ODE for $v$ that matches Eq. \eqref{eq:v_sys2}:
\begin{equation}
    \label{eq:system3_v1}
    \dot{v} = \tilde{\mu}(1-|v|^2)v + i\tilde{\sigma} v - 1.
\end{equation}
Therefore, we refer the reader to Section \ref{sec:system2} where a detailed analysis of this ODE is conducted.

\item  {\bf Case $\boldsymbol{\mu} \boldsymbol{+} \boldsymbol{\varepsilon} \boldsymbol{<} \boldsymbol{0}$.}
We redefine time, the dependent variable, and the parameters as
\begin{equation}
    \label{eq:redef_mu_sigma2}
    \tau = \lambda\sqrt{\tfrac{\mu}{|\mu + \varepsilon|}} t,\quad u = v\sqrt{|\mu + \varepsilon|},\quad \tilde{\mu}: = \frac{|\mu + \varepsilon|}{\lambda}\sqrt{\frac{|\mu + \varepsilon|}{\mu}},\quad \tilde{\sigma} =  \frac{\sigma}{\lambda}\sqrt{\frac{|\mu + \varepsilon|}{\mu}}
\end{equation} 
and obtain the following ODE for $v$:
\begin{equation}
    \label{eq:system3_v2}
    \dot{v} = -\tilde{\mu}(1+|v|^2)v + i\tilde{\sigma} v - 1.
\end{equation}
The ODE system for the real and imaginary components of $v$ is
\begin{equation}
    \label{eq:vrvi_minus}
    \begin{aligned}
        \dot{v}_R & = -\tilde{\mu}(1+|v|^2)v_R - \tilde{\sigma}v_I -1,\\
        \dot{v}_I & = -\tilde{\mu}(1+|v|^2)v_I + \tilde{\sigma}v_R.
    \end{aligned}
\end{equation}
Looking for equilibria of Eq. \eqref{eq:vrvi_minus} and following the steps conducted in Section \ref{sec:system2_periodic}, we obtain the following equation for $|v|^2$:
\begin{equation}
    \label{eq:v2_minus}
    \tilde{\mu}^2(1+|v|^2)^2 |v|^2 +\tilde{\sigma}^2|v|^2 = 1.
\end{equation}

\item {\bf Case $\boldsymbol{\mu} \boldsymbol{+} \boldsymbol{\varepsilon} \boldsymbol{=} \boldsymbol{0}$.}
We redefine time, the dependent variable, and the parameters as
\begin{equation}
    \label{eq:redef_mu_sigma0}
    \tau = \lambda t,\quad u = v\sqrt{\mu},\quad \tilde{\mu}: = \frac{\mu}{\lambda},\quad \tilde{\sigma} =  \frac{\sigma}{\lambda}
\end{equation} 
and obtain the ODE
\begin{equation}
    \label{eq:system3_v0}
    \dot{v} = -\tilde{\mu}|v|^2v + i\tilde{\sigma} v - 1.
\end{equation}
The corresponding ODE system for $v_R$ and $v_I$, where $v = v_R + iv_I$, is
\begin{equation}
    \label{eq:vrvi_zero}
    \begin{aligned}
        \dot{v}_R & = -\tilde{\mu}|v|^2v_R - \tilde{\sigma}v_I -1,\\
        \dot{v}_I & = -\tilde{\mu}|v|^2v_I + \tilde{\sigma}v_R.
    \end{aligned}
\end{equation}
The squared absolute value at the equilibria of Eq. \eqref{eq:vrvi_zero} must satisfy
\begin{equation}
    \label{eq:v2_zero}
    \tilde{\mu}^2|v|^6 +\tilde{\sigma}^2|v|^2 = 1.
\end{equation}
\end{itemize}

Eqs. \eqref{eq:v2_minus} and \eqref{eq:v2_zero} have a unique solution $|v|^2 > 0$ at any positive $\tilde{\mu}^2$ and any $\tilde{\sigma}^2$, because their left-hand sides monotonously increase from zero to infinity as $|v|^2$ grows from zero to infinity, while their right-hand sides are positive constants.  Eq. \eqref{eq:v2_minus} defines a family of ellipses centered at the origin of the $(\tilde{\sigma},\tilde{\mu})$-plane with the major semiaxis $|v|^{-1}$ in the direction of $\tilde{\sigma}$ and the minor semiaxis $(1+|v|^2)^{-1}|v|^{-1}$ in the direction of $\tilde{\sigma}$. 
Eq. \eqref{eq:v2_zero} also defines a family of ellipses centered at the origin of the $(\tilde{\sigma},\tilde{\mu})$-plane and semiaxes $|v|^{-1}$ and $|v|^{-3}$ in the directions of $\tilde{\sigma}$ and $\tilde{\mu}$, respectively. The roles of the major and minor semiaxes switch at $|v| = 1$.

The stability conditions are obtained by computing the Jacobians $J$ of ODEs \eqref{eq:vrvi_minus} and \eqref{eq:vrvi_zero} and requiring that $\det{J} > 0$ and ${\rm tr}\thinspace J < 0$. For ODEs \eqref{eq:vrvi_minus} and \eqref{eq:vrvi_zero}, these conditions, respectively, are:
\begin{align}
&\begin{cases}
     \det J = &\tilde{\mu}^2 (1 + |v|^2)(1 + 3|v|^2) + \tilde{\sigma}^2 > 0\\
    {\rm tr}\thinspace J  = &-2\tilde{\mu}(1 + 2|v|^2) <0 
\end{cases}, \label{eq:sys3_det_minus}\\
&\begin{cases}
     \det J = &3\tilde{\mu}^2|v|^4 + \tilde{\sigma}^2 > 0\\
    {\rm tr}\thinspace J  = &-4\tilde{\mu}|v|^2 <0 
\end{cases}. \label{eq:sys3_det_zero}
\end{align}
These stability conditions hold for all $\tilde{\mu} > 0$, $|v|^2 > 0$, and all $\tilde{\sigma}\in \mathbb{R}$. 
Phase diagrams in the $(\tilde{\sigma},\tilde{\mu})$-planes with the level sets of $|v|^2$ equilibria for the cases where $\mu+\varepsilon \neq 0$ and $\mu+\varepsilon = 0$ are shown in Fig. \ref{fig:system3_ellipses}.
\begin{figure}[htbp]
\centerline{\includegraphics[width=0.9\textwidth]{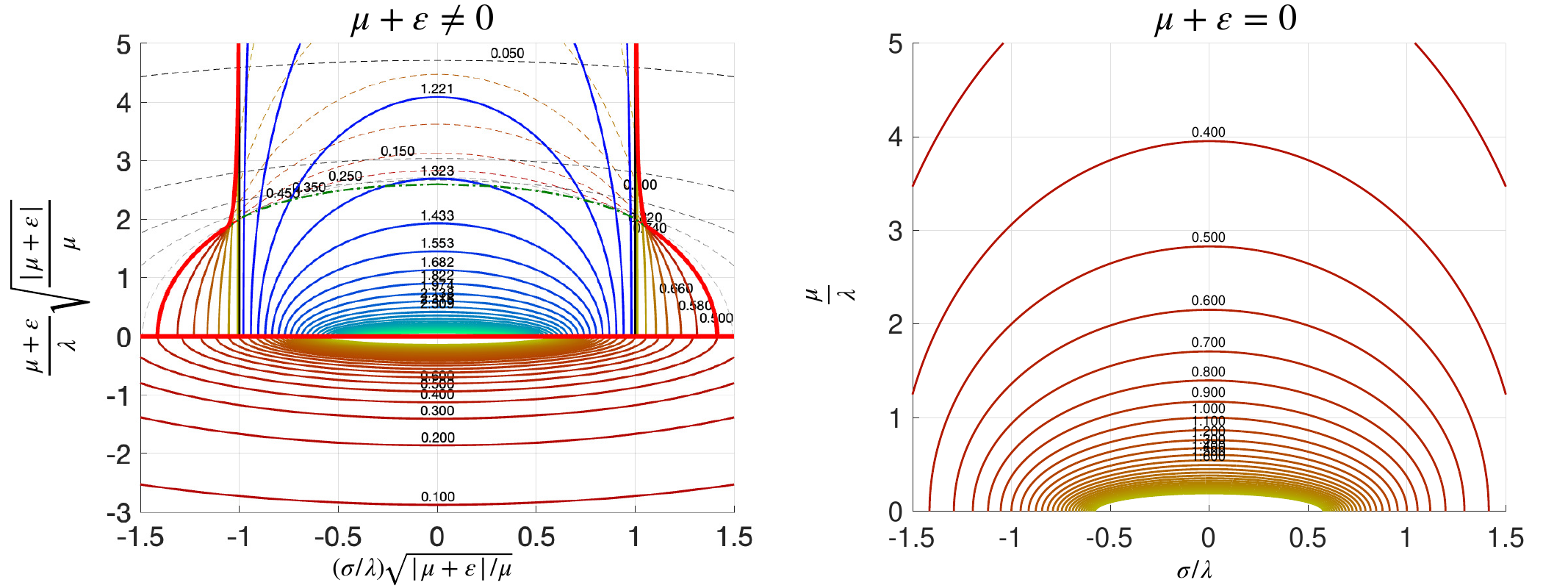}}
\caption{System \eqref{eq:system3}. Phase diagrams for reduced ODEs \eqref{eq:system3_v1} and \eqref{eq:system3_v2}, when $\mu+\varepsilon\neq 0$ (left), and \eqref{eq:system3_v0}, when $\mu+\varepsilon = 0$ (right). The solid and dashed curves are level sets of $|v|^2$ and asymptotically stable and unstable equilibria, respectively. }
\label{fig:system3_ellipses}
\vspace*{12pt}
\end{figure}

The dynamical properties of system \eqref{eq:system3} are summarized in the following proposition.
\begin{proposition}
    \label{prop2}
    Consider ODE \eqref{eq:system3} with $\mu > 0$ and $\lambda > 0$. Suppose cell $z_1$ is at its rotating attractor $z_1(t) = \sqrt{\mu}e^{i\omega t}$.
    \begin{enumerate}
        \item If $\mu+\varepsilon > 0$, then time rescaling, variable change, and parameter redefinition in Eq. \eqref{eq:redef_mu_sigma1} result in ODE \eqref{eq:system3_v1} in the co-rotating frame with angular velocity $\omega$. ODE \eqref{eq:system3_v1} coincides with ODE \eqref{eq:v_sys2}, and Proposition \ref{prop1} holds for it.
        
        \item If $\mu+\varepsilon < 0$, then time rescaling, variable change, and parameter redefinition in Eq. \eqref{eq:redef_mu_sigma2} result in ODE \eqref{eq:system3_v2} in the co-rotating frame with angular velocity $\omega$. At each $\tilde{\mu}>0$ and $\tilde{\sigma}\in\mathbb{R}$, ODE \eqref{eq:system3_v2} has a unique equilibrium solution. The level sets of these equilibria lie on ellipses given by Eq. \eqref{eq:v2_minus}. If $\tilde{\sigma} = 0$, the magnitude $|v|$ of this equilibrium blows us as $\tilde{\mu}^{-1/3}$ as $\tilde{\mu}\rightarrow 0$.
        
        \item If $\mu+\varepsilon = 0$, then time rescaling, variable change, and parameter redefinition in Eq. \eqref{eq:redef_mu_sigma0} result in ODE \eqref{eq:system3_v0} in the co-rotating frame with angular velocity $\omega$. This ODE has a unique equilibrium solution at any $\tilde{\mu}$ and $\tilde{\sigma}$. The level sets of these equilibria lie on ellipses given by Eq. \eqref{eq:v2_zero}. If $\tilde{\sigma} = 0$, the magnitude $|v|$ of this equilibrium blows us as $\tilde{\mu}^{-1/3}$ as $\tilde{\mu}\rightarrow 0$.       
    \end{enumerate}
\end{proposition}


\subsubsection{Bifurcation diagrams of the original system} 
\label{sec:system3_paths}
The system reduction applied in Section \ref{sec:system3_reduced} greatly simplified its analysis and reduced the hardest case with multiple periodic solutions of system \eqref{eq:system3} to the reduced system \eqref{eq:system2}. However, the parameter redefinition in Eq. \eqref{eq:redef_mu_sigma1} obscures how attractors of system \eqref{eq:system3} change as we fix all parameters but one, and plot a bifurcation diagram with respect to the unfixed parameter. The simplest case is when the unfixed parameter is $\lambda$. Since both $\tilde{\sigma}$ and $\tilde{\mu}$ are inversely proportional to $\lambda$, changing $\lambda$ shrinks or expands the phase diagram without changing its structure. Therefore, we will keep $\lambda = 1$ throughout this section. The second simplest unfixed parameter is $\sigma$. Since $\tilde{\mu}$ is independent of $\sigma$, and $\tilde{\sigma}$ is directly proportional to $\sigma$, the paths in the phase plane $(\tilde{\sigma},\tilde{\mu})$ for these case are horizontal lines. Not-so-straightforward cases involve the unfixed parameters $\mu$ and $\varepsilon$. The corresponding paths are shown in Fig. \ref{fig:system3_paths}, left and right, respectively. For simplicity, we will refer to these paths as $\mu$- and $\varepsilon$-paths, respectively.
\begin{figure}[htbp]
\centerline{
\includegraphics[width=0.9\textwidth]
{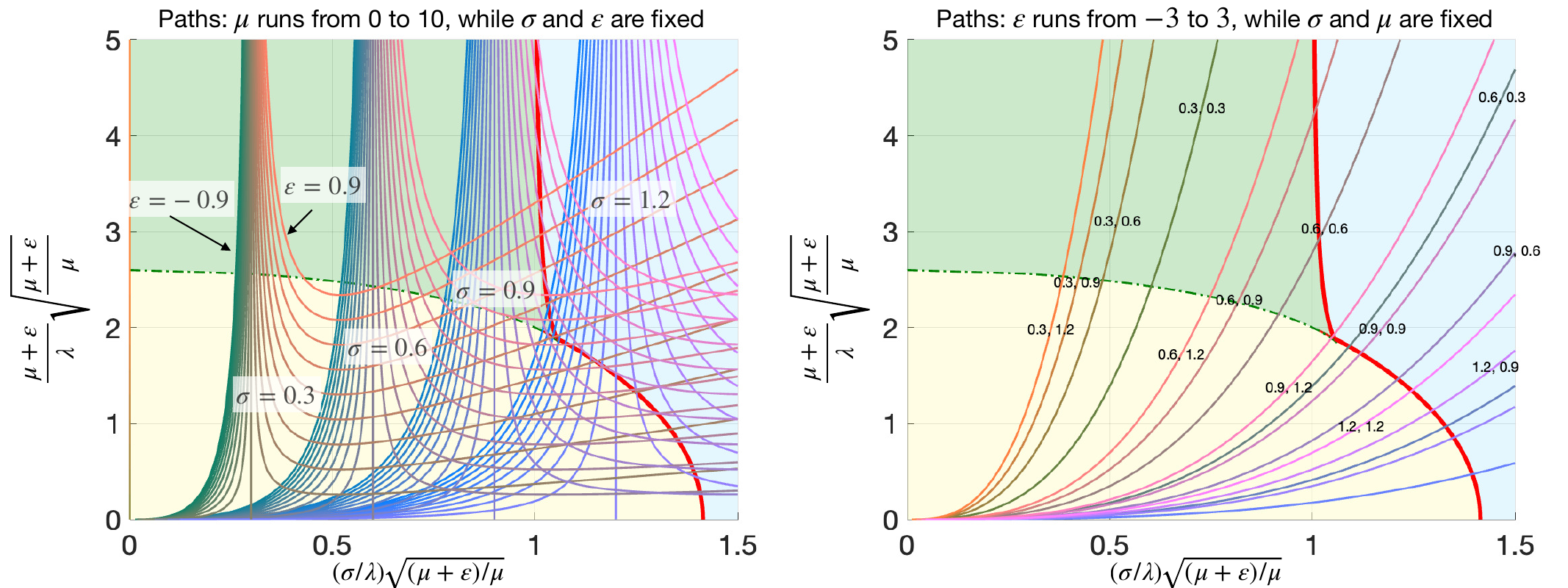}
}
\caption{Left: Paths corresponding to changing $\mu$ at fixed $\sigma$ and $\varepsilon$ on the $(\tilde{\sigma},\tilde{\mu})$-phase diagram. Paths at each fixed $\sigma$ and a set of $\varepsilon$ values $\{-0.9,-0.8,\ldots,0.9\}$ tend to the vertical asymptote $\tilde{\sigma} = \sigma$. The paths with $\varepsilon < 0$ tend to zero as $\tilde{\mu}\rightarrow 0$, while paths with $\varepsilon > 0$ are bounded away from zero.
Right: Paths corresponding to changing $\varepsilon$ at fixed $\sigma$ and $\mu$ on the $(\tilde{\sigma},\tilde{\mu})$-phase diagram. The curves are annotated with the corresponding values of  $\sigma$ and $\mu$, where the first value is $\sigma$, and the second one is $\mu$.  
}
\label{fig:system3_paths}
\vspace*{12pt}
\end{figure}

The $\varepsilon$-paths start in the yellow region where a unique asymptotically stable periodic solution of system \eqref{eq:system3} exists and eventually enter the region with only a torus attractor. In between, they may traverse the green region with two additional unstable periodic solutions, or the tiny pink region visible in Fig. \ref{fig:system2_phase_zoomin} with two asymptotically stable and one unstable periodic solutions of system \eqref{eq:system3}. 

The $\mu$-paths may have more complex bifurcation diagrams because they may enter the region with an asymptotically stable periodic solution of system \eqref{eq:system3} multiple times. An example of such a $\mu$-path at $\sigma = 0.98$ and $\varepsilon = 0.2$ and the resulting bifurcation diagram generated with the aid of the continuation software package XPP AUT~\cite{XPPAUT} is displayed in Fig. \ref{fig:Samir_Masha}. 
The top-left panel shows the norm of the real part of the vector $[z_1,z_2]$  as a function of the excitation parameter, $\mu$. The bottom-right panel displays the corresponding $\mu$-path in the $(\tilde{\sigma},\tilde{\mu})$-plane. 

\begin{figure}[htbp]
\centerline{
\includegraphics[width=0.9\textwidth]
{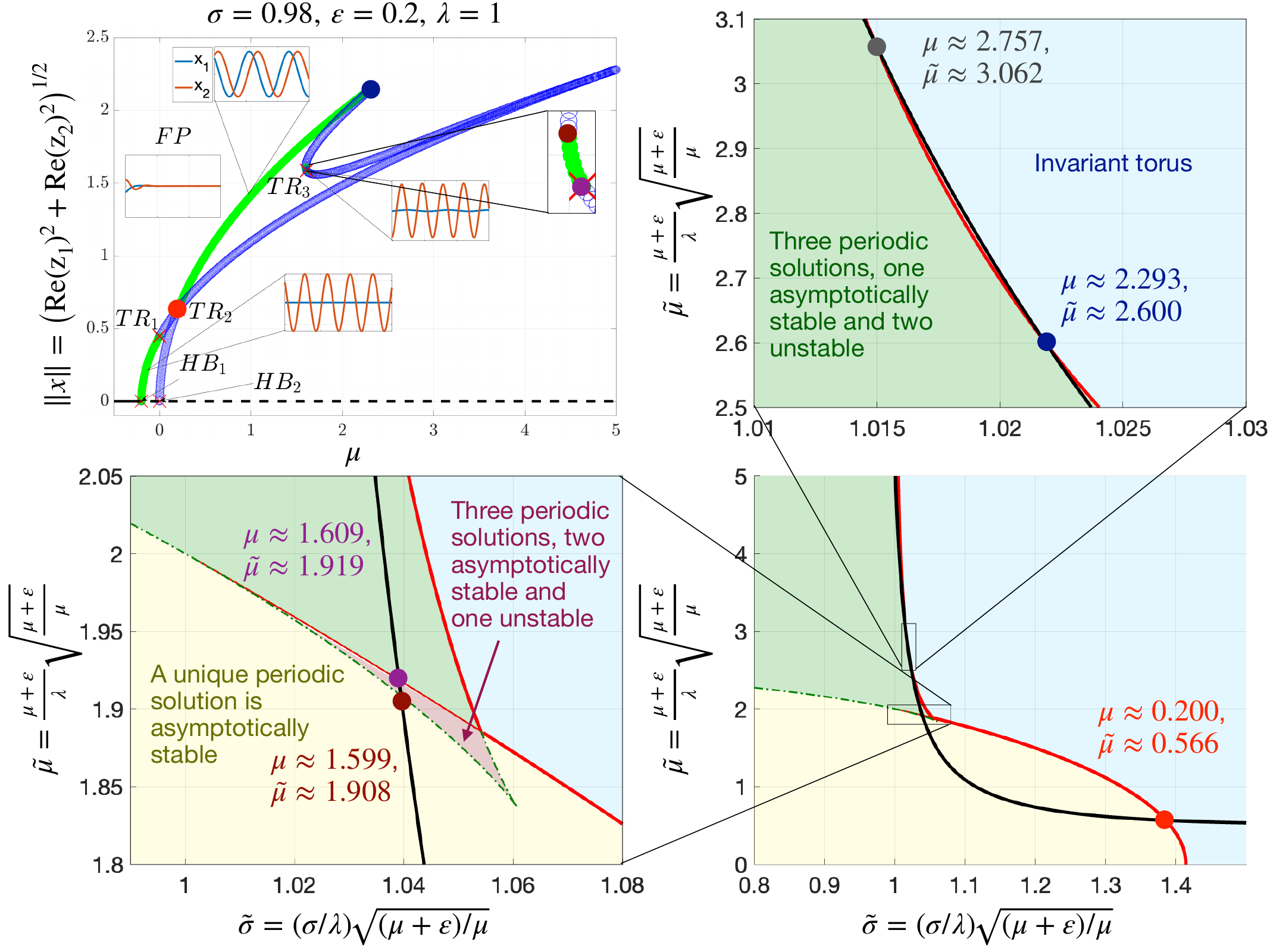}
}
\caption{A bifurcation diagram  (top left) of the response of Eq.~\eqref{eq:system3} as a function of the excitation parameter $\mu$. All other parameter values are shown in the figure. Green (blue) circles depict stable (unstable) oscillations that arise via Hopf bifurcations, labeled $HB_1$ and $HB_2$. Tori bifurcations, which correspond to secondary Hopf bifurcations, are labeled $TR_1$ and $TR_2$. They lead to quasi-periodic motion. The corresponding $\mu$-path is shown in the bottom right, and two zoom-ins of its traverse through the pink region with two asymptotically stable periodic solutions and its re-entry of the green region are shown in the zoom-ins. As $\mu$ increases, starting with negative values, a Hopf bifurcation, labeled $HB_1$, occurs at $\mu = -\varepsilon = -0.2$. This bifurcation brings cell two to oscillate at frequency $\omega+\sigma$, while cell one is at rest, and to follow the green branch in the interval $-0.2<\mu<0$. Another Hopf bifurcation, $HB_2$, at $\mu = 0$ gives rise to limit-cycle oscillations of cell one at frequency $\omega$. Since cell two is already oscillating at a different frequency, the combined effect is a torus bifurcation, $TR_1$, leading to quasi-periodic oscillations. At $\mu > 0$, the events on the $(\mu,\|x\|)$-bifurcation diagram on the top-left panel can be understood by following the corresponding $\mu$-path on the $(\tilde{\mu},\tilde{\sigma})$ phase diagram on the bottom-right panel. The $\mu$-path enters the yellow region at the point marked by a bright red dot and labeled by $TR_2$ in the $(\mu,\|x\|)$-diagram. The torus attractor disappears at this point, giving rise to a periodic attractor represented by the green branch starting at the bright red dot and ending at the dark blue dot in the top-left panel. The $\mu$-path traverses the pink region zoomed in on the bottom-left panel, where two periodic attractors exist. The second periodic solution with smaller amplitude in the second cell lies on the slanted-j-shaped branch in the top-left panel, with a tiny stable interval zoomed in on in the inset. The $\mu$-path enters the green region in the bottom-right panel, then leaves it, and re-enters it again. This excursion to the blue region is zoomed in on the top-right panel. The green branch at the top-left ends at the exit point: XPP AUT has not found its continuation. The traverse of the $\mu$-path through the pink and green regions is marked by the existence of three periodic phase-locked solutions lying on the green and the slanted-j-shaped branches in which both cells oscillate at frequency $\omega$. The other blue branch emanating from the point $TR_2$ corresponds to the unstable periodic solution where cell one is at rest and cell two oscillates at frequency $\omega + \sigma$.
}
\label{fig:Samir_Masha}
\vspace*{12pt}
\end{figure}
\medskip

\subsection{The effect of the cubic imaginary term}
\label{sec:system4}

\begin{figure}[H]
\begin{minipage}{.4\textwidth}
  \centering
\begin{tikzpicture}[->,>=stealth',shorten >=1pt,auto,node distance=2.4cm,
  thick,node1/.style={circle,fill=violet!30,draw,font=\sffamily\Large\bfseries},node2/.style={circle,fill=green!30,draw,font=\sffamily\Large\bfseries}]

  \node[node1] (1) {1};
  \node[node2] (2) [right of =1] {2};

  \path[every node/.style={font=\sffamily\Large}]
   (1) edge node {$\lambda$} (2);
   
\end{tikzpicture}
\end{minipage}%
\begin{minipage}{.6\textwidth}
 \centering
\begin{equation} \label{eq:system4}
\begin{aligned}
\dot{z}_1 & = (\mu +i\omega)z_1 - |z_1|^2 z_1\\
\dot{z}_2 & = ((\mu + \varepsilon)  + i(\omega +\sigma))z_2 - (1 + \gamma \, i) |z_2|^2z_2  - \lambda z_1. \end{aligned}
\end{equation}
\end{minipage}
\caption{A schematic diagram of a two-cell feed-forward network with Stuart-Landau cells with inhomogeneities in the natural frequency, the excitation parameter, and the cubic nonlinearity, represented by $\sigma$, $\varepsilon$, and $\gamma$, respectively.}
\label{fig:FFwd_Disorder_mu}
\end{figure}
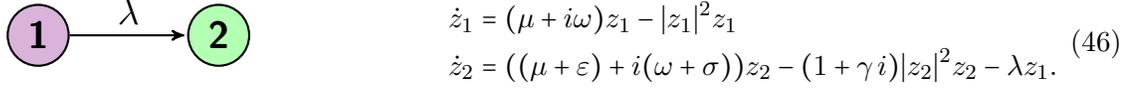

In this section, we consider the case where the complex cubic nonlinearity coefficient, $\gamma$, is nonzero. Figure~\ref{fig:FFwd_Disorder_mu} illustrates this case. A near-identity transformation of the form $z_2 = \tilde{z}_2\left(1+i\tfrac{\gamma}{\mu+\varepsilon}|\tilde{z}_2|^2\right) $ eliminates the complex coefficient $\gamma$ in the cubic nonlinearity, but introduces an $O(|z_2|^5)$ term. If we were only interested in bifurcations when $\mu$ changes from a negative to a small positive value, these higher-order terms would have an insignificant effect; they are not included in the normal form. By contrast, we aim to investigate the dynamics of system \eqref{eq:system3} across a broad range of parameter values, thereby making the effect of these higher-order terms substantial. Therefore, we retain the $\gamma$-term and examine system \eqref{eq:system3} using the approach employed for $\gamma = 0$.

We assume that the first cell in Eq. \eqref{eq:system4} is settled onto its periodic attractor $z_1(t) = \sqrt{\mu}e^{i\omega t}$. Then the ODE for the second cell in Eq. \eqref{eq:system4} is

\begin{equation}
\label{eq:gamma_z_2}
\dot{z}_2 = (\mu + \varepsilon +i(\omega + \sigma))z_2 - (1+i\gamma)|z_2|^2 z_2 -\lambda\sqrt{\mu}e^{i\omega t}.
    \end{equation}
The motion of the second cell in the co-rotating frame is represented by a new variable $u$ that relates to $z_2$ via $z_2 = ue^{i\omega t}$. The ODE for $u$ is:
\begin{equation}
\label{eq:gamma_u}
    \dot{u} = (\mu + \varepsilon +i\sigma) u - (1+i\gamma)|u|^2u -\lambda\sqrt{\mu}.
\end{equation}

\subsubsection{The reduced system analysis}
Let $\mu > 0$. We consider the case $\mu + \varepsilon > 0$. The case $\mu + \epsilon \le 0$ can be analyzed using the same variable changes as those used for this case when $\gamma = 0$.  
\medskip

The time rescaling, the variable change, and parameter redefinition given by
\begin{equation}
    \label{eq:gamma_redef}
    \tau = \lambda\sqrt{\frac{\mu}{\mu+\varepsilon}},\quad u = v\sqrt{\mu+\varepsilon} ,\quad \tilde{\mu} = \frac{\mu+\varepsilon}{\lambda}\sqrt{\frac{\mu+\varepsilon}{\mu}},\quad\tilde{\sigma} = \frac{\sigma}{\lambda}\sqrt{\frac{\mu+\varepsilon}{\mu}},
\end{equation}
result in the following ODE for the new variable $v$:
\begin{equation}
    \label{eq:gamma_v}
    \dot{v} = \tilde{\mu}\left(1-|v|^2\right)v +i\left(\tilde{\sigma}-\tilde{\mu}\gamma|v|^2\right)v - 1.
\end{equation}
The corresponding system of ODEs for the real and imaginary parts of $v$ is:
\begin{equation}
    \label{eq:gamma_vrvi}
    \begin{aligned}
        v_R& = \tilde{\mu}\left(1-|v|^2\right)v_R -\left(\tilde{\sigma}-\tilde{\mu}\gamma|v|^2\right)v_I - 1,\\
        v_I& = \tilde{\mu}\left(1-|v|^2\right)v_I +\left(\tilde{\sigma}-\tilde{\mu}\gamma|v|^2\right)v_R.
    \end{aligned}
\end{equation}
The equilibria of Eq. \eqref{eq:gamma_vrvi} are the solutions of the following equation for $|v|^2$:
\begin{equation}
    \label{eq:gamma_v2}
    \tilde{\mu}^2\left(1-|v|^2\right)^2|v|^2 +\left(\tilde{\sigma} -\tilde{\mu}\gamma|v|^2\right)^2|v|^2 = 1.
\end{equation}
At any fixed $\tilde{\mu}>0$, $\tilde{\sigma}\in\mathbb{R}$, and $\gamma\in\mathbb{R}$, Eq. \eqref{eq:gamma_v2} has at least one solution $|v|^2$, and each solution of Eq. \eqref{eq:gamma_v2} uniquely determines $v_R$ and $v_I$ from Eq. \eqref{eq:gamma_vrvi}.

Eq. \eqref{eq:gamma_v2} defines a family of slanted ellipses in the $(\tilde{\sigma},\tilde{\mu})$-plane at any fixed $\gamma$. The semiaxes directions and magnitudes are, respectively, the eigenvectors and the reciprocals of the square roots of the eigenvalues of the matrix
\begin{equation}
    \label{eq:gamma_matrix}
    M\left(|v|^2,\gamma\right) = \left[\begin{array}{ccc}
    |v|^2 &\quad& -\gamma |v|^4\\-\gamma |v|^4 &\quad& \left(1-|v|^2\right)^2|v|^2 +\gamma|v|^6
    \end{array}\right].
\end{equation}
As $|v|^2\rightarrow\pm 1$, the ellipses degenerate into two parallel lines $\tilde{\sigma} - \tilde{\mu}\gamma = \pm 1$ in the $(\tilde{\sigma},\tilde{\mu})$-plane. These lines and the family of ellipses are plotted in Fig. \ref{fig:system3_phase_gamma1}.

\medskip
The stability of equilibria of ODE \eqref{eq:gamma_vrvi} is defined by the trace and the determinant of its Jacobian matrix
\begin{equation}
    \label{eq:gamma_J}
    J = \left[\begin{array}{ccc}
    \tilde{\mu}\left(1-|v|^2\right) -2\tilde{\mu}\left(v_R^2-\gamma v_Rv_I\right) &~&
    -\left(\tilde{\sigma} -\tilde{\mu}\gamma|v|^2\right)-2\tilde{\mu}\left(v_Rv_I-\gamma v_I^2\right) \\
    \left(\tilde{\sigma} -\tilde{\mu}\gamma|v|^2\right)-2\tilde{\mu}\left(v_Rv_I+\gamma v_R^2\right)&~&
    \tilde{\mu}\left(1-|v|^2\right) -2\tilde{\mu}\left(v_I^2+\gamma v_Rv_I\right)
  \end{array}\right]:     
\end{equation}
\begin{align}
    \det J &= \tilde{\mu}^2\left(1-|v|^2\right)\left(1-3|v|^2\right) +\left(\tilde{\sigma} -\tilde{\mu}\gamma|v|^2\right)\left(\tilde{\sigma} -3\tilde{\mu}\gamma|v|^2\right),\label{eq:gamma_det}\\
    {\rm tr}\thinspace J & = 2\tilde{\mu}\left(1 - 2|v|^2\right). \label{eq:gamma_trace}
\end{align}
The condition ${\rm tr}\thinspace J = 0$ defines the ellipse plotted in red in Fig. \ref{fig:system3_phase_gamma1}:
\begin{equation}
    \label{eq:ellipse_tr0_gamma1}
    \frac{\tilde{\mu}^2}{8}(1 + \gamma) - \frac{\gamma}{2}\tilde{\sigma}\tilde{\mu} +\frac{\tilde{\sigma}^2}{2} = 1.
\end{equation}
The condition $\det J = 0$ together with Eq. \eqref{eq:gamma_v2} defines the curve plotted in Fig. \ref{fig:system3_phase_gamma1} in magenta. This curve also bounds the region where Eq. \eqref{eq:gamma_vrvi} has three equilibria. The singular points, the cusps, of the $\det J = 0$ curve will be found using the singularity theory approach in the next section.

\begin{figure}[htbp]
\centerline{\includegraphics[width=0.9\textwidth]{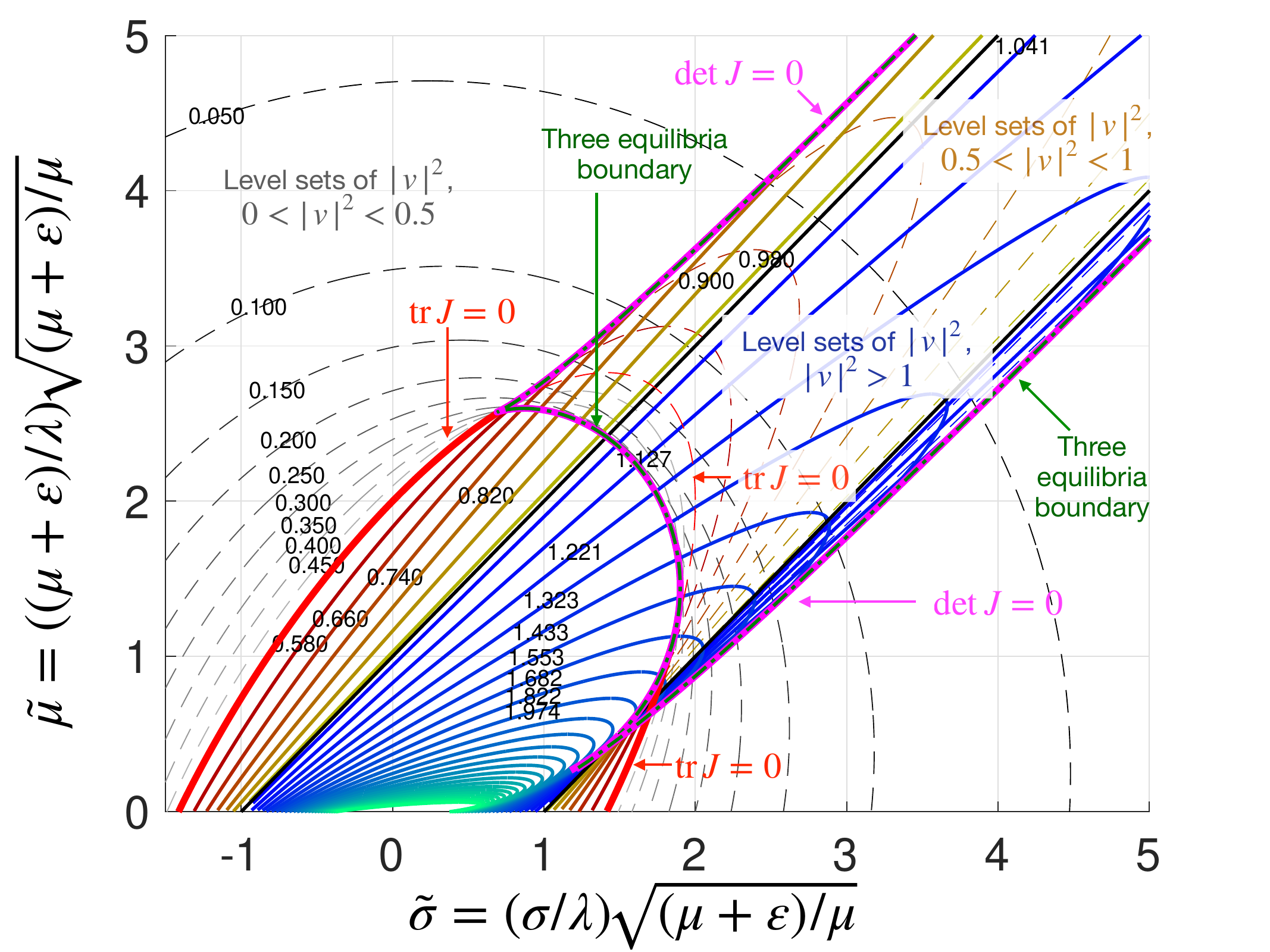}}
 \caption{The phase diagram for system \eqref{eq:system3} at the cubic nonlinearity parameter $\gamma = 1$.
The thick red and magenta curves bound the region spanned by solid ellipses where an asymptotically stable periodic solution to ODE \eqref{eq:system4} exists. The medium-weight solid curves of blue and green shades depict the level sets of $|v|^2 > 1$ corresponding to asymptotically stable periodic solutions. The black slanted lines  $\tilde{\sigma} - \tilde{\mu} = \pm 1$ correspond to stable periodic solutions with $|v|^2 = 1$. The medium-weight solid and thin dashed curves of dark red and dark yellow shades correspond, respectively, to asymptotically stable and unstable periodic solutions with $\tfrac{1}{2} <|v|^2 < 1$. The dashed curves of grey tones depict ellipses with $0 < |v|^2 < \tfrac{1}{2}$.  The black numbers next to the curves indicate the corresponding values of $|v|^2$. 
}
\label{fig:system3_phase_gamma1}
\vspace*{12pt}
\end{figure}

\subsection{Singularity theory approach}
\label{sec:system4_singularity}
In this section, we apply the singularity theory approach~\cite{GS1} to analyze the solutions to the equation
\begin{equation}
    \label{eq:gamma_u2}
    \left(\mu +\varepsilon-|u|^2\right)^2|u|^2 +\left(\sigma - \gamma |u|^2\right)^2|u|^2 = \lambda^2\mu
\end{equation}
for the squared amplitude, $|u|^2$, of the second cell at a periodic solution to system \eqref{eq:system4}. Eq. \eqref{eq:gamma_u2} is obtained from Eq. \eqref{eq:gamma_u} by splitting it into a system of ODEs for the real and complex parts of $u$, $u = u_R + iu_I$, setting their right-hand sides to zero, and taking $\lambda\sqrt{\mu}$ to the left-hand side and squaring. Expanding Eq. \eqref{eq:gamma_u2} and grouping terms according to the powers of $x: = |u|^2$, we obtain the following parametric cubic equation
\begin{equation}
    \label{eq:system4_prelim}
     G(x,\mu,\sigma,\varepsilon,\lambda,\gamma): = x^3(1 + \gamma^2) - 2x^2(\mu +\varepsilon + \sigma\gamma) + x((\mu +\varepsilon)^2 + \sigma^2) - \lambda^2\mu = 0.
\end{equation}
\medskip

The singularity theory approach \cite{GS1} consists of several steps. The first step is the so-called \emph{recognition problem}, whose task is to select one parameter to be the bifurcation parameter, set the remaining parameters to zero, and identify which parametric family, i.e., which normal form, the resulting equation belongs to. The second step is to check whether the original equation, with all parameters nonzero, is a \emph{versal unfolding} of the normal form identified in the first step. The adjective ``versal" means that the parametric family contains all possible one-parameter perturbations of the normal form, thereby encompassing all perturbed bifurcation diagrams. The versal unfolding is universal if it requires the fewest additional parameters beyond the chosen bifurcation parameter to achieve it. In other words, if the number of the additional parameters is equal to the codimension of the tangent space of the normal form. 
The third step is to identify the loci of singularities, i.e., parameter values, at which the bifurcation diagram undergoes a qualitative change. 

Eq. \eqref{eq:system4_prelim} is cubic in $x$. Cubic equations admit two types of singularities, \emph{hysteresis} and \emph{bifurcation}, illustrated in Fig. \ref{fig:HBschematic}.  The loci of singularities partition the parameter space into regions in each of which the bifurcation diagram remains qualitatively the same. Therefore, it suffices to plot one representative bifurcation diagram for each region of the unfolding parameter space, partitioned by the singularity hypersurfaces. 
\begin{figure}[htbp]
\centerline{
(a) \includegraphics[width=0.4\textwidth]{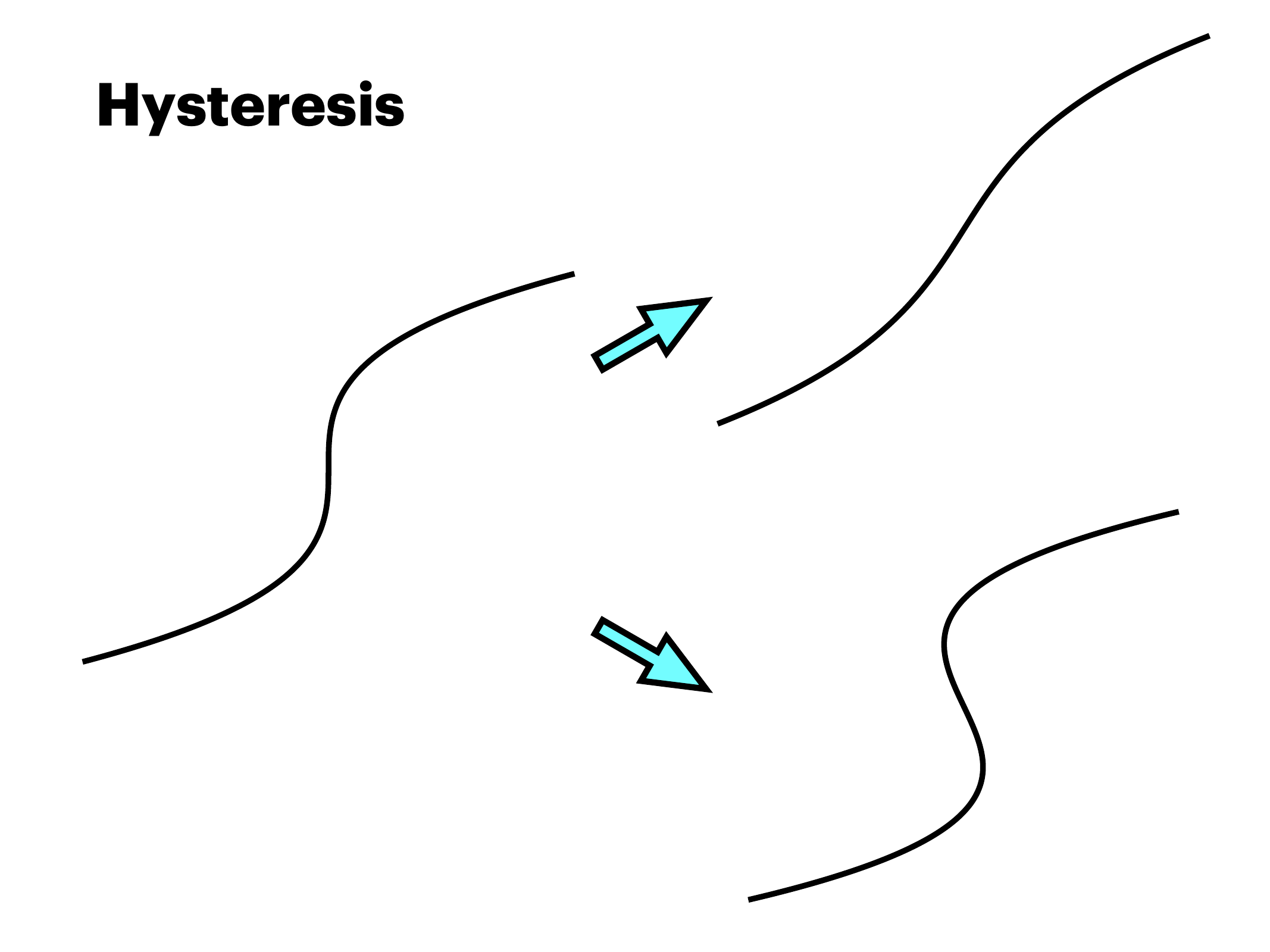}
(b) \includegraphics[width=0.4\textwidth]{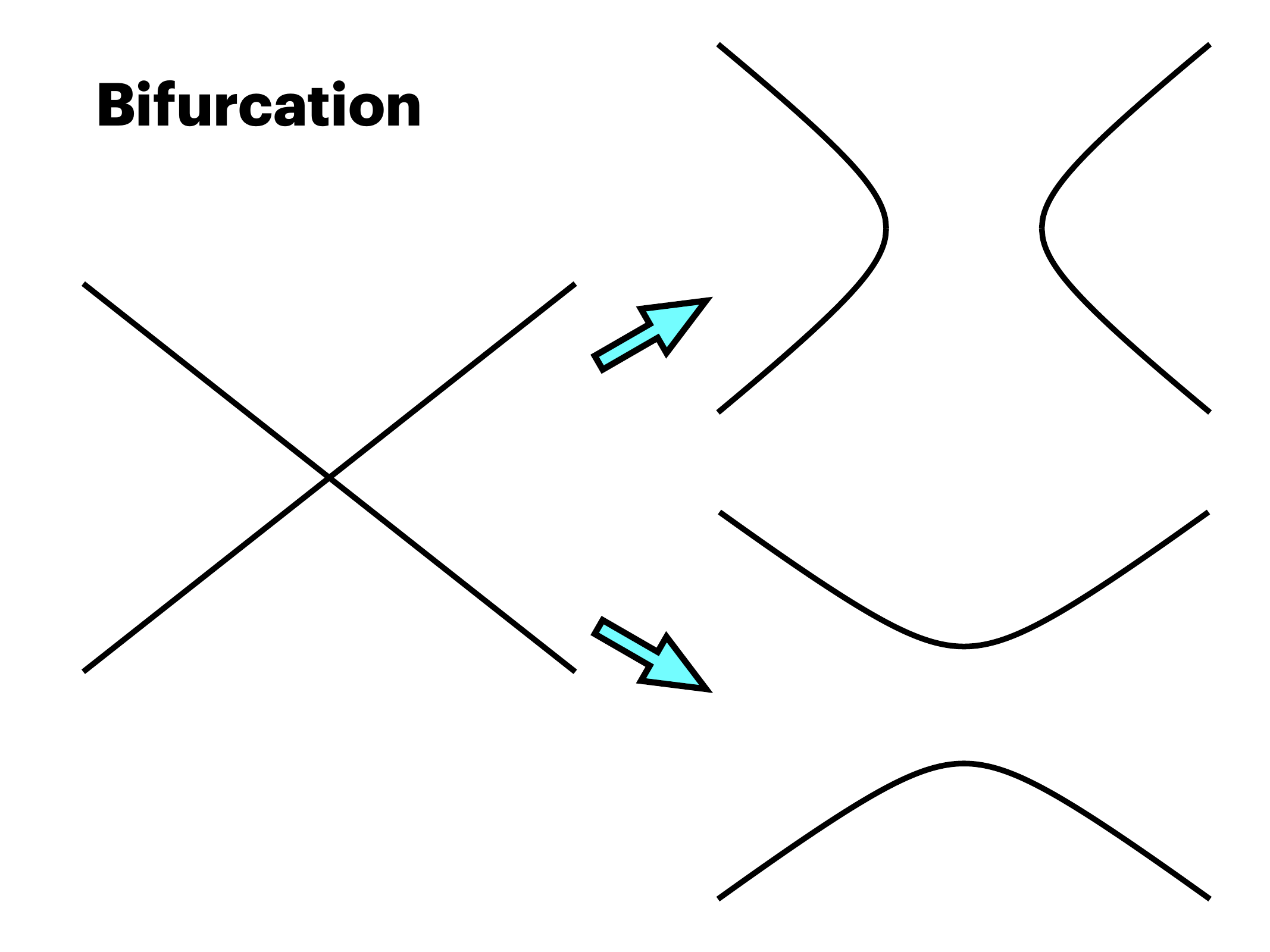}
}
 \caption{Nonpersistent diagrams of hysteresis and bifurcation types. (a): A bifurcation diagram of hysteresis type on the left transforms into a branch with a unique solution on the top right, or with a three-solution interval on the bottom right, as a result of an arbitrarily small perturbation. (b): A bifurcation diagram of bifurcation type on the left transforms into one of the diagrams on the right, as a result of an arbitrarily small perturbation. 
}
\label{fig:HBschematic}
\vspace*{12pt}
\end{figure}

\medskip
Examining Eq. \eqref{eq:system4_prelim}, we choose $\sigma$ as a bifurcation parameter and observe that it can be viewed as an unfolding of the following $\mathbb{Z}_2$-symmetric normal form (see Table 5.1 on page 263 in Section VI.5 in \cite{GS1}):
\begin{equation}
    \label{eq:Z2form}
    g(x,\sigma): = x^3 + \sigma^2x.
\end{equation}
 Its $\mathbb{Z}_2$-codimension is 1, while its full co-dimension is 5. Eq. \eqref{eq:system4_prelim} has four additional parameters, $\mu$, $\varepsilon$, $\lambda$, and $\gamma$. The parameter $\mu$ is the excitation parameter of the first cell in Eq. \eqref{eq:system4}. It is positive by assumption. It can be eliminated by redefining $\lambda$ as $\lambda\sqrt{\mu}$ and $\varepsilon$ as $\mu+\varepsilon$. We choose not to do it. Instead, we fix its value and treat it as a constant.
 
Thus, Eq. \eqref{eq:system4_prelim} can be viewed as a three-parameter unfolding of Eq. \eqref{eq:Z2form}. It is neither a $\mathbb{Z}_2$-unfolding, nor the full universal unfolding of $g(x,\sigma)$. Next, we proceed to finding the loci of nonpersistent phenomena: hysteresis, $\mathcal{H}$, and bifurcation, $\mathcal{B}$, singularities. The conditions for these singularities are specified in Definition 5.1 in Section III.5 in \cite{GS1}.

\medskip
The conditions for the hysteresis singularity are $G = G_x = G_{xx} = 0$. This gives the following system:
\begin{subequations} \label{eq:G_per_sol}
     \begin{align}
	& G = x^3(1+\gamma^2) - 2(\mu + \varepsilon+\sigma\gamma) x^2 + [(\mu + \varepsilon)^2 + \sigma^2]x - \lambda^2 \mu = 0  \label{eq:G_per} \\[3pt]
	& G_x = 3 x^2 (1+\gamma^2)- 4(\mu + \varepsilon+\sigma\gamma)x + (\mu + \varepsilon)^2 + \sigma^2 = 0  \label{eq:G_x} \\[3pt]
	& G_{xx} = 6x(1+\gamma^2) - 4(\mu + \varepsilon+\sigma\gamma) = 0.  \label{eq:G_xx}
    \end{align}
\end{subequations}
System \eqref{eq:G_per}, \eqref{eq:G_x}, \eqref{eq:G_xx} can be solved for $\varepsilon$, $\sigma$, and $|u|^2\equiv x$ in terms of $\lambda$, $\mu$, and $\gamma$, defining the hysteresis set as (see Appendix \ref{appC}):
\begin{equation}
    \label{eq:sys4_H_solution}
   \mathcal{H}=\left\{ \varepsilon = \frac{3}{2}\frac{1\pm\tfrac{1}{\sqrt{3}}\gamma}{(1+\gamma^2)^{1/3}}\lambda^{2/3}\mu^{1/3} - \mu,\quad \sigma = \frac{3}{2}\frac{\gamma\mp\tfrac{1}{\sqrt{3}}}{(1+\gamma^2)^{1/3}}\lambda^{2/3}\mu^{1/3},\quad
    |u|^2 = \frac{\lambda^{2/3}\mu^{1/3}}{(1+\gamma^2)^{1/3}}\right\}.
\end{equation}

\medskip

The bifurcation conditions are $G = G_x = G_{\sigma} = 0$, i.e.,
\begin{subequations} 
\label{eq:G_per_sol_bifur}
     \begin{align}
	& G = x^3(1+\gamma^2) - 2(\mu + \varepsilon+\sigma\gamma) x^2 + [(\mu + \varepsilon)^2 + \sigma^2]x - \lambda^2 \mu = 0  \label{eq:G_bifur} \\[3pt]
	& G_x = 3 x^2 (1+\gamma^2)- 4(\mu + \varepsilon+\sigma\gamma)x + (\mu + \varepsilon)^2 + \sigma^2 = 0  \label{eq:G_x_bifur} \\[3pt]
	& G_{\sigma} = -2\gamma x^2 + 2\sigma x = 0.  \label{eq:G_sig}
    \end{align}
\end{subequations}
From Eq. \eqref{eq:G_sig}, we find $\gamma = \sigma/x$. Plugging it into Eq. \eqref{eq:G_x_bifur} yields
\begin{equation}
    \label{eq:Bif2}
    (x - (\mu +\varepsilon))(3x-(\mu+\varepsilon)) = 0,\quad {\rm i.e.}\quad x = \mu + \varepsilon\quad{\rm or}\quad 3x = \mu + \varepsilon.
\end{equation}
The conditions $x = \mu + \varepsilon$, $\gamma = \sigma/x$ and Eq. \eqref{eq:G_bifur} yield $\lambda^2\mu = 0$, which implies $\lambda = 0$, since we have assumed that $\mu > 0$. The conditions $3x = \mu+\varepsilon$, $\gamma = \sigma/x$ and Eq. \eqref{eq:G_bifur} yield $ 4x^3 = \lambda^2\mu$.
Therefore, the bifurcation set is
\begin{equation}
\label{eq:Bif3}
\mathcal{B} =
\left[
\begin{aligned}
\lambda &= 0, 
&\sigma &= \gamma(\mu + \varepsilon),
&|u|^2 &= \mu + \varepsilon, \\[4pt]
\frac{4}{27}(\mu + \varepsilon)^3 &= \lambda^2 \mu,
&\sigma &= \gamma \frac{\mu + \varepsilon}{3},
&|u|^2 &= \frac{\mu + \varepsilon}{3}.
\end{aligned}
\right.
\end{equation}

\medskip
Eqs. \eqref{eq:sys4_H_solution} and \eqref{eq:Bif3} defining the hysteresis, $\mathcal{H}$, and bifurcation, $\mathcal{B}$, sets  allow us to plot slices of $\mathcal{H}$ and $\mathcal{B}$ in the $(\varepsilon,\lambda)$-plane at any given $\gamma$. 
Fig. \ref{fig:system4_BH} displays these slices at $\mu = 0.2$ and $\gamma \in\{0,\tfrac{1}{3},1.5,10\}$. 

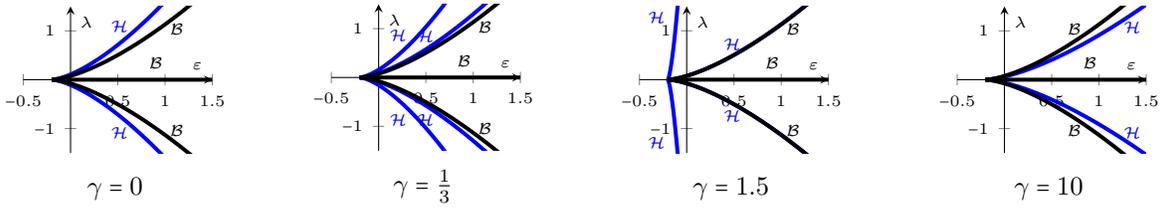
\begin{figure}[h]
\centering
\begin{minipage}{.248\textwidth}
  \centering
  \begin{tikzpicture}
  \begin{axis}[
      width=\textwidth,
      xmin=-0.5, xmax=1.5,
      ymin=-1.5, ymax=1.5,
      axis lines=middle,
      xlabel={$ \varepsilon $},
      ylabel={$ \lambda $},
      domain=-0.2:1.5,
      restrict x to domain=-0.2:1.5,
      legend style={draw=none,font=\tiny},
      every axis/.append style={font=\tiny}
  ]

  \addplot[line width=1.5pt, blue, samples=200]
    {(0.2+x)^(3/2)*(0.2)^(-1/2)*(2/3)^(3/2)}
    node[pos=0.45, left, font=\tiny] {$\mathcal{H}$};
  \addplot[line width=1.5pt, blue, samples=200]
    {-(0.2+x)^(3/2)*(0.2)^(-1/2)*(2/3)^(3/2)}
    node[pos=0.45, left, font=\tiny] {$\mathcal{H}$};

  \addplot[line width=1.5pt, black, samples=2] {0}node[pos=0.65, above, font=\tiny] {$\mathcal{B}$};

  \addplot[line width=1.5pt, black, samples=200]
    {sqrt((x+0.2)^3/1.35)}
    node[pos=0.60, right, font=\tiny] {$\mathcal{B}$};
  \addplot[line width=1.5pt, black, samples=200]
    {-sqrt((x+0.2)^3/1.35)}
    node[pos=0.60, right, font=\tiny] {$\mathcal{B}$};

  \end{axis}
  \end{tikzpicture}

  \vspace{2pt}
  {\small $\gamma = 0$}
\end{minipage}%
\hfill
\begin{minipage}{.248\textwidth}
  \centering
  \begin{tikzpicture}
  \begin{axis}[
      width=\textwidth,
      xmin=-0.5, xmax=1.5,
      ymin=-1.5, ymax=1.5,
      axis lines=middle,
      xlabel={$ \varepsilon $},
      ylabel={$ \lambda $},
      domain=-0.2:1.5,
      restrict x to domain=-0.2:1.5,
      legend style={draw=none,font=\tiny},
      every axis/.append style={font=\tiny}
  ]

  \addplot[line width=1.5pt, blue, samples=200]
    {(2*(0.2+x)/(3*(1+1/3^(3/2))))^(3/2)*((1+1/3^2)/0.2)^(1/2)}
    node[pos=0.45, left, font=\tiny] {$\mathcal{H}$};
  \addplot[line width=1.5pt, blue, samples=200]
    {(2*(0.2+x)/(3*(1-1/3^(3/2))))^(3/2)*((1+1/3^2)/0.2)^(1/2)}
    node[pos=0.25, left, font=\tiny] {$\mathcal{H}$};

  \addplot[line width=1.5pt, blue, samples=200]
    {-(2*(0.2+x)/(3*(1+1/3^(3/2))))^(3/2)*((1+1/3^2)/0.2)^(1/2)}
    node[pos=0.45,left, font=\tiny] {$\mathcal{H}$};
  \addplot[line width=1.5pt, blue, samples=200]
    {-(2*(0.2+x)/(3*(1-1/3^(3/2))))^(3/2)*((1+1/3^2)/0.2)^(1/2)}
    node[pos=0.25, left, font=\tiny] {$\mathcal{H}$};

  \addplot[line width=1.5pt, black, samples=2] {0}node[pos=0.65, above, font=\tiny] {$\mathcal{B}$};

  \addplot[line width=1.5pt, black, samples=200]
    {sqrt((x+0.2)^3/1.35)}
    node[pos=0.60, right, font=\tiny] {$\mathcal{B}$};
  \addplot[line width=1.5pt, black, samples=200]
    {-sqrt((x+0.2)^3/1.35)}
    node[pos=0.60, right, font=\tiny] {$\mathcal{B}$};

  \end{axis}
  \end{tikzpicture}

  \vspace{2pt}
  {\small $\gamma = \tfrac{1}{3}$}
\end{minipage}%
\hfill
\begin{minipage}{.248\textwidth}
  \centering
  \begin{tikzpicture}
  \begin{axis}[
      width=\textwidth,
      xmin=-0.5, xmax=1.5,
      ymin=-1.5, ymax=1.5,
      axis lines=middle,
      xlabel={$ \varepsilon $},
      ylabel={$ \lambda $},
      domain=-0.2:1.5,
      restrict x to domain=-0.2:1.5,
      legend style={draw=none,font=\tiny},
      every axis/.append style={font=\tiny}
  ]

  \addplot[line width=1.5pt, blue, samples=200]
    {(2*(0.2+x)/(3*(1+(1.5)/3^(1/2))))^(3/2)*((1+(1.5)^2)/0.2)^(1/2)}
    node[pos=0.45, left, font=\tiny] {$\mathcal{H}$};
  \addplot[line width=1.5pt, blue, samples=200]
    {(2*(0.2+x)/(3*(1-(1.5)/3^(1/2))))^(3/2)*((1+(1.5)^2)/0.2)^(1/2)}
    node[pos=0.013, left, font=\tiny] {$\mathcal{H}$};
  \addplot[line width=1.5pt, blue, samples=200]
    {-(2*(0.2+x)/(3*(1+(1.5)/3^(1/2))))^(3/2)*((1+(1.5)^2)/0.2)^(1/2)}
    node[pos=0.45, left, font=\tiny] {$\mathcal{H}$};
  \addplot[line width=1.5pt, blue, samples=200]
    {-(2*(0.2+x)/(3*(1-(1.5)/3^(1/2))))^(3/2)*((1+(1.5)^2)/0.2)^(1/2)}
    node[pos=0.013, left, font=\tiny] {$\mathcal{H}$};

  \addplot[line width=1.5pt, black, samples=2] {0}node[pos=0.65, above, font=\tiny] {$\mathcal{B}$};

  \addplot[line width=1.5pt, black, samples=200]
    {sqrt((x+0.2)^3/1.35)}
    node[pos=0.60, right, font=\tiny] {$\mathcal{B}$};
  \addplot[line width=1.5pt, black, samples=200]
    {-sqrt((x+0.2)^3/1.35)}
    node[pos=0.60, right, font=\tiny] {$\mathcal{B}$};

  \end{axis}
  \end{tikzpicture}

  \vspace{2pt}
  {\small $\gamma = 1.5$}
\end{minipage}
\hfill
\begin{minipage}{.248\textwidth}
  \centering
  \begin{tikzpicture}
  \begin{axis}[
      width=\textwidth,
      xmin=-0.5, xmax=1.5,
      ymin=-1.5, ymax=1.5,
      axis lines=middle,
      xlabel={$ \varepsilon $},
      ylabel={$ \lambda $},
      domain=-0.2:1.5,
      restrict x to domain=-0.2:1.5,
      legend style={draw=none,font=\tiny},
      every axis/.append style={font=\tiny}
  ]

  \addplot[line width=1.5pt, blue, samples=200]
    {(2*(0.2+x)/(3*(1+(10)/3^(1/2))))^(3/2)*((1+(10)^2)/0.2)^(1/2)}
    node[pos=0.75, right, font=\tiny] {$\mathcal{H}$};
  \addplot[line width=1.5pt, blue, samples=200]
    {-(2*(0.2+x)/(3*(1+(10)/3^(1/2))))^(3/2)*((1+(10)^2)/0.2)^(1/2)}
    node[pos=0.75, right, font=\tiny] {$\mathcal{H}$};

  \addplot[line width=1.5pt, black, samples=2] {0}node[pos=0.65, above, font=\tiny] {$\mathcal{B}$};

  \addplot[line width=1.5pt, black, samples=200]
    {sqrt((x+0.2)^3/1.35)}
    node[pos=0.60, left, font=\tiny] {$\mathcal{B}$};
  \addplot[line width=1.5pt, black, samples=200]
    {-sqrt((x+0.2)^3/1.35)}
    node[pos=0.60, left, font=\tiny] {$\mathcal{B}$};

  \end{axis}
  \end{tikzpicture}

  \vspace{2pt}
  {\small $\gamma = 10$}
\end{minipage}

\caption{The hysteresis, $\mathcal{H}$, and bifurcation, $\mathcal{B}$ sets for Eq. \eqref{eq:system4_prelim} at $\mu = 0.2$ and selected values of $\gamma$.} 
\label{fig:system4_BH}
\end{figure}

\medskip

Fig. \ref{fig:system4_BH} and Eqs. \eqref{eq:sys4_H_solution} and \eqref{eq:Bif3} suggest the following observations:
\begin{itemize}
\item The slice structure is symmetric with respect to the map $(\sigma,\gamma)\mapsto(-\sigma,-\gamma)$ and $\lambda\mapsto-\lambda$.
\item The $\mathcal{B}$-slices are independent of $\gamma$.
    \item At $\gamma = 0$, the two $\mathcal{H}$-slices in the upper half-place $\lambda > 0$ coincide.
    \item As $\gamma\rightarrow \sqrt{3}$, the $\mathcal{H}$-slices with the ``+" sign in front of $\tfrac{1}{\sqrt{3}}\gamma$ approach the $\mathcal{B}$-slices:
\begin{equation}
    \lambda = \pm\left(\frac{2(\mu+\varepsilon)}{3\left[1 +\tfrac{1}{\sqrt{3}}\gamma\right]}\right)^{3/2}\left(\frac{1+\gamma^2}{\mu}\right)^{1/2} ~\rightarrow ~\pm\frac{2(\mu+\varepsilon)^{3/2}}{(27\mu)^{1/2}}.
\end{equation}
On the other hand, the $\mathcal{H}$-slices with the ``$-$" sign in front of $\tfrac{1}{\sqrt{3}}\gamma$,
\begin{equation}
    \lambda = \pm\left(\frac{2(\mu+\varepsilon)}{3\left[1 -\tfrac{1}{\sqrt{3}}\gamma\right]}\right)^{3/2}\left(\frac{1+\gamma^2}{\mu}\right)^{1/2},
\end{equation}
approach vertical line $\epsilon=-\mu$ as $\gamma\rightarrow\sqrt{3}$.
\item At $\gamma > \sqrt{3}$, only the $\mathcal{H}$-slices with the ``+" sign in front of $\tfrac{1}{\sqrt{3}}\gamma$ remain, while the $\mathcal{H}$-slices with the ``$-$" sign in front of $\tfrac{1}{\sqrt{3}}\gamma$ disappear.
\end{itemize}

Fig. \ref{fig:system4_singularity} displays a comprehensive collection of phase and bifurcation diagrams of the unfolding in Eq. \eqref{eq:system4_prelim}. The bottom left subfigure shows the hysteresis surfaces located in the upper half-space $\lambda > 0$ plotted in blue and the upper bifurcation surface plotted in black. The other three subfigures show slices of these surfaces by the planes $\gamma = 1$, $\gamma = 0$, and $\gamma = -1$, along with representative and special bifurcation diagrams of $|u|^2$ versus $\sigma$. Solid green lines represent asymptotically stable intervals, while dashed blue lines depict unstable ones. The stability of the branches was checked using the determinant and trace stability conditions found in Appendix \ref{appE}.
Note that the local changes in the bifurcation diagrams $|u|^2$ versus $\sigma$ at the hysteresis and bifurcation singularities match those depicted in Fig. \ref{fig:HBschematic}.

\begin{figure}[htbp]
\centerline{\includegraphics[width=0.9\textwidth]{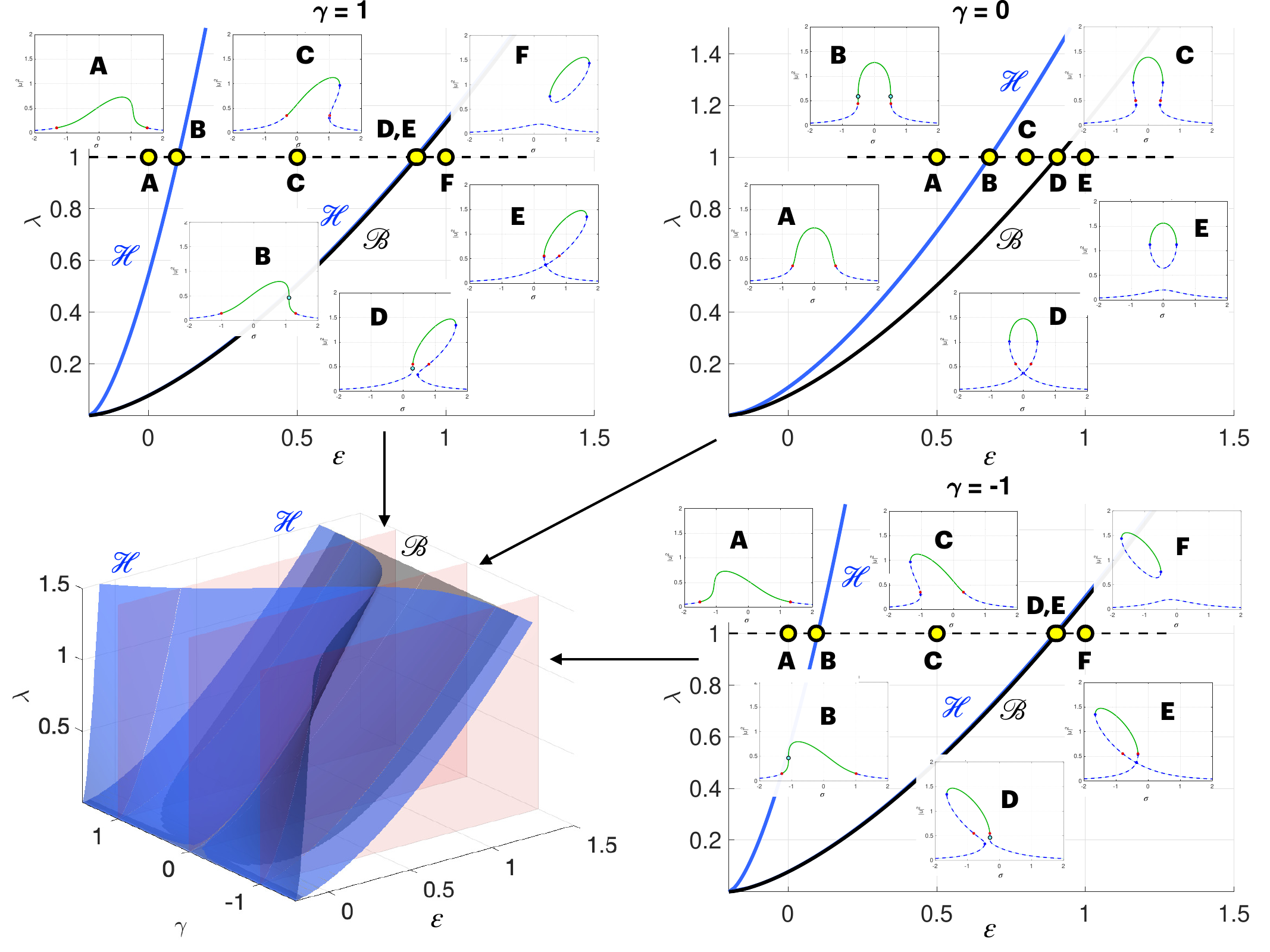}}
 \caption{The results of the singularity theory approach applied to system \eqref{eq:system4}. We treat $\sigma$ as the bifurcation parameters, while $\lambda$, $\varepsilon$, and $\gamma$ as the unfolding parameters in Eq. \eqref{eq:system4_prelim}. The parameter $\mu$ is fixed at $\mu = 0.2$. The 3D figure in the bottom left displays the hysteresis, $\mathcal{H}$ (blue), and the bifurcation, $\mathcal{B}$ (black), surfaces in the unfolding parameters $(\varepsilon,\gamma,\lambda)$-space. The pink planes correspond to the selected values of $\gamma$: $\gamma = 1$, $\gamma = 0$, and $\gamma = -1$. The sections of the phase diagrams by these planes are shown in the top left, top right, and bottom right, respectively. The blue curves are the hysteresis curves, while the black ones are the bifurcation curves. Note that the black and one of the blue curves are very close at $\gamma = \pm 1$. The insets show the plots of $x\equiv|u|^2$ versus $\sigma$ as fixed $\lambda = 1$, $\gamma\in\{1,0,-1\}$, and $\varepsilon$ marked by the yellow dots with black edges. The solid pieces correspond to asymptotically stable periodic solutions of system \eqref{eq:system4}, while the dashed ones to unstable periodic solutions.  The hysteresis points, marked by cyan dots with black edges, are those with the vertical tangent.  The curves $|u|^2$ versus $\sigma$ split into two connected components at the bifurcation points.
}
\label{fig:system4_singularity}
\vspace*{12pt}
\end{figure}

\subsubsection{Reconciling with the phase diagrams for the reduced system}
Fig. \ref{fig:system4_HBpts} relates the hysteresis and bifurcation points given by Eqs. \eqref{eq:sys4_H_solution} and \eqref{eq:Bif3}, respectively, to the reduced system phase diagrams in the $(\tilde{\sigma},\tilde{\mu})$-plane in Figs. \ref{fig:system2_phase} and \ref{fig:system3_phase_gamma1}. 
The hysteresis points in Eq. \eqref{eq:sys4_H_solution} rewritten in terms of $\tilde{\sigma} = \tfrac{\sigma}{\lambda}\sqrt{\tfrac{\mu+\varepsilon}{\mu}}$, $\tilde{\mu}=\tfrac{\mu+\varepsilon}{\lambda}\sqrt{\tfrac{\mu+\varepsilon}{\mu}}$, and $|v|^2 = |u|^2/(\mu + \varepsilon)$ are:
\begin{equation}
\label{eq:Hreduced}
    \tilde{\mu} = \frac{3\sqrt{3}}{2\sqrt{2}} \frac{(1\pm\tfrac{\gamma}{\sqrt{3}})^{3/2}}{(1+\gamma^2)^{1/2}},\quad 
    \tilde{\sigma} =  \frac{3\sqrt{3}}{2\sqrt{2}}\frac{(\gamma\mp\tfrac{1}{\sqrt{3}})(1\pm\tfrac{\gamma}{\sqrt{3}})^{1/2}}{(1+\gamma^2)^{1/2}},\quad
    |v|^2 = \frac{2}{3}\left(1\pm\tfrac{\gamma}{\sqrt{3}}\right)^{-1}.
\end{equation}
One can check that Eq. \eqref{eq:Hreduced} with $\gamma = 0$ matches the singularity described in Statement 4 of Proposition \ref{prop1}:
\begin{equation}
    \label{eq:prop1_singularity}
    \tilde{\sigma}\equiv\frac{\sigma}{\lambda} = \pm\frac{3}{2\sqrt{2}},\quad \tilde{\mu}\equiv \frac{\mu}{\lambda} = \frac{3\sqrt{3}}{2\sqrt{2}},\quad |v|^2\equiv \frac{x}{\mu} = \frac{2}{3}.
\end{equation}
 The bifurcation locus with the condition $\tfrac{4}{27}(\mu + \varepsilon)^3 = \lambda^2\mu$ is equivalent to $\tilde{\mu} = \tfrac{3\sqrt{3}}{2}$, where $\tilde{\mu} = \tfrac{1}{\lambda\mu^{1/2}}(\mu+\varepsilon)^{3/2}$ is defined in Eqs. \eqref{eq:redef_mu_sigma1} and \eqref{eq:gamma_redef}. The corresponding $|v|^2 = |u|^2/(\mu + \varepsilon) = \tfrac{1}{3}$, and $\tilde{\sigma} = \tfrac{1}{3}\gamma\tilde{\mu}$.  The bifurcation points $(\tfrac{3\sqrt{3}}{2},\gamma\tfrac{\sqrt{3}}{2})$ in the diagram in Figs. \ref{fig:system2_phase} and \ref{fig:system3_phase_gamma1} are the highest points on the segments of the dash-dotted green curves between their two cusps corresponding to the hysteresis points.

\begin{figure}[htbp]
\centerline{\includegraphics[width=0.9\textwidth]{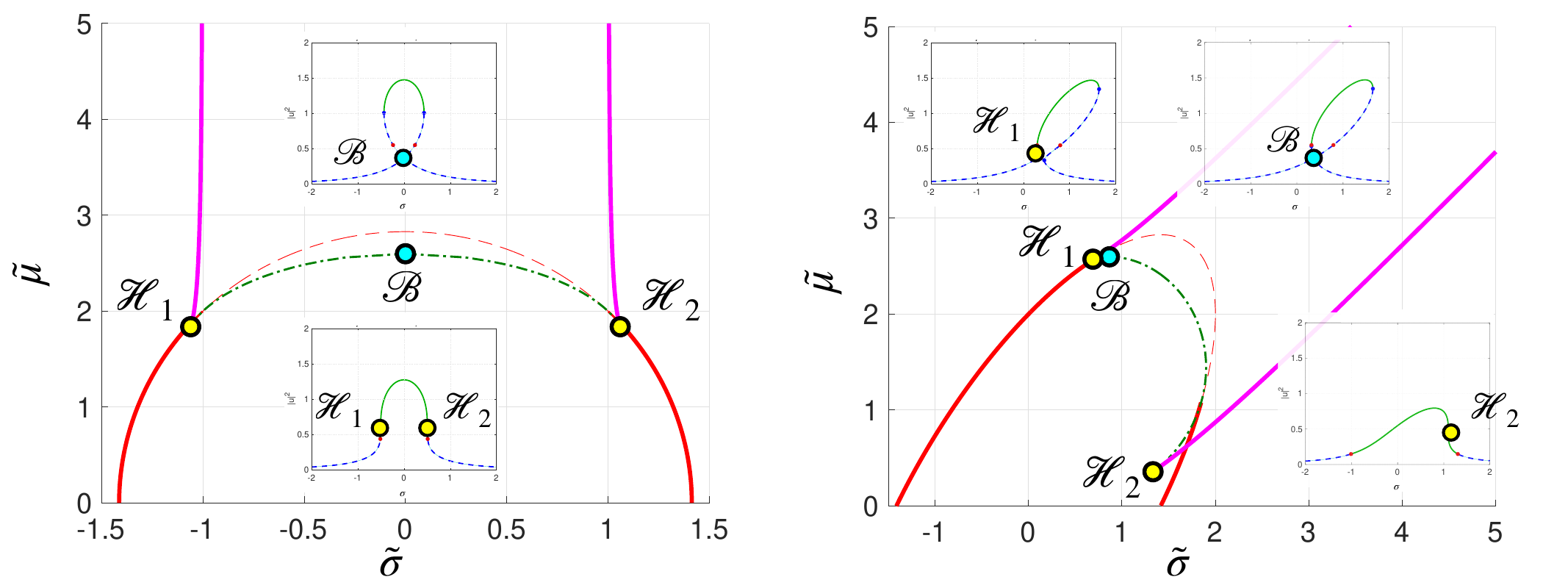}}
 \caption{The hysteresis and bifurcation points on the phase diagram for the reduced system \eqref{eq:gamma_v}. The hysteresis points corresponds to the singularities of the curve $\det J = 0$, while the bifurcation points are the highest points of this curve between its singularities. Left: $\gamma = 0$. Right: $\gamma = 1$.
}
\label{fig:system4_HBpts}
\vspace*{12pt}
\end{figure}


\subsection{Response amplification in two-cell Stuart-Landau feedforward networks}
We have investigated two-cell feedforward networks of inhomogeneous Stuart-Landau oscillators. We now summarize the effects of parameter inhomogeneities on the amplitude of the second cell. 
\begin{figure}[htbp]
\centerline{\includegraphics[width=0.9\textwidth]{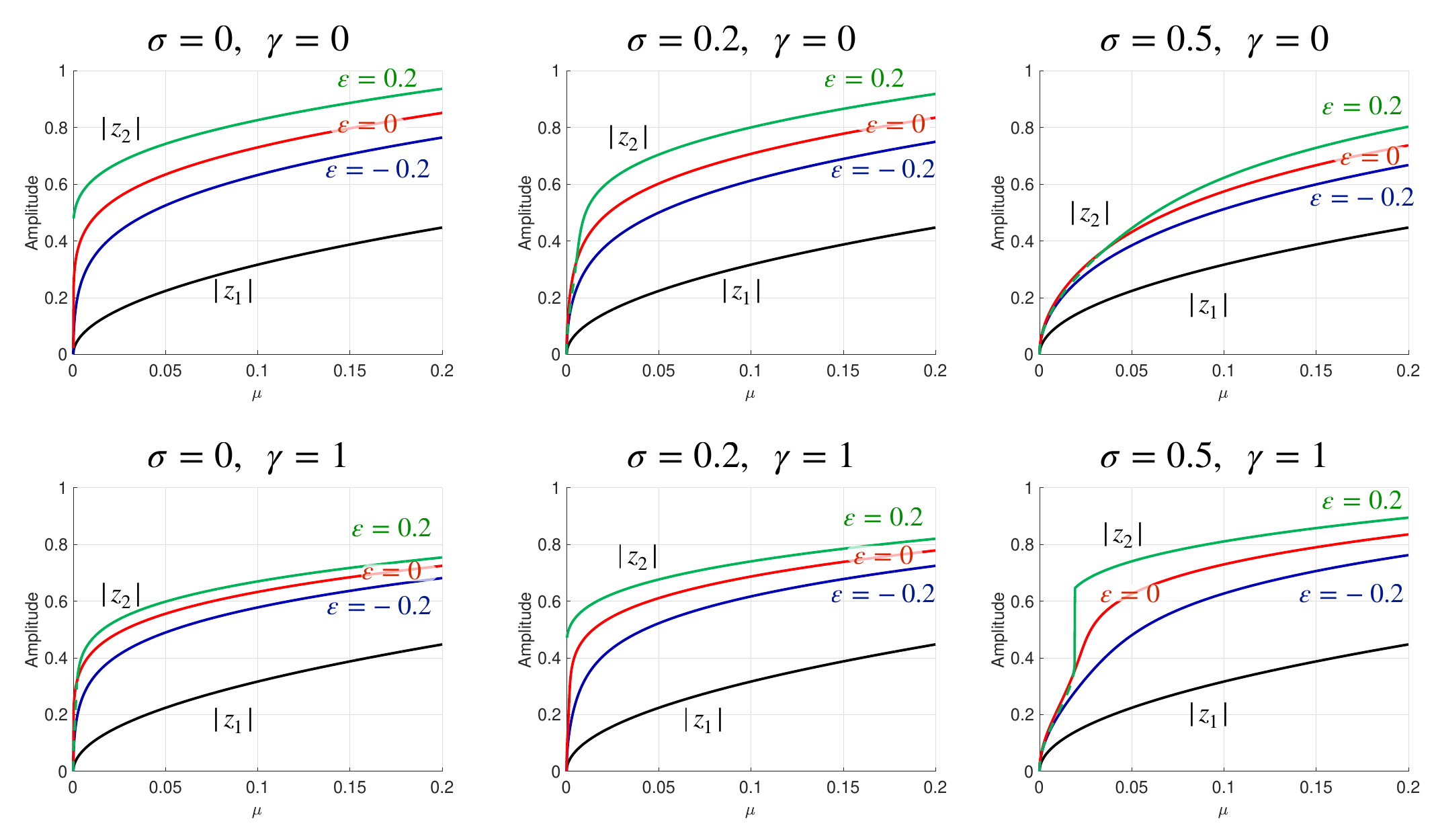}}
 \caption{The amplitudes of the cells $z_1$ and $z_2$ of the periodic solutions of system \eqref{eq:system4} as the functions of $\mu$ at $\sigma\in\{0,0.2,0.5\} $, $\varepsilon\in\{-0.2,0,0.2\}$, $\gamma\in\{0,1\}$, and $\lambda = 1$. The black curves depict the amplitude of the first cell, $|z_1|$, while the navy blue, red, and green curves correspond to the amplitude of the second cell, $|z_2|\equiv |u|$ at $\varepsilon = -0.2$, $\varepsilon =0$, and $\varepsilon =0.2$, respectively. The solid lines correspond to asymptotically stable periodic solutions, while the dashed lines represent unstable ones. The green curves representing the case $\varepsilon = 0.2$ contain unstable intervals in some cases.
}
\label{fig:system4_amplitude}
\vspace*{12pt}
\end{figure}

\begin{enumerate}
    \item {\bf The effect of inhomogeneity in frequency.} The key finding is that the inhomogeneity that fundamentally affects the amplitude of the second cell is the inhomogeneity in frequency, represented by $\sigma$. In particular, if $\gamma = 0$, then the amplitude $|u|$ of the second cell reaches its maximum at $\sigma = 0$. If $\gamma > 0$, the same maximum value of $|u|$ is reached at $\sigma = |u|^2_{\max}\gamma$ -- see the insets in Fig. \ref{fig:system4_singularity}. 
    
     \item {\bf The range of the phase-locked solution.} Figures \ref{fig:system2_phase}, \ref{fig:system3_ellipses}, and \ref{fig:system3_phase_gamma1} show that the phase-locked attractors persist for a broad range of inhomogeneity in frequency. The region where the phase-locked attractor exists includes the strip
    \begin{equation}
    \label{eq:strip_phase_lock}
    \left|\sigma - \gamma(\mu+\varepsilon)\right|\le \lambda\sqrt{\frac{\mu}{\mu+\varepsilon}}.
    \end{equation}
    As $\mu+\varepsilon\rightarrow 0+$, the interval where the phase-locked attractor exists approaches 
    \begin{equation}
    \label{eq:phase_locked_range}
    |\sigma| < \lambda\sqrt{\frac{\mu + \varepsilon}{\mu}}\sqrt{2}.
    \end{equation}
   
    \item {\bf The rapid growth of amplitude at small $\boldsymbol{\mu}$.} Eq. \eqref{eq:gamma_u2} shows that, if $\sigma = 0$, the amplitude $|u|$ of the second cell 
    is approximately equal to $\lambda^{1/3}\mu^{1/6}(1 + \gamma^2)^{-1/3}$ at small $\mu$. This is consistent with the result of \cite{Levasseur_Palacios_2021}.
    If $\sigma \neq 0$, then $|u|\le \tfrac{\lambda^2}{\sigma^2}\sqrt{\mu}$. This rapid amplitude growth is evident in Fig. \ref{fig:system4_amplitude}.

    \item {\bf The effect of inhomogeneity in the excitation parameter.} If $\varepsilon < 0$, the periodic solution of system \eqref{eq:system4} is asymptotically stable. The amplitude at the second cell is smaller than its corresponding amplitude at $\varepsilon = 0$ -- compare the navy blue and red curves in Fig. \ref{fig:system4_amplitude}. 
    In contrast, if $\varepsilon > 0$, the amplitude of the second cell is larger than at $\varepsilon = 0$; however, the periodic solution is unstable at small $\mu$ at $\gamma = 0$ and may be unstable at $\gamma \neq 0$, depending on $\sigma$. This is evident from Fig. \ref{fig:system4_amplitude}, where the green curves tend to lie above the red ones, and may have a dashed interval indicating instability of the corresponding solution.

    \item {\bf The effect of inhomogeneity in the cubic nonlinearity parameter.} The nonzero $\gamma$ does not affect the amplitude of the second cell. It merely distorts the graph of $|u|^2$ versus $\sigma$, as is evident from the insets in Fig. \ref{fig:system4_singularity}. However, the maximal value of the amplitude shifts away from the center of the stability interval of $\sigma$ as $\gamma$ increases.
\end{enumerate}


\section{Discussion}
\label{sec:discussion}
Theoretical works~\cite{Levasseur_Palacios_2021,Levasseur2022beam1,Levasseur2025beam2} reveal that coupling similar nonlinear oscillators in a feedforward fashion can lead, under certain conditions, to accelerated growth rates of amplitudes of oscillations. This feature can be used to create new mechanisms for weak-signal detection and amplification, as well as for new engineering applications, such as beam steering. While related mathematical models may assume identical cells, or oscillators, in practice, we must account for the fact that even the most sophisticated circuit components have imperfections. These imperfections can be due to manufacturing processes that introduce, for example, slight fluctuations or inhomogeneities in the internal operating parameters of each unit, or to variations in the oscillators' coupling circuitry. On the other hand, it is worth checking whether inhomogeneities may have a positive effect and should therefore be introduced by design. Motivated by these considerations, we have studied the effects of inhomogeneities on the collective response of two-cell feedforward networks. 
\medskip

We have investigated the effects of parameter inhomogeneities in two-cell feedforward networks consisting of $(i)$ pitchfork cells and $(ii)$ Stuart-Landau cells. We used analysis of cubic roots, numerical simulations, and singularity theory for pitchfork cells, and model reduction, numerical simulations, and singularity theory for Stuart-Landau cells. Singularity theory is a powerful tool for identifying loci of bifurcation and hysteresis, where the number of equilibrium points changes. The other analytical and numerical tools we employed were instrumental in finding equilibria and periodic solutions and in assessing their stability.
\medskip

The effects of fluctuations in the excitation parameters in pitchfork cells on the number of equilibrium points and transitions between them to amplified values were studied. Our analysis of the basins of attraction reveals that macroscopic jumps in equilibrium points are possible under certain conditions. In practice, it can be very difficult to control the initial conditions of electronic components, so the analysis has led to a new configuration that guarantees the network response will always transition to the highest amplification rate. This new configuration can prove useful for amplifying DC (Direct Current) signals, for instance. 
\medskip

The combined effects of inhomogeneities in the excitation, frequency, and cubic nonlinearity parameters of Stuart-Landau cells were also investigated. A phase-space analysis revealed complex behavior, e.g., tori bifurcations into quasi-periodic oscillations. 
The phase-locked attractor persists through a surprisingly broad range of inhomogeneities in the natural frequency. While this inhomogeneity reduces the second cell's amplification factor relative to the first, the effect is small for small inhomogeneities. Inhomogeneity in the excitation parameter may amplify the response of the second cell, but when combined with other inhomogeneities, may lead to instability of the phase-locked solution. The effect of small inhomogeneity in the cubic nonlinearity is primarily neutral: it makes the graph of the amplitude of the second cell versus the inhomogeneity in frequency slanted, without changing its height. 
\medskip

The work conducted in this manuscript is not exhaustive. Indeed, there are many other issues worth investigating. Chief among them is the interplay between external signals and the feedforward response. In beam steering, for instance, a feedforward network can be used to transmit (while amplified) signals or to receive incoming radio-frequency signals. In the homogeneous case, 1:1 synchronization with the external signal has been demonstrated, but in the inhomogeneous case, the effects of the interaction remain unknown. Another important issue is the effects of inhomogeneous coupling and disorder on the delay of a signal traveling down the chain of oscillators in the feedforward array. And, of course, inhomogeneities in larger arrays can lead to much more complicated dynamics worth investigating.
These and many other related issues remain open and constitute future work.

\section*{Code availability}
Our codes for generating figures are available in the GitHub repository\\
 \href{https://github.com/mar1akc/TwoCellBifurcatingNetworks}{https://github.com/mar1akc/TwoCellBifurcatingNetworks}.

\section*{ACKNOWLEDGEMENTS}
We gratefully acknowledge support from the Office of Naval Research, Grant No. N000142412547 (M.C. and A.P.). A.P. was also supported by DoD Naval Information Warfare Center (NIWC) Pacific, San Diego, Grant No. N66001-21-D-0041 and by the NIWC internal S\&T program. We thank Mr. Perrin Ruth for valuable discussions of this work.

\appendix
\setcounter{equation}{0}
\renewcommand{\theequation}{\Alph{section}-\arabic{equation}}
    \setcounter{lemma}{0}
    \renewcommand{\thelemma}{\Alph{section}\arabic{lemma}}

\section*{Appendix}

\section{Equilibria of the feedforward network of pitchfork cells }
\label{AppA}
In this appendix, we will elaborate on the dependence of the roots of the cubic polynomials $p_{\pm}$  on parameters $\varepsilon$, $\mu$, and $\lambda$.

Let $\mu > 0$ and $\mu + \varepsilon > 0$. Let $x = \pm\sqrt{\mu}$, i.e., the first cell is in one of its stable equilibria. The equilibria of the second cell $y$ are the roots of the cubic polynomials
\begin{equation}
    \label{eq:p1cubic}
    p_{+}(y): = (\mu + \varepsilon) y - y^3  - \lambda \sqrt{\mu}\quad {\rm and}\quad p_{-}(y): = (\mu + \varepsilon) y - y^3  + \lambda \sqrt{\mu},
\end{equation}
respectively. These polynomials tend to $\pm\infty$ as $y\rightarrow\mp\infty$. They are nonmonotone if $\mu + \varepsilon> 0$ and strictly decreasing otherwise. If $\mu+\varepsilon > 0$, their minimum and maximum are located at $y_{\min} = -\sqrt{\tfrac{\mu+\varepsilon}{3}}$ and $y_{\max} = \sqrt{\tfrac{\mu+\varepsilon}{3}}$, respectively, as sketched in Fig. \ref{fig:sys1_cubic}(a).  Their values at these extrema are 
\begin{equation}
\label{eq:cubic_sys1_extrema}
\begin{aligned}
    p_{+}(y_{\min}) &= -\frac{2}{3\sqrt{3}}(\mu+\varepsilon)^{3/2} -\lambda\sqrt{\mu},\quad  p_{+}(y_{\max}) &= \frac{2}{3\sqrt{3}}(\mu+\varepsilon)^{3/2} -\lambda\sqrt{\mu},\\
     p_{-}(y_{\min}) &= -\frac{2}{3\sqrt{3}}(\mu+\varepsilon)^{3/2} +\lambda\sqrt{\mu},\quad  p_{-}(y_{\max}) &= \frac{2}{3\sqrt{3}}(\mu+\varepsilon)^{3/2}+\lambda\sqrt{\mu}.   
\end{aligned}
\end{equation}
\begin{figure}[htbp]
\centerline{
(a)\includegraphics[width=0.9\textwidth]{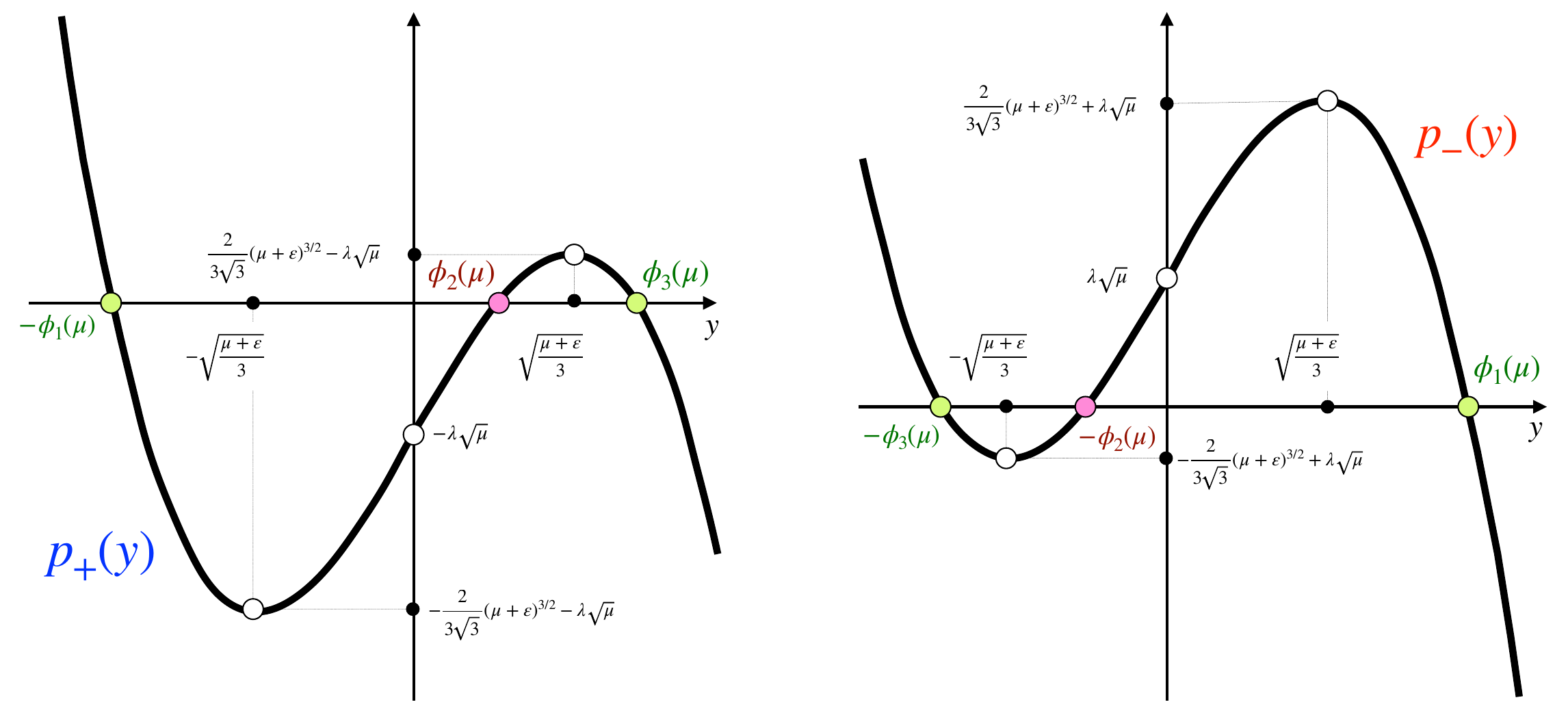}
}
\centerline{
(b) \includegraphics[width=0.9\textwidth]{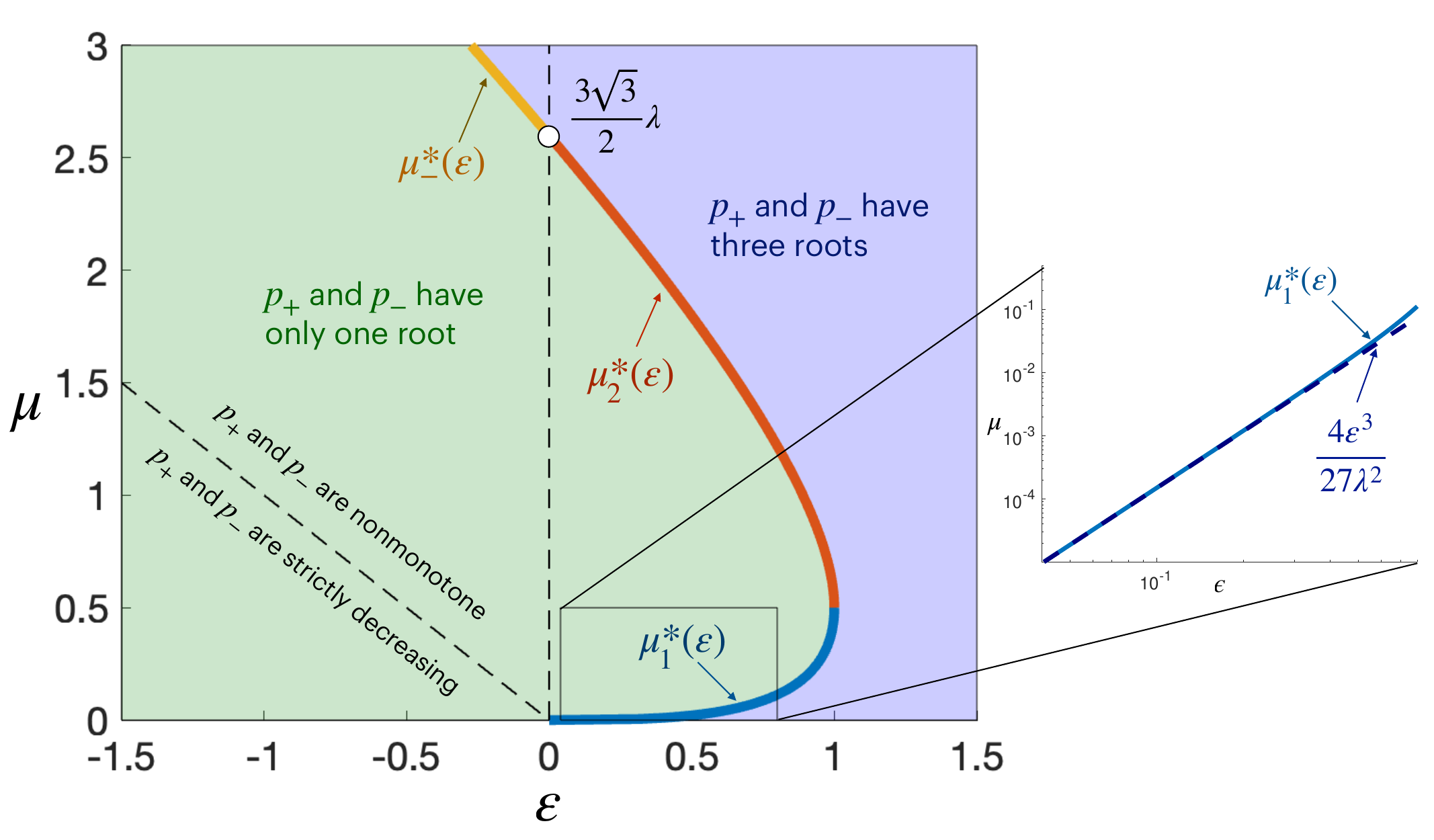}
}
\caption{(a): The cubic polynomials $p_{+}(y)$ and $p_{-}(y)$ defined in \eqref{eq:p1cubic}. (b): A phase diagram of real roots of polynomials $p_{+}$ and $p_{-}$ in the $(\varepsilon,\mu)$-plane at $\lambda = 1$. The curve consisting of blue, red, and yellow pieces, corresponding, respectively, to $\mu_1^*(\varepsilon)$, $\mu_2^*(\varepsilon)$, and $\mu^*(\varepsilon)$, is defined by Eq. \eqref{eq:sys1:critical}.}
\label{fig:sys1_cubic}
\vspace*{12pt}
\end{figure}
Therefore, if $x = \sqrt{\mu}$, $y$ has one or three equilibria, depending on whether  $p_{+}(y_{\max})$ is less or greater than zero. If $x = -\sqrt{\mu}$, $y$ has one or three equilibria depending on whether  $p_{-}(y_{\max})$ is greater or less than zero. In both cases, the steady-state bifurcation in $y$ occurs at the parameter values satisfying
\begin{equation}
    \label{eq:sys1:critical}
   2 (\mu+\varepsilon)^{3/2} = 3\sqrt{3}\lambda\mu^{1/2}.
\end{equation}
The left-hand side of Eq.~\eqref{eq:sys1:critical} is convex in $\mu$ while the right-hand side is concave in $\mu$. Hence, Eq.~\eqref{eq:sys1:critical} has at most two roots. If $\varepsilon < 0$, it has one root on its domain $\mu\ge -\varepsilon > 0$. If $\varepsilon \ge 0$, it transitions from having two roots to having no roots at the value of $\varepsilon$ where Eq.~~\eqref{eq:sys1:critical} holds and the derivatives of its right- and left-hand side are equal. The latter condition is equivalent to  
\begin{equation}
    \label{eq:sys1_critical1}
    (\mu+\varepsilon)^{1/2} = \frac{\lambda\sqrt{3}}{2\mu^{1/2}}.
\end{equation} 
Substituting Eq.~\eqref{eq:sys1_critical1} to Eq.~\eqref{eq:sys1:critical} and taking into account that $\lambda > 0$, we find that if Eq.~\eqref{eq:sys1:critical} has only one root, it must be $\mu = \tfrac{1}{2}\lambda$. Then Eq.~\eqref{eq:sys1_critical1} yields the critical value of $\varepsilon$: 
\begin{equation}
    \label{eq:sys1_critical2}
    \varepsilon = \lambda.
\end{equation} 
Therefore, if $\varepsilon<\lambda$, Eq.~\eqref{eq:sys1:critical} has two roots, if $\varepsilon > \lambda$, it has no roots, and it has a unique root $\mu = \tfrac{1}{2}\lambda$ at $\varepsilon = \lambda$. 

Now we describe the structure and stability of the roots of polynomials $p_{+}(y)$ and $p_{-}(y)$. The stability type is defined by the eigenroots of the Jacobian in Eq. \eqref{eq:system1_J}: $\mu-3x^2$ and $\mu + \varepsilon - 3y^2$. We assume that $\lambda > 0$. A complete phase diagram of real roots of $p_+$ and $p_-$ is shown in Figs. \ref{fig:sys1_phase_bifur} and \ref{fig:sys1_cubic}(b).

\begin{itemize}
\item {\bf Case $\boldsymbol{\varepsilon}\boldsymbol{ =}\boldsymbol{ 0}$.}  This case was studied in~\cite{Levasseur_Palacios_2021}. Eq. \eqref{eq:sys1:critical} has two roots, $0$ and $\tfrac{3\sqrt{3}}{2}\lambda$. The bifurcation diagram for $\varepsilon = 0$ is shown in Fig.~\ref{fig:sys1_phase_bifur}.

    \begin{itemize}
        \item If $\mu < 0$, $(x=0,y=0)$ is a stable node.
        \item If $0< \mu<\tfrac{3\sqrt{3}}{2}\lambda$, the three equilibria corresponding to $x = 0$, $(x=0,y=0)$ and $(x = 0,y=\pm\sqrt{\mu})$, are source and saddles respectively. If $x = \pm\sqrt{\mu}$,
        the polynomials $p_{\pm}$ have unique roots, $\mp\phi_1(\mu)$, respectively. The corresponding equilibria of system~\eqref{eq:system1}, $(x=\sqrt{\mu},y=-\phi_1(\mu))$ and $(x=-\sqrt{\mu},y=\phi_1(\mu))$, are stable nodes. Note that $\phi_1(\mu) > {\tfrac{\mu}{3}}$ as is clear from Fig. \ref{fig:sys1_cubic}.
        \item If $\mu > \tfrac{3\sqrt{3}}{2}\lambda$, system~\eqref{eq:system1} acquires two additional stable equilibria, because the polynomials $p_{+}(y)$ and $p_{-}(y)$ acquire an additional pair of roots, $\phi_2(\mu)$, $\phi_3(\mu)$ and $-\phi_2(\mu)$, $-\phi_3(\mu)$, respectively. Fig. \ref{fig:sys1_cubic} shows that $0<\phi_2(\mu) < \sqrt{\tfrac{\mu}{3}}$,  while $\phi_3(\mu) > \sqrt{\tfrac{\mu}{3}}$. The equilibria $(x = \sqrt{\mu},y=\phi_2(\mu))$  and $(x = -\sqrt{\mu},y=-\phi_2(\mu))$ are saddles, while $(x = \sqrt{\mu},y=\phi_3(\mu))$  and $(x = -\sqrt{\mu},y=-\phi_3(\mu))$ are stable nodes. 
        \end{itemize}

\item {\bf Case $\boldsymbol{ 0}\boldsymbol{<}\boldsymbol{\varepsilon}\boldsymbol{ <}\boldsymbol{\lambda}$.} 
The bifurcation diagram for this case is exemplified by the one for $\varepsilon = 0.8$ in Fig.~\ref{fig:sys1_phase_bifur}. 
Eq. \eqref{eq:sys1:critical} has two roots, $\mu_1^*(\varepsilon)$ and $\mu_2^*(\varepsilon)$.
The qualitative behavior of $p_{+}(y)$ and $p_{-}(y)$ for at various $\mu > 0$ for $0 < \varepsilon <skip \lambda$ is shown in Fig.~\ref{fig:sys1roots}.
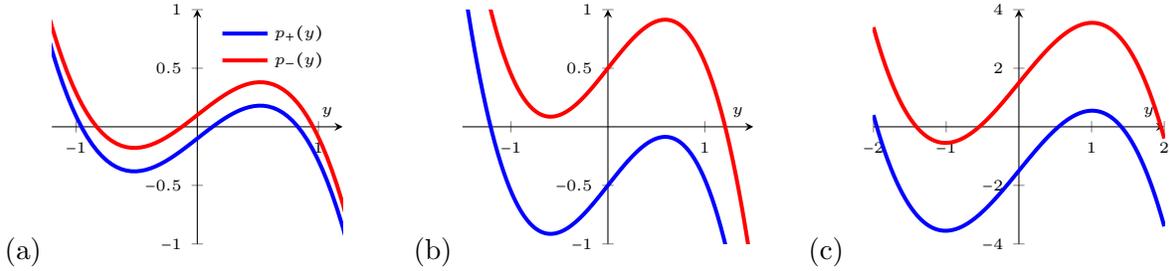
\begin{figure}[h]
\begin{minipage}{.33\textwidth}
  \centering
  (a)
\begin{tikzpicture}
    \begin{axis}[width = \textwidth,
    xmin = -1.2,xmax  = 1.2,
    ymin = -1,ymax = 1,
    axis lines = middle,
    xlabel = {$y$},
    domain = -2:2,
    restrict x to domain=-2:2,
    legend style={draw=none,font = \tiny},
    every axis/.append style={font=\tiny}
    ]
        \addplot[line width=1.5pt,color = blue,samples = 100]{0.81*x - x^3 - 0.1};
        \addlegendentry{$p_{+}(y)$}
        \addplot[line width=1.5pt,color = red,samples = 100]{0.81*x - x^3 + 0.1};
        \addlegendentry{$p_{-}(y)$}
    \end{axis}
\end{tikzpicture}
\end{minipage}%
\begin{minipage}{.33\textwidth}
 \centering
 (b)
\begin{tikzpicture}
    \begin{axis}[width = \textwidth,
    xmin = -1.5,xmax  = 1.5,
    ymin = -1,ymax = 1,
    axis lines = middle,
    xlabel = {$y$},
    domain = -1.5:1.5,
    restrict x to domain=-1.5:1.5,
    legend style={draw=none,font = \tiny},
    every axis/.append style={font=\tiny}
    ]
        \addplot[line width=1.5pt,color = blue,samples = 100]{1.05*x - x^3 - 0.5};
        \addplot[line width=1.5pt,color = red,samples = 100]{1.05*x - x^3 + 0.5};
    \end{axis}
\end{tikzpicture}
\end{minipage}%
\begin{minipage}{.33\textwidth}
 \centering
 (c)
\begin{tikzpicture}
    \begin{axis}[width = \textwidth,
    xmin = -2,xmax  = 2,
    ymin = -4,ymax = 4,
    axis lines = middle,
    xlabel = {$y$},
    domain = -2:2,
    restrict x to domain=-2:2,
    legend style={draw=none,font = \tiny},
    every axis/.append style={font=\tiny}
    ]
        \addplot[line width=1.5pt,color = blue,samples = 100]{3.05*x - x^3 - 1.5};
        \addplot[line width=1.5pt,color = red,samples = 100]{3.05*x - x^3 + 1.5};
    \end{axis}
\end{tikzpicture}
\end{minipage}%
\caption{The qualitative behavior of the polynomials $p_{+}(y)$ and $p_{-}(y)$ as $\mu$ increases from $0$ to $+\infty$ fixed $\lambda > 0$ and $0 < \varepsilon<\lambda$. (a): $p_{+}(y)$ and $p_{-}(y)$ have three roots at $\mu < \mu_{1}^*(\varepsilon)$. (b): Unique roots at $\mu_{1}^*(\varepsilon) < \mu < \mu_{2}^*(\varepsilon)$. (c): Three roots at $\mu < \mu_{2}^*(\varepsilon)$. $\mu_{1}^*(\varepsilon)$ and $\mu_{2}^*(\varepsilon)$ are the two roots of \eqref{eq:sys1:critical} at $0\le \varepsilon\le \lambda$. Here, $\lambda = 1$, $\varepsilon = 0.8$, and $\mu = 0.01$ (a), $\mu = 0.5$ (b), and $\mu = 1.5$ (c).} 
\label{fig:sys1roots}
\end{figure}

\medskip
The root $\mu_1^*(\varepsilon)$ tends to zero as $\varepsilon\rightarrow 0$. This means that the interval of $0<\mu<\mu_1^*(\varepsilon)$ where three roots pf $p_+(y)$ and $p_-(y)$ exist is short. To find how $\mu_1^*(\varepsilon)$ scales with $\varepsilon$ and $\lambda$, we recast Eq.~\eqref{eq:sys1:critical} as a cubic equation with respect to $t: = \mu^{1/3}$:
\begin{equation}
    \label{eq:sys1_cubic4t}
    t^3 - at + b=0,\quad{\rm where}\quad a: = \left(\frac{3\sqrt{3}\lambda}{2}\right)^{2/3},\quad b = \varepsilon.
\end{equation}
Its roots are given by Viete's trigonometric formula~\cite{Zucker_2008} 
\begin{equation}
    \label{eq:sys1_cubic_roots_trig}
    t_k = 2\sqrt{\frac{a}{3}}\cos\left[\frac{1}{3}\arccos\left(-\frac{3b}{2a}\sqrt{\frac{3}{a}}\right)+\frac{2\pi k}{3}\right],\quad k = 0, 1, 2.
\end{equation}
The roots corresponding to $k = 0,1$ tend, respectively, to $\pm\sqrt{a}$, while the root corresponding to $k = 2$ tends to zero as $b\equiv\varepsilon\rightarrow 0$. Taylor-expanding the root $t_2$ at $\varepsilon = 0$ we find
\begin{equation}
    \label{eq:mu_eps0}
    \mu_1^*(\varepsilon)\approx \frac{4\varepsilon^3}{27\lambda^2}.
\end{equation}
Cf. the inset of Fig. \ref{fig:sys1_cubic}(b). 
If this value of $\mu$ is much less than $\mu_1^*(\varepsilon)$, the following approximation for the three roots of $p_{+}(y)$ are obtained by Taylor-expanding Viete's formula \eqref{eq:sys1_cubic_roots_trig} with $a = \mu + \varepsilon$ and $b = \lambda\sqrt{\mu}$ and treating $\mu$ as a small positive parameter:
\begin{equation}
    \label{eq:sys1_three_roots_approx}
    y^*_{0,1} = \pm\sqrt{\varepsilon} - \frac{\lambda\sqrt{\mu}}{2\varepsilon} + O(\mu^{3/2}),\quad y^*_2 = \frac{\lambda\sqrt{\mu}}{\varepsilon}+O(\mu^{3/2}).
\end{equation}

    \begin{itemize}
        \item If $\mu < -\varepsilon$, $(x=0,y=0)$ is the only equilibrium of system \eqref{eq:system1}, and it is a stable node.
        \item If $-\varepsilon < \mu < 0$, system \eqref{eq:system1} has three equilibria: two stable nodes $(x=0,y=\pm\sqrt{\varepsilon + \mu})$ and a saddle $(x=0,y=0)$. 
        \item If $\mu > 0$,  
    the equilibria $(x=0,y=0)$ and $(x = 0,y=\pm\sqrt{\mu+\varepsilon})$ are source and saddles respectively. Let $x = \pm\sqrt{\mu}$.  Then the polynomials $p_{\pm}(y)$ have three roots if $0<\mu<\mu^*_1(\varepsilon)$ or $\mu>\mu^*_2(\varepsilon)$, and one root if $\mu^*_1(\varepsilon)<\mu<\mu^*_2(\varepsilon)$. The eigenvalues $\mu + \varepsilon - 3y^2$ of $J$ in Eq. \eqref{eq:system1_J}, corresponding to $y$, are negative if $y=\pm\phi_1(\mu)$ or $y = \pm \phi_3(\mu)$, as $\phi_1(\mu),\phi_3(\mu) > \sqrt{\tfrac{\mu+\varepsilon}{3}}$, and these eigenvalues are positive if $y = \pm\phi_2(\mu)$ -- see Fig. \ref{fig:sys1_cubic}(a). Hence, the equilibria $(x=\sqrt{\mu},y=-\phi_1(\mu))$, $(x=\sqrt{\mu},y=-\phi_3(\mu))$, $(x=-\sqrt{\mu},y=\phi_1(\mu))$, and $(x=-\sqrt{\mu},y=-\phi_3(\mu))$ are stable nodes, while the equilibria $(x=\sqrt{\mu},y=\phi_2(\mu))$ and $(x=-\sqrt{\mu},y=-\phi_2(\mu))$ are saddles. 
    A similar equilibrium stability structure occurs at $\mu > \mu_2^*(\varepsilon)$. If $\mu_1^*(\varepsilon) < \mu < \mu_2^*(\varepsilon)$, $p_{\pm}(y)$ have unique roots $\mp\phi_1(\mu)$, respectively. The corresponding equilibria of system~\eqref{eq:system1}, $(x=\sqrt{\mu},y=-\phi_1(\mu))$ and $(x=-\sqrt{\mu},y=\phi_1(\mu))$, are stable nodes. 

    \end{itemize}

\item {\bf Case $\boldsymbol{ \varepsilon}\boldsymbol{>}\boldsymbol{\lambda}$.} Eq. \eqref{eq:sys1:critical} has no roots. The polynomials $p_{+}(y)$ and $p_{-}(y)$ have three roots at all $\mu > 0$, and the schematic in Fig. \ref{fig:sys1_cubic} (a) applies. A typical bifurcation diagram for this case is like the one for $\varepsilon = 1.2$ in Fig.~\ref{fig:sys1_phase_bifur}. 
If $\mu < 0$, $x = 0$ is the only equilibrium of the first cell, and the structure and stability of equilibria of system ~\eqref{eq:system1} are the same as in the previous case. If $\mu > 0$, the polynomials $p_{\pm}(y)$ have three roots at all $\mu > 0$. The structure and stability of the equilibria of system~\eqref{eq:system1} are the same as in the previous case with $0<\mu < \mu_1^*(\varepsilon)$ and $\mu > \mu_2^*(\varepsilon)$. 

\item {\bf Case $\boldsymbol{ \varepsilon}\boldsymbol{<}\boldsymbol{0}$.} If $\varepsilon < 0$, $p_{+}(y)$ and $p_{-}(y)$ are monotonously decreasing at $\mu\rightarrow0+$. The steady-state bifurcation from one to three equilibria occurs at the only root $\mu^*(\varepsilon)$ of Eq.~\eqref{eq:sys1:critical} at $\varepsilon < 0$.
A bifurcation diagram for this case is exemplified by the one for $\varepsilon = -0.8$ in Fig.~\ref{fig:sys1_phase_bifur}.  If $0<\mu<\mu^*(\varepsilon)$, the polynomials $p_{\pm}(y)$ have unique roots $\mp\phi_1(\mu)$, respectively. The structure and stability of the equilibria of system~\eqref{eq:system1} are the same as in the case $0 < \varepsilon < \lambda$ with $\mu_1^*(\varepsilon)<\mu < \mu_2^*(\varepsilon)$. If $\mu > \mu^*(\varepsilon)$, the polynomials $p_{\pm}(y)$ have three roots, whose structure and stability are the same as in the case $0 < \varepsilon < \lambda$ with $\mu > \mu_2^*(\varepsilon)$.

\end{itemize}

\section{Proof of Proposition \ref{prop1}}
\label{appB}

\begin{proof}
\begin{enumerate}
\item
   Eq. \eqref{eq:ellipse0} can be rewritten as 
    \begin{equation}
        \label{eq:witch}
        |v|^2 = \frac{1}{\tilde{\mu}^2(1-|v|^2)^2 + \tilde{\sigma}^2}.
    \end{equation}
    The right-hand side of Eq. \eqref{eq:witch} is the function of $|v|^2$ known as the ``Witch of Agnesi". This function is positive, symmetric with respect to $|v|^2 = 1$, and reaches its unique extremum, the maximum $|\sigma|^{-2} \ge 1$, at $|v|^2 = 1$. Since the Witch of Agnesi is strictly decreasing on $|v|^2 \ge 1$, and $f(|v|^2) = |v|^2$ is strictly increasing on $[1,\infty)$, there exists a unique solution on $[1,\infty)$ -- see Fig. \ref{fig:sys2roots}(a). Up to two additional solutions may exist on $(0,1)$ --  see Fig. \ref{fig:sys2roots}(b). Note that if $\tilde{\sigma} = 0$, the graph of the Witch of Agnesi acquires a vertical asymptote at $|v|^2 = 1$, but the same argument applies. 

\begin{figure}[h]
\begin{minipage}{.245\textwidth}
  \centering
  (a)
\begin{tikzpicture}
    \begin{axis}[width = \textwidth,
    xmin = 0,xmax  = 2,
    ymin = 0,ymax = 2,
    axis lines = middle,
    xlabel = {$|v|^2$},
    domain = 0:2,
    restrict x to domain=-0:2,
    legend style={draw=none,font = \tiny},
    every axis/.append style={font=\tiny}
    ]
     \draw [dashed] (1,0) -- (1,2); 
        \addplot[line width=1.5pt,color = blue,samples = 100]{1/(1.8*(x-1)^2 + 0.55)};
        \addplot[line width=1.5pt,color = red,samples = 100]{x};
    \end{axis}
\end{tikzpicture}
\end{minipage}%
\begin{minipage}{.245\textwidth}
 \centering
 (b)
\begin{tikzpicture}
    \begin{axis}[width = \textwidth,
    xmin = 0,xmax  = 2,
    ymin = 0,ymax = 2,
    axis lines = middle,
    xlabel = {$|v|^2$},
    domain = 0:2,
    restrict x to domain=-0:2,
    legend style={draw=none,font = \tiny},
    every axis/.append style={font=\tiny}
    ]
     \draw [dashed] (1,0) -- (1,2); 
        \addplot[line width=1.5pt,color = blue,samples = 100]{1/(8*(x-1)^2 + 0.55)};
        \addplot[line width=1.5pt,color = red,samples = 100]{x};
    \end{axis}
\end{tikzpicture}
\end{minipage}%
\begin{minipage}{.245\textwidth}
 \centering
 (c)
\begin{tikzpicture}
    \begin{axis}[width = \textwidth,
    xmin = 0,xmax  = 1,
    ymin = 0,ymax = 1,
    axis lines = middle,
    xlabel = {$|v|^2$},
    domain = 0:1,
    restrict x to domain=0:1,
    legend style={draw=none,font = \tiny},
    every axis/.append style={font=\tiny}
    ]
     \draw [dashed] (1,0) -- (1,1); 
        \addplot[line width=1.5pt,color = blue,samples = 100]{1/(2*(x-1)^2 + 1.2)};
        \addplot[line width=1.5pt,color = red,samples = 100]{x};
    \end{axis}
\end{tikzpicture}
\end{minipage}%
\begin{minipage}{.245\textwidth}
 \centering
 (d)
\begin{tikzpicture}
    \begin{axis}[width = \textwidth,
    xmin = 0,xmax  = 1,
    ymin = 0,ymax = 1,
    axis lines = middle,
    xlabel = {$|v|^2$},
    domain = 0:1,
    restrict x to domain=0:1,
    legend style={draw=none,font = \tiny},
    every axis/.append style={font=\tiny}
    ]
     \draw [dashed] (1,0) -- (1,1); 
        \addplot[line width=1.5pt,color = blue,samples = 100]{1/(5*(x-1)^2 + 1.05)};
        \addplot[line width=1.5pt,color = red,samples = 100]{x};
    \end{axis}
\end{tikzpicture}
\end{minipage}%
\caption{The structure of roots of Eq. \eqref{eq:witch} depending n $\tilde{\mu}$ and $\tilde{\sigma}$. If $|\tilde{\sigma}| \le 1$, then Eq. \eqref{eq:witch} has a unique root on $[1,\infty)$. If $\tilde{\mu}$ is small enough, this is the only root of Eq. \eqref{eq:witch} as in (a). If $\tilde{\mu}$ is large enough, Eq. \eqref{eq:witch} has two additional roots on $(0,1)$ as in (b). If $|\sigma| > 1$, then Eq. \eqref{eq:witch} has at least one and up to three roots on $(0,1)$, as in (c) and (d), respectively.} 
\label{fig:sys2roots}
\end{figure}
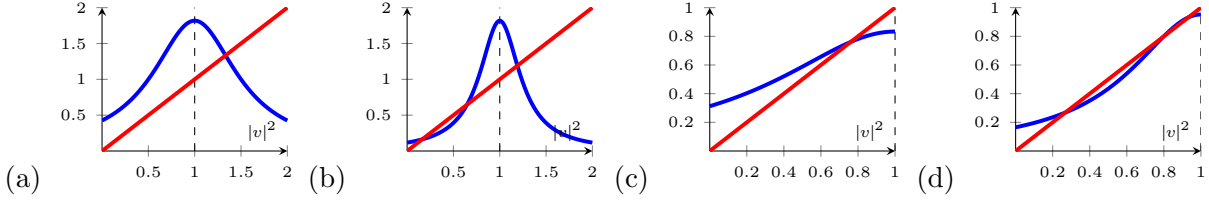

    Once $|v|^2$ is found, $v_R$ and $v_I$ are uniquely determined from Eq. \eqref{eq:vrvi_equil}:
    \begin{equation}
        \label{eq:vrvisol}
        v_R = \frac{\tilde{\mu}(1-|v|^2)}{\tilde{\mu}^2(|v|^2 - 1)^2 + \tilde{\sigma}^2},\quad v_I = -\frac{\tilde{\sigma}}{\tilde{\mu}^2(1-|v|^2)^2 + \tilde{\sigma}^2} 
    \end{equation}
   Note that if $|v|^2 = 1$, then $|\sigma| = 1$,  $v_R = 0$, and $v_I = -{\rm sgn}(\tilde{\sigma})$. 

   To assess the stability of the solution with $|v|^2\ge 1$, we write out the Jacobian of Eq. \eqref{eq:vrvi}:
    \begin{equation}
        J(v_R,v_I) = \left[\begin{array}
        {cc}\tilde{\mu}(1-|v|^2) - 2\tilde{\mu}v_R^2 & -\tilde{\sigma} - 2\tilde{\mu}v_Iv_R \\
        \tilde{\sigma} - 2\tilde{\mu}v_Rv_I & \tilde{\mu}(1-|v|^2) - 2\tilde{\mu}v_I^2\end{array}\right].
    \end{equation}
The conditions for the asymptotic stability are
\begin{align}
    \det J& = \tilde{\mu}^2 (1 - |v|^2)(1 - 3|v|^2) + \tilde{\sigma}^2 > 0,\label{eq:sys2_det}\\
    {\rm tr}\thinspace J & = 2\tilde{\mu}(1 - 2|v|^2) <0. \label{eq:sys2_trace}
\end{align}
We note that $|v|^2 = 1$ implies $|\sigma| = 1$. Then $\det J > 0$ and ${\rm tr}\thinspace J < 0$ if $|v|^2 \ge 1$.  This completes the proof of statement (1). 

\item Statement (2) follows from the expansion of the left-hand side of Eq. \eqref{eq:ellipse0} at $\tilde{\sigma} = 0$:
\begin{equation}
    \label{eq:v2asymp}
    |v|^6 - 2|v|^4 + |v|^2 = \frac{1}{\tilde{\mu}^2}.
\end{equation}
Letting $\tilde{\mu}\rightarrow 0$ and observing that
the term $\tilde{\mu}^2|v|^6 $ dominates the left-hand side of Eq. \eqref{eq:v2asymp}, we obtain that $|v|\approx \tilde{\mu}^{-1/3}$.

\item
Given $|v|^2\in(0,+\infty)$, $\tilde{\sigma}\in\mathbb{R}$, and $\tilde{\mu}> 0$  uniquely define an equilibrium $(v_R,v_I)$ of ODE \eqref{eq:vrvi} by Eq. \eqref{eq:vrvisol}, provided that $(\tilde{\sigma},\tilde{\mu})$ lie on the ellipse \eqref{eq:ellipse0}. This equilibrium is asymptotically stable if and only if inequalities \eqref{eq:sys2_det}--\eqref{eq:sys2_trace} hold. Hence, we are seeking the region in the $(\tilde{\sigma},\tilde{\mu})$-space where the system 
\begin{align}
    |v|^2\tilde{\sigma}^2 +|v|^2(1-|v|^2)^2 \tilde{\mu}^2&  = 1,\label{eq:c1}\\
       \tilde{\mu}^2 (1 - |v|^2)(1 - 3|v|^2) + \tilde{\sigma}^2 &> 0,\label{eq:c2}\\
  2\tilde{\mu}(1 - 2|v|^2)& <0 ,\label{eq:c3}
\end{align}
is compatible. 
By Statement (1), the sought region includes the strip $\tilde{\sigma}^2\le 1$, corresponding to the range $|v|^2\in[1,\infty)$. Hence, it remains to examine the region $|v|^2 \in (0,1)$. Eq. \eqref{eq:c3} implies that $|v|^2 > \tfrac{1}{2}$, which means that the region spanned by the family of ellipses \eqref{eq:c1} with $|v|^2\in(\tfrac{1}{2},1)$ includes the part of the sought region lying beyond the strip $\tilde{\sigma}^2\le 1$.  Below, we will find the part of this region with $\tilde{\sigma} > 1$. Its other part with $\tilde{\sigma} < -1$ is obtained by a mirror reflection with respect to the axis $\tilde{\sigma} = 0$.  

At each fixed $\tilde{\mu}^2>0$, we seek the largest $\tilde{\sigma}^2$ satisfying Eq. \eqref{eq:c1}:
\begin{equation*}
    \tilde{\sigma}^2 = \frac{1}{|v|^2} - (1-|v|^2)^2\tilde{\mu}^2 ~\rightarrow~\max,\quad \frac{1}{2} < |v|^2 < 1.
\end{equation*}
Denoting $|v|^2$ by $x$, taking the derivative $\tfrac{d(\tilde{\sigma}^2)}{dx}$ and setting it to zero, we find that the interior point extrema of $\tilde{\sigma}^2$ exist if and only if the equation
\begin{equation}
    \label{eq:mucond}
    \tilde{\mu}^2 = \frac{1}{2x^2(1-x)}
\end{equation}
has roots, which is the case if $\tilde{\mu}^2 \ge \tfrac{27}{8}$. If $\tilde{\mu}^2 > \tfrac{27}{8}$, Eq. \eqref{eq:mucond} has two positive roots, one greater than $\tfrac{2}{3}$, that corresponds to a local maximum of $\tilde{\sigma}^2$, and one less than $\tfrac{2}{3}$, that corresponds to a local minimum of  $\tilde{\sigma}^2$. Thus, if $\tilde{\mu}<\tfrac{3\sqrt{3}}{2\sqrt{2}}$, the maximum of $\tilde{\sigma}^2$ is achieved at the end of the interval $|v|^2 = \tfrac{1}{2}$. If $\tilde{\mu}<\tfrac{3\sqrt{3}}{2\sqrt{2}}$, the maximum of $\tilde{\sigma}^2$ may be achieved at the larger root of \eqref{eq:mucond} or $x = \tfrac{1}{2}$. The value of $\tilde{\sigma}^2$ at $x = \tfrac{1}{2}$ is larger than that at the larger root of  Eq. \eqref{eq:mucond} for $\tilde{\mu}\in(\tfrac{3\sqrt{3}}{2\sqrt{2}},\tfrac{4\sqrt{2}}{3})$, and the other way around for $\tilde{\mu}>\tfrac{4\sqrt{2}}{3}$ -- see Fig. \ref{fig:system2_phase_zoomin}. Noting that $\tilde{\mu}>\tfrac{4\sqrt{2}}{3}$ corresponds to $x \equiv |v|^2 = \tfrac{3}{4}$, we conclude that the region spanned by ellipses \eqref{eq:c1} is bounded by the union of 
\begin{itemize}
    \item the parametric curve in Eq. \eqref{eq:sys2_det}, where the expression for $\tilde{\mu}$ comes from Eq. \eqref{eq:mucond} and the expression for $\tilde{\sigma}$ is obtained by plugging in Eq. \eqref{eq:mucond} into Eq. \eqref{eq:c1}, and
    \item the ellipse $|v|^2 = \tfrac{1}{2}$.
\end{itemize}
Now it remains to check which part of the region spanned by ellipses \eqref{eq:c1} with $|v|^2\ge\tfrac{1}{2}$ satisfies stability condition \eqref{eq:c2}. Since we are considering the case $|v|^2\in(\tfrac{1}{2},1)$, $1 - |v|^2 > 0$ and $1-3|v|^2 < 0$. Hence, we rewrite Eq. \eqref{eq:c2} as
\begin{equation}
    \label{eq:c2eq}
    \tilde{\sigma}^2 >  \tilde{\mu}^2 (1 - |v|^2)(3|v|^2 - 1).
\end{equation}
Plugging Eq. \eqref{eq:c2eq} into Eq. \eqref{eq:c1} and dividing the result by $\tilde{\mu}^2$, we obtain 
\begin{equation}
    \label{eq:mucond1}
    2|v|^4(1-|v|^2) < \frac{1}{\tilde{\mu}^2},\quad |v|^2\in (\tfrac{1}{2},1),
\end{equation}
cf. Eq. \eqref{eq:mucond}. Hence, we have the following situation:
\begin{itemize}
    \item If $\tilde{\mu}< \frac{3\sqrt{3}}{2\sqrt{2}}$, Eq. \eqref{eq:mucond1} holds for all $|v|^2\in(\tfrac{1}{2},1)$. Hence, the intersection of the region inside the  ellipse \eqref{eq:c1} with $|v|^2 = \tfrac{1}{2}$ and the strips $0<\tilde{\mu}< \frac{3\sqrt{3}}{2\sqrt{2}}$, $\tilde{\sigma}^2 > 1$,  has an asymptotically stable periodic solution.
    \item If $\tilde{\mu} >  \frac{3\sqrt{3}}{2\sqrt{2}}$, inequality \eqref{eq:mucond1} holds for $\tfrac{1}{2}< |v|^2 < x_1$ and $x_2< |v|^2 < 1$, where $x_1 < x_2$ are the two positive roots of the cubic equation $2x^2(1-x) = \tilde{\mu}^{-2}$. The analysis of this cubic equation and Fig. \ref{fig:sys2mucubic} show that $x_1\in(0,\tfrac{2}{3})$ and $x_2 \in(\tfrac{2}{3},1)$.    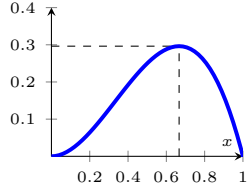
\begin{figure}
    \centering
\begin{tikzpicture}
    \begin{axis}[width = 0.25\textwidth,
    xmin = 0,xmax  = 1,
    ymin = 0,ymax = 0.4,
    axis lines = middle,
    xlabel = {$x$},
    domain = 0:2,
    restrict x to domain=-0:2,
    legend style={draw=none,font = \tiny},
    every axis/.append style={font=\tiny}
    ]
     \draw [dashed] (2/3,0) -- (2/3,8/27);
     \draw [dashed] (0,8/27) -- (2/3,8/27);
        \addplot[line width=1.5pt,color = blue,samples = 100]{2*x^2*(1-x)};
    \end{axis}
\end{tikzpicture}
\caption{The graph of $y = 2x^2(1-x)$, where $x \equiv |v|^2$.} 
\label{fig:sys2mucubic}
\end{figure}

 The interval $\tfrac{1}{2}< |v|^2 < x_1$ defines the region between the ellipse Eq. \eqref{eq:c1} with $|v|^2 = \tfrac{1}{2}$, i.e., $\tfrac{\tilde{\mu}^2}{8} + \tfrac{\tilde{\sigma}^2}{2} = 1$, and the contour
     \begin{equation}
        \label{eq:dash-dot1}
        |v|^2 \in \left(\frac{1}{2},\frac{2}{3}\right),\quad \tilde{\sigma} = \sqrt{\frac{3|v|^2 - 1}{2|v|^4}},\quad \tilde{\mu} = \frac{1}{|v|^2\sqrt{2(1-|v|^2)}}.
    \end{equation}
 Note that $|v|^2 = \tfrac{1}{2}$ defines the point $(\tilde{\sigma} = 1,\tilde{\mu} = 2)$.   
 The interval $x_2< |v|^2 < 1$ defines the region between the line $|v|^2 = 1$ and the contour
     \begin{equation}
        \label{eq:dash-dot2}
        |v|^2 \in \left(\frac{2}{3},1\right),\quad \tilde{\sigma} = \sqrt{\frac{3|v|^2 - 1}{2|v|^4}},\quad \tilde{\mu} = \frac{1}{|v|^2\sqrt{2(1-|v|^2)}}.
    \end{equation}
The union of these two regions is bounded by the curves in Eqs. \eqref{eq:dash-dot1} and \eqref{eq:dash-dot2} and the line $\tilde{\sigma}^2 = 1$. Finally, as we have shown above, the ellipse \eqref{eq:c1} with $|v|^2 = \tfrac{1}{2}$, $\tfrac{\tilde{\mu}^2}{8} + \tfrac{\tilde{\sigma}^2}{2} = 1$, intersects with the curve \eqref{eq:dash-dot2} at $|v|^2 = \tfrac{3}{4}$, $\tilde{\sigma} = \tfrac{\sqrt{10}}{3}$, $\tilde{\mu} = \tfrac{4\sqrt{2}}{3}$.

Thus, an asymptotically stable equilibrium of ODE \eqref{eq:vrvi} exists in the union of these two regions and the region with $\tilde{\mu}\le \tfrac{3\sqrt{3}}{2\sqrt{2}}$, i.e., in the area shaded yellow, pink, and green in Fig. \ref{fig:system2_phase_zoomin}.  This completes the proof of Statement (3). 

\item
To identify the region where three periodic solutions exist, we return to the consideration of the solutions to Eq. \eqref{eq:witch}. We focus on the case $\tilde{\sigma} \ge 0$, as the sought region is symmetric with respect to the $\tilde{\mu}$-axis.  For convenience, we denote $|v|^2$ by $x$. We seek a set of critical values of $(\tilde{\sigma},\tilde{\mu})$, where this equation has exactly two solutions. In this case, the graph of $y = x$ must be tangent to the graph of $y = [\tilde{\mu}^2(1-x)^2 + \tilde{\sigma}^2]^{-1}$ at one of these solutions. Hence, we must solve the system
\begin{equation}
\label{eq:mucrit}
    \begin{aligned}
        x &= \frac{1}{\tilde{\mu}^2(1-x)^2 + \tilde{\sigma}^2},\\
        1& = \frac{2\tilde{\mu}^2(1-x)}{(\tilde{\mu}^2(1-x)^2 + \tilde{\sigma}^2)^2}.
    \end{aligned}
\end{equation}
Substituting the first equation into the second one, we obtain the familiar equation $2\tilde{\mu}^2x^2(1-x) = 1$. It has a solution if $\tilde{\mu}\ge \tfrac{3\sqrt{3}}{2\sqrt{2}}$. 
 Substituting  $\tilde{\mu}^2 = [2x^2(1-x)]^{-1}]$ to the first equation in Eq. \eqref{eq:mucrit}, we find
 \begin{equation}
     \label{eq:sigmacrit}
     \tilde{\sigma}^2 = \frac{3x-1}{2x^2}.
 \end{equation}
Since $\tilde{\sigma}\ge 0$,  we find $x\ge \tfrac{1}{3}$.
The graph of $y =x$ is below the graph of the Witch of Agnesi between two additional solutions to Eq. \eqref{eq:mucrit} in the interval $x\in(0,1)$. This implies that the region in the $(\tilde{\sigma},\tilde{\mu})$-space where three solutions exist lies between the curve \eqref{eq:dash-dot2} and the curve
    \begin{equation}
        \label{eq:dash-dot3}
        |v|^2 \in \left(\frac{1}{3},\frac{2}{3}\right),\quad \tilde{\sigma} = \sqrt{\frac{3|v|^2 - 1}{2|v|^4}},\quad \tilde{\mu} = \frac{1}{|v|^2\sqrt{2(1-|v|^2)}}.
    \end{equation}

Now it remains to check the stability of these three solutions. The proof of Statement (3) reveals that asymptotically stable equilibria lie on ellipses of the form \eqref{eq:c1} with 
\begin{enumerate}
    \item $|v|^2 \ge 1$, or
    \item $\tfrac{1}{2} <|v|^2 \le \tfrac{2}{3}$ and with $\tilde{\mu}$ increasing until they touch the $\det J = 0$ curve in Eq. \eqref{eq:dash-dot3} from above, or
    \item $\tfrac{2}{3} <|v|^2 < 1$ and  with $\tilde{\mu}$ increasing until they touch the curve $\det J = 0$ in Eq. \eqref{eq:dash-dot2} from below.
\end{enumerate}
Ellipse family (a)  does not intersect with families (b) and (c) and foliates the strip $|\tilde{\sigma}| \le 1$. Ellipse family (b) foliates the region under the ellipse  $\tfrac{\tilde{\mu}^2}{8} + \tfrac{\tilde{\sigma}^2}{2} = 1$ with $|v|^2 = \tfrac{1}{2}$ and above curve \eqref{eq:dash-dot3}. Ellipse family (c) foliates the region between the the ellipse with $|v|^2=\tfrac{2}{3}$, the line $\tilde{\sigma} = 1$, and curve \eqref{eq:dash-dot2}. Families (b) and (c) overlap over the region shaded dirty pink in Fig. \ref{fig:system2_phase_zoomin}. The tip of the bistability region corresponds to $|v|^2 = \tfrac{2}{3}$, $\tilde{\sigma} = \tfrac{3}{2\sqrt{2}}$, $\tilde{\mu} = \tfrac{3\sqrt{3}}{2\sqrt{2}}$. Curves \eqref{eq:dash-dot3} and \eqref{eq:dash-dot2} envelope families (b) and (c).
This completes the proof of Statement (4).

\end{itemize}


\end{enumerate}
\end{proof}

\section{Calculations for Section \ref{sec:system4_singularity}}
\label{appC}
Eq. \eqref{eq:G_xx} yields
\begin{equation}
    \label{eq:sys4_x_hyst}
    x = \frac{2(\mu + \varepsilon + \gamma\sigma)}{3(1+\gamma^2)}.
\end{equation}
Plugging Eq. \eqref{eq:sys4_x_hyst} into Eq. \eqref{eq:G_x} gives:
\begin{equation}
\label{eq:sys4_A_hyst}
    \frac{4}{3}\frac{(\mu + \varepsilon + \gamma\sigma)^2}{1+\gamma^2} = (\mu + \varepsilon)^2 + \sigma^2.
    \end{equation}
Plugging Eq. \eqref{eq:sys4_x_hyst} into Eq. \eqref{eq:G_per} and using Eq. \eqref{eq:sys4_A_hyst} results in
\begin{equation}
    \label{eq:sys4_hyst2}
    \frac{8}{27}\frac{(\mu + \varepsilon + \gamma\sigma)^3}{(1+\gamma^2)^2} = \lambda^2\mu.
\end{equation}
\medskip

System \eqref{eq:sys4_x_hyst}, \eqref{eq:sys4_A_hyst}, \eqref{eq:sys4_hyst2} can be solved for $\varepsilon$, $\sigma$, and $|u|^2\equiv x$ in terms of $\lambda$, $\mu$, and $\gamma$ as follows.
We express $ \mu + \varepsilon + \gamma\sigma$ from Eq. \eqref{eq:sys4_hyst2} as
\begin{equation}
\label{eq:sys4_aux1}
   \mu + \varepsilon + \gamma\sigma = \frac{3}{2}\left(\lambda^2 \mu (1+\gamma^2)^2\right)^{1/3} =: C(\lambda).
\end{equation}
Then we denote $\mu +\varepsilon$ by $X$, express $\gamma\sigma$ from Eq. \eqref{eq:sys4_aux1} as $\gamma\sigma = C(\lambda) - X$, multiply Eq. \eqref{eq:sys4_A_hyst} by $\gamma^2$, and obtain a quadratic equation for $X$:
\begin{equation}
    \label{eq:sys4_aux2}
    X^2(1 + \gamma^2) - 2XC(\lambda) + C^2(\lambda)\left(1 - \frac{4}{3}\frac{\gamma^2}{1+\gamma^2}\right) = 0.
\end{equation}
Its solution is
\begin{equation}
    \label{eq:sys4_aux3}
    \mu + \varepsilon\equiv X = C(\lambda)\frac{1\pm\tfrac{1}{\sqrt{3}}\gamma} {1+\gamma^2}\equiv  \frac{3}{2}\frac{1\pm\tfrac{1}{\sqrt{3}}\gamma}{(1+\gamma^2)^{1/3}}\lambda^{2/3}\mu^{1/3}.
\end{equation}
Then, we get
\begin{equation}
    \label{eq:sys4_H}
    \varepsilon = \frac{3}{2}\frac{1\pm\tfrac{1}{\sqrt{3}}\gamma}{(1+\gamma^2)^{1/3}}\lambda^{2/3}\mu^{1/3} - \mu,\quad \sigma = \frac{3}{2}\frac{\gamma\mp\tfrac{1}{\sqrt{3}}}{(1+\gamma^2)^{1/3}}\lambda^{2/3}\mu^{1/3},\quad
    |u|^2 = \lambda^{2/3}\mu^{1/3}(1+\gamma^2)^{-1/3}.
\end{equation}

\section{Additional equations for system \eqref{eq:system4}}
\label{appE}
The ODE system for the real and complex parts of the new variable $u = u_R + iu_I$ defined as $z_2 = ue^{i\omega t}$, where $z_2$ is the state variable of the second state of system \eqref{eq:system4} is
\begin{equation}
    \label{eq:gamma_urui}
    \begin{aligned}
        u_R& = \left(\mu + \varepsilon -|u|^2\right)u_R -\left({\sigma}-\gamma|u|^2\right)u_I - \lambda\sqrt{\mu},\\
        u_I& = \left(\mu + \varepsilon -|u|^2\right)u_I +\left(\sigma -\gamma|u|^2\right)u_R.
    \end{aligned}
\end{equation}
The Jacobian matrix of Eq. \eqref{eq:gamma_urui} is
\begin{equation}
    \label{eq:gamma_uJ}
    J = \left[\begin{array}{ccc}
    (\mu + \varepsilon - |u|^2) -2u_R^2 + 2\gamma u_Ru_I&\quad&
    -2u_Ru_I-(\sigma - \gamma|u|^2) + 2\gamma u_I^2\\
    (\sigma - \gamma|u|^2)-2\gamma u_R^2 - 2u_Ru_I&\quad&
     (\mu + \varepsilon - |u|^2)  - 2u_I^2 - 2\gamma u_Iu_R
    \end{array}\right].
\end{equation}
Its determinant and trace are:
\begin{align}
    \det J &= \left(\mu + \varepsilon - |u|^2\right) \left(\mu + \varepsilon - 3|u|^2\right) +\left(\sigma - \gamma|u|^2\right)\left(\sigma - 3\gamma|u|^2\right),\\
    {\rm tr}\thinspace J& = 2\left(\mu + \varepsilon - 2|u|^2\right).
\end{align}
An equilibrium of ODE \eqref{eq:gamma_urui} is asymptotically stable if $\det J > 0$ and ${\rm tr}\thinspace J < 0$.

\bibliographystyle{alpha}

\end{document}